\documentclass[12pt, a4paper]{amsart}

\usepackage[hmargin=30mm, vmargin=25mm, includefoot, twoside]{geometry}
\usepackage[bookmarksopen=true]{hyperref}

\usepackage{amscd}
\usepackage{amsfonts,amssymb,verbatim}
\usepackage{latexsym}
\usepackage{mathrsfs}
\usepackage{stmaryrd}
\usepackage{xspace}
\usepackage{enumerate, paralist}
\usepackage{graphicx}
\usepackage[all]{xy}
\usepackage{extarrows}

\usepackage[usenames,dvipsnames]{color}

\usepackage{txfonts, pxfonts}

\usepackage{amsthm}
\usepackage{amsmath}

\usepackage{todonotes}
\newtheorem{thm}{Theorem}[section]
 \newtheorem{cor}[thm]{Corollary}
 \newtheorem{lem}[thm]{Lemma}
 \newtheorem{prop}[thm]{Proposition}

 \newtheorem{conj}[thm]{Conjecture}

\numberwithin{equation}{section}

 \theoremstyle{definition}
  \newtheorem{defn}[thm]{Definition}

 \theoremstyle{remark}
 \newtheorem{rem}[thm]{Remark}
  \newtheorem{ex}[thm]{Example}

\newtheorem*{claim*}{Claim}

\def\NN{\mathbb{N}}
\def\RR{\mathbb{R}}
\def\CC{\mathbb{C}}
\def\AA{\mathbb{A}}

\def\U{\mathcal{U}}
\def\H{\mathcal{H}}
\def\L{\mathcal{L}}
\def\Nd{\mathcal{N}}
\def\S{\mathcal{S}}
\def\Ac{\mathcal{A}}
\def\C{\mathcal{C}}
\def\D{\mathcal{D}}
\def\O{\mathcal{O}}

\def\HH{\mathscr{H}}
\def\SS{\mathscr{S}}

\def\B{\mathfrak{B}}
\def\K{\mathfrak{K}}
\def\A{\mathfrak{A}}

\def\supp{\mathrm{supp}}
\def\ppg{\mathrm{prop}}
\def\Id{\mathrm{Id}}
\def\Cliff{\mathrm{Cliff}}
\def\Ind{\mathrm{Ind}}
\def\Ad{\mathrm{Ad}}
\def\diam{\mathrm{diam}}
\def\ev{\mathrm{ev}}

\begin{document}

\title[The equivariant coarse Baum-Connes conjecture and a-T-menability]{The equivariant coarse Baum-Connes conjecture for actions by a-T-menable groups}

\author{Benyin Fu and Jiawen Zhang}

\address[B. Fu]{College of Statistics and Mathematics, Shanghai Lixin University of Accounting and Finance, Shanghai, 201209, China.}
\email{fuby@lixin.edu.cn}

\address[J. Zhang]{School of Mathematical Sciences, Fudan University, 220 Handan Road, Shanghai, 200433, China.}
\email{jiawenzhang@fudan.edu.cn}

\date{}

\thanks{Both of the authors are supported by NSFC11871342.}

\begin{abstract}
The equivariant coarse Baum-Connes conjecture was firstly introduced by Roe \cite{Roe96} as a unified way to approach both the Baum-Connes conjecture and its coarse counterpart. In this paper, we prove that if an a-T-menable group $\Gamma$ acts properly and isometrically on a bounded geometry metric space $X$ with controlled distortion such that the quotient space $X/\Gamma$ is coarsely embeddable, then the equivariant coarse Baum-Connes conjecture holds for this action. This answers a question posed in \cite{DFW21} affirmatively.
\end{abstract}

\date{\today}
\maketitle

\parskip 4pt


\section{Introduction}

Over the last century, one of the most significant achievements in mathematics is the Atiyah-Singer index theorem \cite{AS68}. It provides a topological formula for Fredholm indices of elliptic pseudodifferential operators on closed manifolds, and has profound applications in geometry and topology (see, \emph{e.g.}, \cite{LM89}).

To obtain finer topological and geometric information of the underlying manifold, people consider lifted operators on the universal cover. These operators are by definition equivariant under deck transformations, and hence can be cooked up to provide an element (called the \emph{higher index} \cite{CM90}) in the K-theory of the reduced group $C^*$-algebra of the fundamental group \cite{GL83, Ros83}. Higher indices reveals more information of the underlying manifolds, and hence it is important to compute $K$-theories of reduced group $C^*$-algebras. The famous Baum-Connes conjecture (\cite{BC00}, see also \cite{BCH94}) provides a practical and systematic approach, which has fruitful applications in analysis, geometry and topology (see, \emph{e.g.}, \cite{BC88, Val02}). It has been verified for a large classes of groups, including a-T-menable groups (\emph{i.e.}, groups with the Haagerup property) \cite{HK01} and hyperbolic groups \cite{Laf12}.

On the other hand, Roe introduced a coarse geometric approach in his pioneering work on higher index theory for elliptic differential operators on open manifolds \cite{Roe88}, and formulated a coarse version to the Baum-Connes conjecture \cite{roe1993coarse} (see also \cite{HR95}). Using Higson's descent principle \cite{Hig00}, the coarse Baum-Connes conjecture also found significant applications in geometry and topology (see, \emph{e.g.}, \cite{Roe96, willett2020higher}), for example the Novikov conjecture concerning homotopy invariance of higher signatures and the Gromov-Lawson conjecture concerning positive scalar curvature on manifolds. Over the last three decades, a lot of results on the coarse Baum-Connes conjecture have been achieved \cite{CWY13, GWY08, KY06, KY12, WY12} after Yu's groundbreaking work for coarsely embeddable metric spaces \cite{Yu00}.

As indicated in \cite[Section 5]{Roe96} (see also \cite[Section 7]{willett2020higher}), the Baum-Connes conjecture and its coarse counterpart can be formulated in a unified way using the language of group actions. To be more precise, let us assume that $\Gamma$ is a countable discrete group acting on a proper metric space $X$ properly (\emph{not necessarily} cocompactly) by isometries. There is an \emph{equivariant index map} (see Section \ref{ssec: equiv CBC} for precise definitions)
\[
\Ind^\Gamma: \lim_{r\rightarrow\infty} K_*^\Gamma(P_r(X))\longrightarrow  K_*(C^*(X)^\Gamma)
\]
for $\ast=0,1$ where $K_*^\Gamma(P_r(X))$ is the $\Gamma$-equivariant $K$-homology group of the Rips complex $P_r(X)$ of $X$ with scalar $r$, and $K_*(C^*(X)^\Gamma)$ is the $K$-theory of the $\Gamma$-equivariant Roe algebra of $X$. The equivariant coarse Baum-Connes conjecture asserts that the equivariant index map $\Ind^\Gamma$ is an isomorphism.

Noted by Roe \cite{Roe96}, the equivariant Roe algebra $C^*(X)^\Gamma$ is Morita equivalent to the reduced group $C^*$-algebra $C^*_r(\Gamma)$ when the action is cocompact, and hence the equivariant index map $\Ind^\Gamma$ being an isomorphism is just a reformulation of the classic Baum-Connes conjecture. A typical example is the fundamental group $\Gamma$ of a closed manifold $M$ acting on its universal cover $\widetilde{M}$ via deck transformations. As mentioned above, a lifted differential operator (from $M$ to $\widetilde{M}$) is $\Gamma$-equivariant and hence its higher index sits naturally in the $K$-theory of the equivariant Roe algebra $C^*(\widetilde{M})^\Gamma$, which coincides with the $K$-theory of $C^*_r(\Gamma)$.
On the other hand when $\Gamma$ is trivial, $\Ind^\Gamma$ being an isomorphism is nothing but the coarse Baum-Connes conjecture. Therefore, the equivariant coarse Baum-Connes conjecture can be regarded as an ``interpolation'' between the Baum-Connes conjecture for groups and the coarse Baum-Connes conjecture for metric spaces.

The equivariant coarse Baum-Connes conjecture was studied by couples of mathematicians, and recently has attracted more and more attention. In \cite{shan2008equivariant}, Shan proved the equivariant index map is injective for torsion free groups acting on simply connected complete Riemannian manifolds with non-positive scalar curvature. Later in \cite{FW16}, the first author together with Wang showed that the equivariant coarse Baum-Connes conjecture holds for group actions on metric spaces with bounded geometry which admit equivariant coarse embeddings into Hilbert space, serving as an equivariant analogue of Yu's celebrated work \cite{Yu00}.

However, a recent ingenious example due to Arzhantseva and Tessera \cite{AT19} illuminates that a metric space with a group action might not admit a coarse embedding into Hilbert space even if both the group and the quotient space are coarsely embeddable. Inspired by their example, the first author together with Wang and Yu \cite{FWY20} managed to show that the equivariant index map remains injective for group actions with bounded distortion when both the group and the quotient space admit coarse embeddings into Hilbert space.

Along this way, very recently the first author together with Deng and Wang \cite{DFW21} proved that the equivariant coarse Baum-Connes conjecture holds for amenable groups acting on metric spaces such that all orbits are equivariantly uniformly coarsely equivalent and quotient spaces are coarsely embeddable. They also asked a question whether the result remains true when the involved group is a-T-menable, inspired by Higson and Kasparov's significant result that the Baum-Connes conjecture holds for a-T-menable groups \cite{HK01}. However, the techniques in \cite{DFW21} are no longer applicable for a-T-menable groups due to an obstruction in coarse $K$-amenability (see \cite[Theorem 1.2]{DFW21}).

In this paper, we use a different approach to bypass the issue of coarse $K$-amenability and answer the question posed in \cite{DFW21} affirmatively. More precisely, we show that the equivariant coarse Baum-Connes conjecture holds for a-T-menable groups acting on metric spaces with controlled distortion such that quotient spaces are coarsely embeddable. To state our main result, let us introduce some notions first (see Section \ref{sec:pre} for more details).

For a metric space $(X,d)$ with a $\Gamma$-action, a subset $\D \subseteq X$ is called a \emph{fundamental domain} if $X$ can be decomposed into the disjoint union of $\Gamma$-orbits for points in $\D$. An action is said to have \emph{controlled distortion} if there exists a fundamental domain $\D \subseteq X$ such that the family of orbit maps $\{\O_x: \Gamma \to \Gamma x, \gamma \mapsto \gamma x\}_{x\in \D}$ are uniformly coarsely equivalent.

The following is the main result of this paper:

\begin{thm}\label{thm:main result equiv. CBC}
Let $\Gamma$ be a countable discrete group acting properly and isometrically on a discrete metric space $X$ with bounded geometry. If the action has controlled distortion, the quotient space $X /\Gamma$ admits a coarse embedding into Hilbert space and $\Gamma$ is a-T-menable, then the equivariant index map
\[
\Ind^\Gamma: \lim_{r\rightarrow\infty} K_*^\Gamma(P_r(X))\longrightarrow  K_*(C^*(X)^\Gamma)
\]
is an isomorphism for $\ast=0,1$.
\end{thm}

Note that when $\Gamma$ is trivial, Theorem \ref{thm:main result equiv. CBC} recovers Yu's celebrated result \cite{Yu00} on the coarse Baum-Connes conjecture for coarsely embeddable metric spaces. On the other hand, note that when the action is cocompact then it automatically has controlled distortion, and hence Theorem \ref{thm:main result equiv. CBC} recovers Higson and Kasparov's famous result \cite{HK01} on the Baum-Connes conjecture for a-T-menable groups. Therefore, Theorem \ref{thm:main result equiv. CBC} can be regarded as a combination of \cite{HK01} and \cite{Yu00} in virtue of the language of group actions.

Readers might already notice that the hypothesis of Theorem \ref{thm:main result equiv. CBC} is slightly different from that of \cite[Theorem 1.1]{DFW21} except for the requirement on the group. More precisely, in Theorem \ref{thm:main result equiv. CBC} we assume that the action has controlled distortion, while \cite[Theorem 1.1]{DFW21} requires that all orbits are equivariantly uniformly coarsely equivalent. As revealed in Section \ref{ssec:comparing}, having controlled distortion can be deduced from equivariantly uniformly coarse equivalence of orbits, which means that our hypothesis is (at least formally) weaker than the one used in \cite[Theorem 1.1]{DFW21}\footnote{In fact, we think that controlled distortion with other assumptions are also sufficient to prove \cite[Theorem 1.1]{DFW21}.}.  In Section \ref{ssec:comparing}, we also provide some analysis on the relations between our hypothesis of Theorem \ref{thm:main result equiv. CBC} and those used in \cite{FW16, FWY20}.

The proof of Theorem \ref{thm:main result equiv. CBC} follows the machinery from \cite[Chapter 12]{willett2020higher}, which was originally invented by Yu in \cite{Yu00}. Roughly speaking, we use an equivariant version of the coarse Mayer-Vietoris argument to chop the space and decompose the associated algebras preserving group actions. This leads to the reduction of the proof for Theorem \ref{thm:main result equiv. CBC} to the case of sequences of cocompact actions with block-diagonal operators thereon (Corollary \ref{cor:final reduction}). Then we construct equivariant twisted Roe and localisation algebras (Definition \ref{defn:twisted Roe} and \ref{defn:twisted localisation}), and index maps in terms of the Bott-Dirac operators (Proposition \ref{prop:index maps}). Finally, we prove an equivariant version of the local isomorphism (Theorem \ref{thm:iso. of twisted algebras in $K$-theory}) thanks to (a family version of) the Baum-Connes conjecture with coefficients verified for a-T-menable groups \cite{HK01}, and conclude the proof via a canonical diagram chasing argument.

We would like to highlight the difference between our approach and the one used in \cite{DFW21}. Although both of them consult Yu's machinery, the constructions of equivariant twisted algebras and index maps are different. More precisely in the approach of \cite{DFW21}, index maps (usually called the \emph{Bott} and \emph{Dirac} maps) are defined on dense subalgebras in the form of asymptotic morphisms, and therefore one needs a coarse $K$-amenability result to extend these maps. While in the current paper, we use a coarse Mayer-Vietoris argument to chop the space and decompose the associated algebras at first, which allows us to construct index maps directly on the $K$-theories of the whole algebras (see Proposition \ref{prop:index maps defn}). An ingenious ``\emph{space for speed}'' argument from \cite{willett2020higher} is crucial to bypass the issue of coarse $K$-amenability and accomplish the job.

The paper is organised as follows. In Section \ref{sec:pre}, we collect necessary preliminaries and provide comparisons for different geometric hypotheses used in \cite{DFW21, FW16, FWY20}. Section \ref{sec:reduction} is devoted to the reduction of the proof for Theorem \ref{thm:main result equiv. CBC} to the case of sequences of cocompact actions. We introduce equivariant twisted algebras in Section \ref{sec:twisted algebras}, and construct index maps in Section \ref{sec:index map}. In Section \ref{sec:local isom}, we prove that the equivariant twisted Roe and localisation algebras have the same $K$-theory using a family version of the Baum-Connes conjecture with coefficients for a-T-menable groups essentially from \cite{HK01}, and finally conclude the proof in Section \ref{sec:pf of main thm}.

We also provide an appendix on a precise statement of the family version of the Baum-Connes conjecture with coefficients for a-T-menable groups which we need to conclude Theorem \ref{thm:iso. of twisted algebras in $K$-theory}. This result is well-known to experts, while we cannot find an explicit statement or proof in literature. For convenience to readers and also for completeness, we provide a detailed proof based on ideas from \cite{HK01} and \cite{Yu00}.


\section{Preliminaries}\label{sec:pre}

In this section, we recall some notions and definitions.

\subsection{Notions from coarse geometry}\label{ssec:Notions from coarse geometry}

Here we collect several basic notions.

\begin{defn}\label{defn: basic notions for metric spaces}
Let $(X,d)$ be a metric space and $R>0$.
\begin{enumerate}
 \item For $x\in X$, denote the closed \emph{$R$-ball} $B(x,R):=\{y\in X: d(x,y) \leq R\}$. For $A \subseteq X$, denote its \emph{$R$-neighbourhood} by $\Nd_R(A):=\{y\in X: d(y,A) \leq R\}$.
 \item A subset $A \subseteq X$ is called \emph{bounded} if its diameter $\sup\{d(x,y): x,y\in A\}$ is finite; $A$ is called a \emph{net} in $X$ if there exists a $c>0$ such that $\Nd_c(A) = X$.
 \item $(X,d)$ is called \emph{proper} if any bounded closed subset in $X$ is compact.
 \item If $(X,d)$ is discrete, we say that $(X,d)$ has \emph{bounded geometry} if $\sup_{x\in X} \sharp B(x,R)$ is finite for any $R>0$. In the general case, we say that $(X,d)$ has \emph{bounded geometry} if it contains a discrete net with bounded geometry.
\end{enumerate}
\end{defn}

\begin{defn}\label{defn: basic notions for maps}
A map $f: (X,d_X) \to (Y,d_Y)$ between metric spaces is called a \emph{coarse embedding} if there exist two proper non-decreasing functions $\rho_+, \rho_-: [0,\infty) \to [0,\infty)$ such that
\[
\rho_-(d_X(x,y)) \leq d_Y(f(x),f(y)) \leq \rho_+(d_X(x,y))
\]
holds for any $x,y\in X$. The map $f$ is called a \emph{coarse equivalence} if it is a coarse embedding and the image of $f$ is a net in $Y$. We say that $X$ and $Y$ is \emph{coarsely equivalent} if there exists a coarse equivalence $f: X \to Y$.

A family of maps $\{f_i: X_i \to Y_i\}_{i\in I}$ between metric spaces are called \emph{uniformly coarsely equivalent} if $f_i$ is a coarse equivalence with the same control functions $\rho_+$ and $\rho_-$, and there exists a $c>0$ such that the $c$-neighbourhood of the image of $f_i$ coincides with $Y_i$ for each $i\in I$.
\end{defn}

For a proper function $\rho: [0,\infty) \to [0,\infty)$ (which is not necessarily monotonely increasing), we denote
\[
\rho^{-1}(R):=\sup \{S \in [0,\infty): \rho(S) \leq R\}, \quad \mbox{for~} R \in [0,\infty).
\]

Now we move to the case of group actions.

\begin{defn}\label{defn: basic notions for group actions}
Let $\Gamma$ be a countable discrete group acting on a proper metric space $(X,d)$. The action is called \emph{proper} if for any compact set $K\subseteq X$, the set $\{\gamma\in \Gamma: \gamma K\cap K\neq \emptyset\}$ is finite. The action is called \emph{isometric} if $d(x,y)=d(\gamma x,\gamma y)$ for all $\gamma\in\Gamma$ and $x,y\in X$.
\end{defn}

To simplify the notation, we say that a proper metric space $(X,d)$ is a \emph{$\Gamma$-space} if the $\Gamma$-action is proper and isometric. For a $\Gamma$-space $(X,d)$ and $x\in X$, the \emph{orbit of $x$} is defined to be $\Gamma x$ and the \emph{orbit map at $x$} is defined by:
\[
\O_x: \Gamma \longrightarrow \Gamma x, \quad \gamma \mapsto \gamma x \quad \mbox{for} \quad \gamma \in \Gamma.
\]

\begin{defn}\label{defn: basic notions for maps btw group spaces}
Let $\Gamma$ be a countable discrete group. A map $f: X \to Y$ between two $\Gamma$-spaces is called \emph{equivariant} if for any $x\in X$ and $\gamma\in \Gamma$, we have $f(\gamma x) = \gamma f(x)$. The map $f$ is called an \emph{equivariantly coarse equivalence} if $f$ is equivariant and coarsely equivalent. 
A family of maps $\{f_i: X_i \to Y_i\}_{i\in I}$ between $\Gamma$-spaces is called an \emph{equivariantly uniformly coarse equivalence} if each $f_i$ is $\Gamma$-equivariant and the family $\{f_i\}_{i\in I}$ is a uniformly coarse equivalence.

A family of $\Gamma$-spaces $X_i$ for $i\in I$ are called \emph{equivariantly uniformly coarsely equivalent} if for any $i,j\in I$ there exists an equivariant map $f_{i,j}: X_i \to X_j$ such that the family $\{f_{i,j}:X_i \to X_j\}_{i,j\in I}$ is a uniformly coarse equivalence.
\end{defn}

Recall that for a $\Gamma$-space $(X,d)$, a subset $\D \subseteq X$ is called a \emph{fundamental domain} if $X$ can be decomposed into the disjoint union of $\Gamma$-orbits of points in $\D$. An action is said to have \emph{controlled distortion} if there exists a fundamental domain $\D \subseteq X$ such that the family of orbit maps $\{\O_x:x\in \D\}$ are uniformly coarsely equivalent. In this case, we also say that the action has controlled distortion with respect to $\D$.

Recall from \cite{FWY20} that for a $\Gamma$-space $(X,d)$, the action is said to have \emph{bounded distortion} if there exists a fundamental domain $\D$ such that for any $\gamma \in \Gamma$, we have
\[
\sup_{y\in \D} d(\gamma y,y) < +\infty.
\]
In this case, we also say that the action has bounded distortion with respect to $\D$.
It follows directly from definitions that having controlled distortion implies bounded distortion.

\subsection{Comparing different geometric hypotheses}\label{ssec:comparing}

In this subsection, we would like to compare the geometric hypothesis of Theorem \ref{thm:main result equiv. CBC} with those used in the main results of \cite{DFW21, FW16, FWY20}. For convenience to readers, we record these hypotheses chronologically as follows. Assume that $(X,d)$ is a $\Gamma$-space.

\begin{enumerate}
 \item \emph{hypothesis of \cite[Theorem 1.2]{FW16}}: $X$ admits a $\Gamma$-equivariant coarse embedding into Hilbert space.
 \item \emph{hypothesis of \cite[Theorem 1.2]{FWY20}}: both $\Gamma$ and $X/\Gamma$ are coarsely embeddable, and the action has bounded distortion.
 \item \emph{hypothesis of \cite[Theorem 1.1]{DFW21}}: $\Gamma$ is amenable and $X/\Gamma$ is coarsely embeddable, and the orbit spaces are equivariantly uniformly equivalent.
 \item \emph{hypothesis of Theorem \ref{thm:main result equiv. CBC} (the current paper)}: $\Gamma$ is a-T-menable and $X/\Gamma$ is coarsely embeddable, and the action has controlled distortion.
\end{enumerate}

Note that the conclusion of \cite[Theorem 1.2]{FWY20} says that the equivariant index map is injective, while the rest conclude isomorphisms.

The following lemma shows that our hypothesis of Theorem \ref{thm:main result equiv. CBC} is (at least formally) weaker than the one of \cite[Theorem 1.1]{DFW21}.



\begin{lem}\label{lem:assump for main thm}
For a $\Gamma$-space $(X,d)$, if all $\Gamma$-orbits are equivariantly uniformly coarsely equivalent, then the action has controlled distortion.
\end{lem}

\begin{proof}
Fix a point $x_0 \in X$ and a fundamental domain $\D' \subseteq X$ with $x_0 \in \D'$. Since the action is proper, it follows from the Milnor-\v{S}varc lemma that the orbit map $\O_{x_0}: \Gamma \longrightarrow \Gamma x_0$ is a coarsely equivalent. Now for any $y\in \D'$, by assumption there exists an equivariant map $\psi_{y}: \Gamma x_0 \to \Gamma y$ such that $\{\psi_{y}: y\in \D'\}$ is a uniformly coarse equivalence. Setting $\hat{y}:=\psi_y(x_0)$, the equivariance of $\psi_y$ implies that $\psi_y$ has the following form:
\[
\psi_y: \Gamma x_0 \longrightarrow \Gamma \hat{y}, \quad  \gamma x_0 \mapsto \gamma \hat{y} \quad \mbox{for}~ \gamma \in \Gamma.
\]
Taking $\D:=\{\hat{y}: y\in \D'\}$, it is clear that $\D$ is also a fundamental domain. Moreover, for $\hat{y} \in \D$ the orbit map $\O_{\hat{y}}$ coincides with $\psi_y \circ \O_{x_0}$. Hence we conclude the proof.
\end{proof}

Now we turn to the relation with \cite{FW16} and \cite{FWY20}. The general situation is still unclear while we try to offer some analysis in special cases.

To compare with \cite{FWY20}, note that having controlled distortion implies bounded distortion. We show in the following that for certain special case, these two notions are ``almost'' the same.

\begin{ex}\label{ex:semi-direct product}
Let $N,Q$ be countable discrete groups with an action $\alpha: Q \to \mathrm{Aut}(N)$. Equip the associated semi-direct product $N \rtimes Q$ with a proper length function $\ell$, which derives a left-invariant proper metric $d$. Consider the action of $N$ on $N \rtimes Q$ by left multiplication, which is clearly isometric and proper.

Assume that the action has bounded distortion with respect to the fundamental domain $Q$. Then obviously there exists a proper function $\rho_+: [0,\infty) \to [0,\infty)$ such that $d(hq,q) \leq \rho_+(\ell(h))$ for any $h\in N$ and $q\in Q$. Hence for any $h\in N$ and $q\in Q$, we also have:
\[
\ell(h) = d(q^{-1} h q \cdot q^{-1}, q^{-1}) \leq \rho_+(\ell(q^{-1} h q)) = \rho_+(d(hq,q)),
\]
which implies that the action has controlled distortion. Finally we remark that it is unclear to the authors that whether the same result holds if the action has bounded distortion with respect to an arbitrary fundamental domain.
\end{ex}

To compare with \cite{FW16}, recall that a recent ingenious example due to Arzhantseva and Tessera \cite{AT19} illuminates that a $\Gamma$-space might not admit a coarse embedding into Hilbert space even if both the group and the quotient space are coarsely embeddable. However as we show below, the situation often gets better under the hypothesis of controlled distortion.

\begin{lem}\label{lem:ctrl distortion and ce}
Let $(X,d)$ be a $\Gamma$-space with the quotient map $\pi: X \to X/\Gamma$. Equip $\Gamma$ with a proper left-invariant metric $d_\Gamma$ and $X/\Gamma$ with the quotient metric $d_q$. Assume that there exists a fundamental domain $\D \subseteq X$ such that $\pi|_{\D}: \D \to \pi(\D)$ is a coarse equivalence and the action has controlled distortion with respect to $\D$. Then $X$ is coarsely equivalent to the product metric $\Gamma \times (X/\Gamma)$ equipped with the product metric $d_\Gamma \times d_q$.

Consequently, if additionally both $\Gamma$ and $X/\Gamma$ are coarsely embeddable, then so is $X$.
\end{lem}

\begin{proof}
Assume that $\D=\{x_\lambda: \lambda \in \Lambda\}$, and $\pi|_{\D}$ is a coarse equivalence with controlled functions $\rho_+$ and $\rho_-$. Also assume that the orbits maps $\{\O_{x_\lambda}: \lambda \in \Lambda\}$ are uniformly coarsely equivalent with the same controlled functions. For any two points $\gamma_1 x_{\lambda_1}$ and $\gamma_2 x_{\lambda_2}$ in $X$, note that
\begin{align*}
d(\gamma_1 x_{\lambda_1}, \gamma_2 x_{\lambda_2}) &= d(x_{\lambda_1}, \gamma_1^{-1}\gamma_2 x_{\lambda_2}) \leq d(x_{\lambda_1}, x_{\lambda_2}) + d(\gamma_1^{-1}\gamma_2 x_{\lambda_2}, x_{\lambda_2}) \\
&\leq \rho_-^{-1}(d_q(\pi(x_{\lambda_1}), \pi(x_{\lambda_2}))) + \rho_+(d_\Gamma(\gamma_1, \gamma_2)).
\end{align*}

On the other hand, given $R>0$ and assume that $d(\gamma_1 x_{\lambda_1}, \gamma_2 x_{\lambda_2}) \leq R$. Then
\[
d_q(\pi(x_{\lambda_1}), \pi(x_{\lambda_2})) = d_q(\pi(\gamma_1 x_{\lambda_1}), \pi(\gamma_2 x_{\lambda_2})) \leq R,
\]
which implies that $d(x_{\lambda_1}, x_{\lambda_2}) \leq \rho_+(R)$. Moreover, we have
\[
d(x_{\lambda_2}, \gamma_1^{-1} \gamma_2 x_{\lambda_2}) \leq d(x_{\lambda_1}, x_{\lambda_2}) + d(x_{\lambda_1}, \gamma_1^{-1} \gamma_2 x_{\lambda_2}) \leq \rho_+(R) + R,
\]
which implies that $d_{\Gamma}(\gamma_1, \gamma_2) \leq \rho_-^{-1}(\rho_+(R) + R)$.

Combining the above two paragraphs, we conclude the proof.
\end{proof}

\begin{rem}\label{rem:semi-direct}
It is unclear to the authors whether Lemma \ref{lem:ctrl distortion and ce} holds or not without the assumption that $\pi|_{\D}$ is a coarse equivalence. On the other hand, note that this condition holds for many examples, including the semi-direct product case studied in Example \ref{ex:semi-direct product}.

More precisely, it is obvious that the quotient map $N\rtimes Q \to Q$ restricted to the fundamental domain $Q$ is a coarse equivalence. As a consequence to Lemma \ref{lem:ctrl distortion and ce}, we obtain that $N\rtimes Q$ is coarsely embeddable if both $N$ and $Q$ are coarsely embeddable and the action of $N$ on $N \rtimes Q$ by left multiplication has bounded distortion with respect to $Q$. Finally, note that the example constructed in \cite{AT19} has the form of semi-direct products. It follows by straightforward calculations that their example does not have bounded distortion, alternatively by combining the fact that it is not coarsely embeddable together with the analysis above.
\end{rem}

\subsection{(Equivariant) Roe algebras}

Now we introduce the notion of Roe algebras and their equivariant counterparts.

For a proper metric space $(Z,d)$, recall that a \emph{$Z$-module} is a non-degenerate $\ast$-representation $\phi: C_0(Z) \to \B(\H_Z)$ where $\H_Z$ is some infinite-dimensional separable Hilbert space. We also say that $\H_Z$ is a $Z$-module if the representation is clear from the context, and simply write $f$ as a bounded linear operator on $\H_Z$ instead of $\phi(f)$ for $f\in C_0(Z)$. A $Z$-module is called \emph{ample} if no non-zero element of $C_0(Z)$ acts as a compact operator on $\H_Z$.

\begin{defn}\label{defn: Roe algebra}
Let $\H_Z$ be an ample module of a proper metric space $(Z,d)$.
\begin{enumerate}
\item[(1)] For $T\in B(\H_Z)$, its \emph{support} $\supp(T)$ is defined to be the complement of the set of points $(x,y)\in Z \times Z$ for which there exist $f,g\in C_0(Z)$ satisfying:
\[
gTf=0,\quad f(x)\neq 0 \quad \mbox{and} \quad g(y) \neq 0.
\]
\item[(2)] The {\em propagation} of $T\in B(\H_Z)$ is defined to be
\[
\ppg(T):=\sup\{d(x,y):(x,y)\in \supp(T)\}.
\]
We say that $T$ has {\em finite propagation} if the number $\ppg(T)$ is finite.
\item[(3)] An operator $T\in B(\H_Z)$ is called \emph{locally compact} if $fT$ and $Tf$ are compact for any $f\in C_0(Z)$.
\end{enumerate}
\end{defn}

\begin{defn}\label{defn:Roe alg}
	For a proper metric space $Z$ and an ample $Z$-module $\H_Z$, the \emph{algebraic Roe algebra} $\CC[\H_Z]$ of $\H_Z$ is defined to be the $*$-algebra of locally compact finite propagation operators on $\H_Z$, and the \emph{Roe algebra $C^*(\H_Z)$} of $\H_Z$ is defined to be the norm-closure of $\CC[\H_Z]$ in $\B(\H_Z)$.
\end{defn}

It is a standard result that the Roe algebra $C^*(\H_Z)$ does not depend on the chosen ample module $\H_Z$ up to $*$-isomorphisms, hence denoted by $C^*(Z)$ and called the \emph{Roe algebra of $Z$}. Furthermore, $C^*(Z)$ is a coarse invariant of the metric space $Z$ (up to non-canonical $*$-isomorphisms), and their $K$-theories are coarse invariants up to canonical isomorphisms (see, \emph{e.g.}, \cite{roe1993coarse}).

Now we move to the equivariant case. Here we follow the setting of \cite[Section 4.5 and 5.2]{willett2020higher}, which is in fact equivalent to those in \cite{DFW21, FW16, Roe96}.

For a proper $\Gamma$-space $(Z,d)$, we define a $\Gamma$-action on $C_0(Z)$ by
\[
(\gamma\cdot f)(x)=f(\gamma^{-1}x)
\]
for all $\gamma \in \Gamma$ and $f\in C_0(Z)$. For a unitary representation $U:\Gamma\rightarrow \mathcal {U}(\H)$ on some Hilbert space $\H$, we denote the adjoint
\[
\Ad_{U_\gamma} (T) = U_\gamma T U_\gamma^* =:\gamma \cdot T
\]
for $\gamma \in \Gamma$ and $T \in \B(\H)$. An operator $T \in \B(\H)$ is called ($\Gamma$-)\emph{invariant} if $\gamma \cdot T = T$ for any $\gamma \in \Gamma$. Denote the set of all $\Gamma$-invariant operators by $\B(\H)^\Gamma$.

The following notion essentially comes from \cite[Section 4.5]{willett2020higher}.

\begin{defn}\label{defn: covariant system}
Let $(Z,d)$ be a proper $\Gamma$-space, and $(\H_Z,\phi)$ be an ample $Z$-module.
\begin{enumerate}
 \item $\H_Z$ is called \emph{covariant} if it is equipped with a unitary representation $U:\Gamma\rightarrow \mathcal {U}(\H_Z)$ such that for any $\gamma\in \Gamma$, we have:
\[
\phi(\gamma \cdot f) = \gamma \cdot \phi(f) \left(= U_\gamma \phi(f)  U_\gamma^*\right).
\]
 \item $\H_Z$ is called \emph{locally free} if for any finite subgroup $F$ of $\Gamma$ and any $F$-invariant Borel subset $E$ of $Z$, there exists a Hilbert space $\H_E$ equipped with the trivial representation of $F$ such that $\chi_E \H_Z$ and $\ell^2(F) \otimes \H_E$ are isomorphic as $F$-representations.
\end{enumerate}
To simplify the notation, we call an ample covariant locally free $Z$-module an \emph{admissible} $Z$-module.
\end{defn}

It follows from \cite[Lemma 4.5.5]{willett2020higher} that admissible modules always exist. Now we introduce the notion of equivariant Roe algebras (see, \emph{e.g.}, \cite[Definition 5.2.1]{willett2020higher}).

%
%
%

\begin{defn}
Let $(Z,d)$ be a proper $\Gamma$-space, and $(\H_Z,\phi)$ be an admissible $Z$-module. The \emph{algebraic equivariant Roe algebra} $\CC[\H_Z]^\Gamma$ of $\H_Z$ is defined to be the $*$-algebra of locally compact finite propagation $\Gamma$-invariant operators on $\H_Z$, and the \emph{equivariant Roe algebra $C^*(\H_Z)^\Gamma$} of $\H_Z$ is defined to be the norm-closure of $\CC[\H_Z]^\Gamma$ in $\B(\H_Z)$.
\end{defn}


Analogous to the case of Roe algebras, it follows from \cite[Theorem 5.2.6]{willett2020higher} that the equivariant Roe algebra $C^*(\H_Z)^\Gamma$ does not depend on the chosen admissible module $\H_Z$ up to $*$-isomorphisms, hence denoted by $C^*(Z)^\Gamma$ and called the \emph{equivariant Roe algebra of $Z$}. Moreover, equivariant Roe algebras are invariant under equivariant coarse equivalences (up to non-canonical $*$-isomorphisms), and their $K$-theories are equivariant coarse invariants up to canonical isomorphisms.

We would also like to point out an important observation from \cite[Lemma 5.14]{Roe96} that the equivariant Roe algebra $C^*(Z)^\Gamma$ is Morita equivalent to the reduced group $C^*$-algebra $C_r^*(\Gamma)$ provided that the action is cocompact in the sense that $Z/\Gamma$ is compact (see also \cite[Theorem 5.3.2]{willett2020higher}). This allows us to include the setting of the classical Baum-Connes conjecture to the scope of equivariant coarse Baum-Connes conjecture which will be introduced below.



\subsection{Equivariant $K$-homology groups and index maps}
Let us briefly recall the notion of equivariant $K$-homology groups introduced by Kasparov \cite{Kas88}.

\begin{defn}\label{defn: K-homology}
Let $(Z,d)$ be a proper $\Gamma$-space. For $i=0$ and $1$, the \emph{equivariant $K$-homology group} $K_i^\Gamma(Z)=KK_i^\Gamma(C_0(Z),\CC)$ is generated by certain cycles modulo a certain equivalent relation:

\begin{enumerate}[(1)]
\item each cycle for $K_0^\Gamma(Z)$ is a triple $(\H,\phi,F)$, where $\phi:C_0(Z)\to B(\H)$ is a covariant $*$-representation and $F\in B(\H)^\Gamma$ such that $\phi(f) F-F\phi(f)$, $\phi(f)(FF^*-I)$ and $\phi(f)(F^*F-I)$ are compact operators for all $f\in C_0(Z)$ and $\gamma\in \Gamma$;
\item each cycle for $K_1^\Gamma(Z)$ is a triple $(\H,\phi,F)$, where $\phi:C_0(Z)\to B(\H)$ is a covariant $*$-representation and $F\in B(\H)^\Gamma$ is self-adjoint such that $\phi(f)(F^2-I)$ and $\phi(f) F-F\phi(f)$ are compact for all $f\in C_0(Z)$ and $\gamma\in \Gamma$.
\end{enumerate}
In both cases, the equivalence relation on cycles is given by homotopy of the operator $F$.
\end{defn}

Now we define the \emph{equivariant index map} for a $\Gamma$-space $Z$:
\[
\Ind^\Gamma: K^\Gamma_*(Z)\longrightarrow K_*(C^*(Z)^\Gamma).
\]
Note that every class in $K^{\Gamma}_*(Z)$ can be represented by a cycle $(\H,\phi,F)$ such that $(\H,\phi)$ is an admissible $Z$-module (see, \emph{e.g.}, \cite[Section 3]{KY12}).

Consider such a cycle $(\H,\phi, F)$. Take a locally finite, $\Gamma$-equivariant and uniformly bounded open cover $\{U_i\}_{i\in I}$ of $Z$, and let $\{\psi_i\}_{i\in I}$ be a $\Gamma$-equivariant partition of unity subordinate to $\{U_i\}_{i\in I}$. Define
\[
F'=\sum\limits_{i\in I}\phi(\sqrt{\psi_i})F\phi(\sqrt{\psi_i}),
\]
where the sum converges in strong topology. It is clear that the cycle $(\H,\phi, F')$ is equivalent to $(\H, \phi, F)$ in $K^\Gamma_0(Z)$ and $F'$ has finite propagation. Hence $F'$ is a multiplier of $C^*(Z)^\Gamma$ and invertible modulo $C^*(Z)^\Gamma$. Applying the boundary map, $F'$ gives rise to an element in $K_0(C^*(Z)^\Gamma)$, denoted by $\partial([F'])$. We define the equivariant index map
\[
\Ind^\Gamma: K^\Gamma_0(Z)\longrightarrow K_0(C^*(Z)^\Gamma)
\]
by
\[
\Ind^\Gamma([(\H,\phi, F)])=\partial([F'])
\]
with the notation as above. Similarly, we can define the equivariant index map for the $K_1$-case:
\[
\Ind^\Gamma: K^\Gamma_1(Z)\longrightarrow K_1(C^*(Z)^\Gamma).
\]

\subsection{The equivariant coarse Baum--Connes conjecture}\label{ssec: equiv CBC}
In this subsection, we shall recall the equivariant coarse Baum--Connes conjecture for a discrete metric space with bounded geometry.

Let $(X,d)$ be a discrete metric space with bounded geometry. For $r>0$, the \emph{Rips complex $P_r(X)$ with scale $r$} is the simplicial complex with vertices $X$ such that a finite subset $Y=\{x_0, x_1, \cdots, x_n\}\subseteq X$ spans a simplex if and only if $\diam(Y) \leq r$.

The Rips complex $P_r(X)$ is endowed with the following \emph{spherical metric}. On each connected component of $P_r(X)$, the spherical metric is the maximal metric whose restriction on each simplex $\Delta:=\{\sum_{i=0}^n t_ix_i : t_i\ge0,\sum_{i=0}^n t_i=1\} $ is the metric obtained by identifying $\Delta$ with $ S_+^n $ via the map
\[
\sum_{i=0}^n t_ix_i \mapsto \left(\frac{t_0}{\sqrt{\sum_{i=0}^n t_i^2}},\frac{t_1}{\sqrt{\sum_{i=0}^n t_i^2}},\dots,\frac{t_n}{\sqrt{\sum_{i=0}^n t_i^2}}\right),
\]
where $S_+^n:=\{(s_0,s_1,\dots,s_n)\in\RR^{n+1} : s_i\ge0,\sum_{i=0}^n s_i^2=1\} $ is endowed with the standard Riemannian metric. If $y_0$ and $y_1$ belong to two different connected components $Y_0$ and $Y_1$ of $P_r(X)$, respectively, we define
\[
d(y_0,y_1)=\min\big\{d(y_0,x_0)+d_X(x_0,x_1)+d(x_1,y_1) : x_0\in X\cap Y_0,x_1\in X\cap Y_1\big\}.
\]


Assume further that $X$ is a $\Gamma$-space. Then each Rips complex $P_r(X)$ also admits a $\Gamma$-action defined by
\[
\gamma \cdot \sum\limits_{i=1}^k c_i x_i=\sum\limits_{i=1}^k c_i( \gamma x_i)
\]
for all $\sum\limits_{i=1}^k c_i x_i \in P_r(X)$ and $\gamma\in \Gamma$. It is clear that the action is proper and isometric, and hence $P_r(X)$ is a proper $\Gamma$-space as well.

The following is the main conjecture we study in this paper:

\begin{conj}[The equivariant coarse Baum--Connes conjecture]
Let $X$ be a discrete metric space with bounded geometry and $\Gamma$ be a countable discrete group acting on $X$ properly by isometries. Then the equivariant index map
\[
\Ind^\Gamma: \lim_{r\rightarrow\infty} K_*^\Gamma(P_r(X))\longrightarrow \lim_{r\rightarrow\infty}K_*(C^*(P_r(X))^\Gamma)\cong K_*(C^*(X)^\Gamma)
\]
is an isomorphism for $\ast=0,1$.
\end{conj}

As noticed in Introduction, the equivariant coarse Baum-Connes conjecture includes both the classical Baum-Connes conjecture for groups and the coarse Baum-Connes conjecture for metric spaces as special cases. More precisely, it coincides with the Baum-Connes conjecture when the action is cocompact, and when the action is trivial it is nothing but the coarse Baum-Conness conjecture (see, \emph{e.g.}, \cite[Section 7]{willett2020higher}).

We reformulate the following significant result due to Higson and Kasparov \cite{HK01}. Recall that a discrete group $\Gamma$ is said to be \emph{a-T-menable} \cite{Gro93} (or has the \emph{Haagerup property}) if it admits a metrically proper action on some Hilbert space by affine isometries.

\begin{prop}[\cite{HK01}]\label{prop:equiv-CBC for cocompact}
Let $\Gamma$ be a countable discrete a-T-menable group acting properly and cocompactly on a proper metric space $X$ with bounded geometry by isometries. Then the equivariant index map
\[
\mathrm{Ind}^\Gamma: \lim_{r\to \infty} K^\Gamma_\ast(P_r(X)) \longrightarrow K_\ast(C^*(X)^\Gamma)
\]
is an isomorphism for $\ast=0,1$.
\end{prop}

\subsection{Equivariant localisation algebras}

Here we introduce another version on the $K$-homology group, which was originally introduced by Yu \cite{Yu97}.

\begin{defn}\label{defn: equiv localisation alg}
Let $(Z,d)$ be a proper $\Gamma$-space. The \emph{equivariant localisation algebra} $C_L^*(Z)^\Gamma$ is the supremum norm closure of the algebra of all bounded and uniformly norm-continuous functions
\[
f:[1,+\infty)\rightarrow C^*(Z)^\Gamma
\]
such that
\[
\ppg(f(t))\rightarrow 0 \quad \mbox{and} \quad t\rightarrow\infty.
\]
\end{defn}

We recall that for a proper $\Gamma$-space $Z$, there exists an equivariant local index map which originally comes from \cite{Yu97} (with an equivariant formulation from \cite{FWY20}):
\[
\Ind^\Gamma_L: K_*^\Gamma(Z)\longrightarrow K_*(C^*_L(Z)^\Gamma).
\]
Moreover, we have the following result which is an equivariant analogue of \cite[Theorem 3.2]{Yu97} and \cite[Theorem 3.4]{QR10}.

\begin{prop}\label{prop: equiv localisation alg = K-homology}
Let $Z$ be a proper $\Gamma$-space. Then the equivariant local index map
\[
\Ind^\Gamma_ L: K_*^{\Gamma}(Z)\rightarrow K_*(C_L^*(Z)^{\Gamma})
\]
is an isomorphism for $\ast=0,1$.
\end{prop}

On the other hand, note that for a proper $\Gamma$-space $Z$ there is a natural evaluation-at-one map
\[
\ev:C_L^*(Z)^{\Gamma} \to C^*(Z)^{\Gamma}
\]
given by
\[
\ev(f)=f(1)
\]
for all $f\in C_L^*(Z)^{\Gamma}$. This is clearly a $\ast$-homomorphism, and thus induces a homomorphism
\[
\ev_\ast:K_*(C_L^*(Z)^{\Gamma}) \to K_*(C^*(Z)^{\Gamma})
\]
on $K$-theories. Furthermore, we have the following commutative diagram:
\[
\xymatrix{
                &         K_*(C_L^*(Z)^{\Gamma}) \ar[dr]^{\ev_\ast}   &  \\
  K_*^{\Gamma}(Z) \ar[ur]_{\cong}^{\Ind_L^\Gamma}  \ar[rr]^{\Ind^\Gamma} & & K_*(C^*(Z)^\Gamma).
  }
\]

Consequently, for a discrete $\Gamma$-space $X$ with bounded geometry, in order to prove that equivariant coarse Baum-Connes conjecture holds for $X$ it suffices to show that the map
\[
\ev_\ast:\lim_{r\rightarrow \infty}K_*(C_L^*(P_r(X))^{\Gamma})\longrightarrow \lim_{r\rightarrow \infty}K_*(C^*(P_r(X))^{\Gamma}) \cong K_*(C^*(X)^\Gamma)
\]
induced by the evaluation-at-one map is an isomorphism for $\ast=0,1$.

\section{Reduction}\label{sec:reduction}

The aim of this section is to reduce the proof of Theorem \ref{thm:main result equiv. CBC} to the case of sequences of metric spaces with proper and cocompact group actions. Let us recall the following notion:

\begin{defn}[\cite{HRY93}]\label{defn:omega-excisive}
Let $Z$ be a proper metric space and $A,B$ be closed subsets in $Z$ such that $A \cup B = Z$. We say that $(A,B)$ is  \emph{$\omega$-excisive} if for any $r>0$ there exists $s>0$ such that
\[
\Nd_r(A) \cap \Nd_r(B) \subseteq \Nd_s(A \cap B).
\]
\end{defn}

Throughout the rest of this section, let us assume that $(X,d)$ is a proper $\Gamma$-space with bounded geometry. Denote by $\pi: X \to X/\Gamma$ the quotient map and equip $X /\Gamma$ with the quotient metric $\bar{d}$.

\begin{lem}\label{lem:pre-image}
For any $Z \subseteq X / \Gamma$ and $r>0$, we have $\Nd_r(\pi^{-1}(Z)) = \pi^{-1}(\Nd_r(Z))$.
\end{lem}

\begin{proof}
For any $x\in \Nd_r(\pi^{-1}(Z))$, there exists $y\in \pi^{-1}(Z)$ such that $d(x,y) \leq r$. Hence we have $\bar{d}(\pi(x),\pi(y)) \leq d(x,y) \leq r$, which implies that $x\in \pi^{-1}(\Nd_r(Z))$.

On the other hand, given $x \in \pi^{-1}(\Nd_r(Z))$ we have $\pi(x) \in \Nd_r(Z)$. Hence there exists $z\in Z$ such that $\bar{d}(\pi(x),z) \leq r$. Assume that $z=\pi(y)$ for some $y\in X$. Hence,
\[
r \geq \bar{d}(\pi(x),z) = \inf_{\gamma, \gamma' \in \Gamma} d(\gamma x, \gamma' y) = \inf_{\gamma}d(x,\gamma y)
\]
where we use that the action is by isometries in the third equality. Since $X$ has bounded geometry and $\gamma y \in \pi^{-1}(Z)$ for each $\gamma\in \Gamma$, we obtain that $x\in \Nd_r(\pi^{-1}(Z))$.
\end{proof}

\begin{lem}\label{lem:pre-image omega-excisive}
Suppose $(A,B)$ is an $\omega$-excisive closed cover of $X / \Gamma$. Then $(\pi^{-1}(A), \pi^{-1}(B))$ is an $\omega$-excisive closed cover of $X$.
\end{lem}

\begin{proof}
Given $r>0$ there exists $s>0$ such that $\Nd_r(A) \cap \Nd_r(B) \subseteq \Nd_s(A \cap B)$. Hence by Lemma \ref{lem:pre-image}, we have:
\begin{align*}
\Nd_r(\pi^{-1}(A)) \cap \Nd_r(\pi^{-1}(B)) & = \pi^{-1}(\Nd_r(A)) \cap \pi^{-1}(\Nd_r(B)) = \pi^{-1}(\Nd_r(A) \cap \Nd_r(B))\\
 &\subseteq \pi^{-1}(\Nd_s(A \cap B)) = \Nd_s(\pi^{-1}(A \cap B)) = \Nd_s(\pi^{-1}(A) \cap \pi^{-1}(B)),
\end{align*}
which concludes the proof.
\end{proof}

Now for the quotient space $X/ \Gamma$, we fix a basepoint $w_0 \in X/ \Gamma$. For each $n \in \NN\cup\{0\}$, we set
\[
W_n:=\{w\in X/ \Gamma: n^3-n\le \bar{d}(w,w_0)\le (n+1)^3+(n+1)\}.
\]
Let $A:=\bigsqcup_{n: even} W_n$ and $B:=\bigsqcup_{n: odd} W_n$. It is obvious that $(A,B)$ is an $\omega$-excisive cover of $X/ \Gamma$. Hence Lemma \ref{lem:pre-image omega-excisive} implies that $(\pi^{-1}(A), \pi^{-1}(B))$ is an $\omega$-excisive closed cover of $X$ such that both $\pi^{-1}(A)$ and $\pi^{-1}(B)$ are $\Gamma$-invariant. Denote $A'=\pi^{-1}(A)$ and $B'=\pi^{-1}(B)$.
	
Applying the Mayer-Vietoris sequence arguments to the $\Gamma$-invariant $\omega$-excisive cover $(A',B')$ for the $K$-theory of equivariant Roe algebras and the equivariant coarse $K$-homology (which is similar to \cite{shan2008equivariant} and \cite[Section 7.5]{willett2020higher}), we obtain the following commutative diagram of long-exact sequences:
\begin{footnotesize}
\[
 \begin{CD}
 \cdots @>>>\lim\limits_{r\to \infty}K_*^\Gamma(P_r(A'\cap B')) @>>> \lim\limits_{r\to \infty}K_*^\Gamma(P_r(A'))\oplus K_*^\Gamma(P_r(B')) @>>> \lim\limits_{r\to \infty} K_*^\Gamma(P_r(X)) @>>> \cdots\\
 & &   @VVV            @VVV            @VVV \\
 \cdots @>>>K_*(C^*(A'\cap B')^\Gamma) @>>> K_*(C^*(A')^\Gamma)\oplus K_*(C^*(B')^\Gamma) @>>> K_*(C^*(X)^\Gamma) @>>> \cdots.
 \end{CD}
\]
\end{footnotesize}
Hence it suffices to prove Theorem \ref{thm:main result equiv. CBC} for the spaces $A', B'$ and $A' \cap B'$. Therefore, we obtain the following:

\begin{lem}\label{lem: first reduction}
To prove Theorem \ref{thm:main result equiv. CBC}, it suffices to prove the result for a discrete metric space $(X,d)$ of bounded geometry and having the form $X=\bigsqcup_{n=1}^\infty X_n$ with $d(X_n, X_m) \to \infty$ as $n+m \to \infty~ (n\neq m)$ such that $\Gamma$ acts on each $X_n$ properly and cocompactly by isometries.
\end{lem}

Our next aim is to further reduce the proof of Theorem \ref{thm:main result equiv. CBC} to the case of block-diagonal operators. First note that the operations of taking direct limit and direct sum commute. Hence as a consequence of Proposition \ref{prop:equiv-CBC for cocompact}, we obtain the following:

%

\begin{cor}\label{cor:equiv-CBC for direct sum}
With the same notation as above and assuming that $\Gamma$ is a-T-menable, the following map induced by the equivariant index map
\[
\lim\limits_{r\to \infty} \bigoplus_n K^\Gamma_\ast(P_r(X_n)) \longrightarrow  \bigoplus_n K_\ast(C^*(X_n)^\Gamma)
\]
is an isomorphism for $\ast=0,1$.
\end{cor}

Now for each $r>0$, we consider the following $C^*$-subalgebras in $C^*(P_r(X))^\Gamma$ and $C^*_L(P_r(X))^\Gamma$, respectively:
\[
\A_r:=\lim\limits_{n\to \infty} C^*\big(P_r\big(\bigsqcup_{k \leq n} X_k\big)\big)^\Gamma = \overline{\bigcup_{n=1}^\infty C^*\big(P_r\big(\bigsqcup_{k \leq n} X_k\big)\big)^\Gamma}
\]
and
\[
\A_{L,r}:=\lim\limits_{n\to \infty} C^*_L\big(P_r\big(\bigsqcup_{k \leq n} X_k\big)\big)^\Gamma = \overline{\bigcup_{n=1}^\infty C^*_L\big(P_r\big(\bigsqcup_{k \leq n} X_k\big)\big)^\Gamma}.
\]
The following lemma is straightforward, hence we omit the proof.

\begin{lem}\label{lem:A and Ar}
With the same notation as above and for $r>0$, we have the following:
\begin{eqnarray*}
C^*(P_r(X))^\Gamma &=& \A_r + \big( C^*(P_r(X))^\Gamma \cap \prod_{n=1}^\infty C^*(P_r(X_n))^\Gamma \big),\\
\bigoplus_{n=1}^\infty C^*(P_r(X_n))^\Gamma &=& \A_r \cap \big( C^*(P_r(X))^\Gamma \cap \prod_{n=1}^\infty C^*(P_r(X_n))^\Gamma \big).
\end{eqnarray*}
The same holds for the case of equivariant localisation algebras.
\end{lem}

Now we are able to reduce the proof of Theorem \ref{thm:main result equiv. CBC} to the case of block diagonal operators.

\begin{prop}\label{prop:reduction to block diagonal}
Let $(X,d)$ be a discrete metric space of bounded geometry and having the form $X=\bigsqcup_{n=1}^\infty X_n$ with $d(X_n, X_m) \to \infty$ as $n+m \to \infty~ (n\neq m)$ such that $\Gamma$ acts on each $X_n$ properly and cocompactly by isometries. To prove that
\[
\mathrm{Ind}^\Gamma: \lim_{r\to \infty} K_\ast^\Gamma(P_r(X)) \longrightarrow K_\ast(C^*(X)^\Gamma)
\]
is an isomorphism for $\ast=0,1$, it suffices to prove the map
\[
\ev_\ast: \lim_{r\to \infty} K_\ast \big(C^*_L(P_r(X))^\Gamma \cap \prod_{n=1}^\infty C^*_L(P_r(X_n))^\Gamma \big) \longrightarrow K_\ast \big(C^*(P_r(X))^\Gamma \cap \prod_{n=1}^\infty C^*(P_r(X_n))^\Gamma \big)
\]
induced by the evaluation-at-one map is an isomorphism for $\ast=0,1$.
\end{prop}

\begin{proof}
First note that for each $r>0$, we have the following commutative diagram:
\begin{scriptsize}
\[
\xymatrix{
			\cdots \ar[r] & \bigoplus\limits_{n=1}^\infty K_\ast(C^*_L(P_r(X_n))^\Gamma)  \ar[r]  \ar[d] & K_\ast\big( C^*_L(P_r(X))^\Gamma \cap \prod\limits_{n=1}^\infty C^*_L(P_r(X_n))^\Gamma \big) \ar[r]  \ar[d] & K_\ast \big( \frac{C^*_L(P_r(X))^\Gamma \cap \prod_{n=1}^\infty C^*_L(P_r(X_n))^\Gamma}{\bigoplus_{n=1}^\infty C^*_L(P_r(X_n))^\Gamma} \big) \ar[r]  \ar[d] & \cdots \\
			\cdots \ar[r] & \bigoplus\limits_{n=1}^\infty K_\ast(C^*(P_r(X_n))^\Gamma)  \ar[r]  & K_\ast\big( C^*(P_r(X))^\Gamma \cap \prod\limits_{n=1}^\infty C^*(P_r(X_n))^\Gamma \big) \ar[r]  & K_\ast \big( \frac{C^*(P_r(X))^\Gamma \cap \prod_{n=1}^\infty C^*(P_r(X_n))^\Gamma}{\bigoplus_{n=1}^\infty C^*(P_r(X_n))^\Gamma} \big)\ar[r]  & \cdots,
	}
\]
\end{scriptsize}
where all vertical lines are induced by evaluation-at-one maps and horizontal lines are exact sequences.
When $r\to \infty$, it follows from Corollary \ref{cor:equiv-CBC for direct sum} and the assumption that the left and middle vertical lines are isomorphisms. Hence we obtain that the right vertical line is also an isomorphism when $r\to \infty$.

From Lemma \ref{lem:A and Ar}, we obtain the following for each $r>0$:
\[
\frac{C^*(P_r(X))^\Gamma \cap \prod_{n=1}^\infty C^*(P_r(X_n))^\Gamma}{\bigoplus_{n=1}^\infty C^*(P_r(X_n))^\Gamma} \cong \frac{C^*(P_r(X))^\Gamma}{\A_r},
\]
and
\[
\frac{C^*_L(P_r(X))^\Gamma \cap \prod_{n=1}^\infty C^*_L(P_r(X_n))^\Gamma}{\bigoplus_{n=1}^\infty C^*_L(P_r(X_n))^\Gamma} \cong \frac{C^*_L(P_r(X))^\Gamma}{\A_{L,r}}.
\]
Hence the above implies that
\[
\lim\limits_{r\to \infty} K_\ast\big( \frac{C^*_L(P_r(X))^\Gamma}{\A_{L,r}} \big) \longrightarrow \lim\limits_{r\to \infty} K_\ast\big( \frac{C^*(P_r(X))^\Gamma}{\A_r} \big)
\]
is an isomorphism for $\ast=0,1$.

On the other hand, we have the following commutative diagram for each $r>0$:
\begin{small}
\[
\xymatrix{
			\cdots \ar[r] & K_\ast(\A_{L,r})  \ar[r]  \ar[d] & K_\ast ( C^*_L(P_r(X))^\Gamma ) \ar[r]  \ar[d] & K_\ast \big( \frac{C^*_L(P_r(X))^\Gamma}{\A_{L,r}} \big) \ar[r]  \ar[d] & \cdots \\
			\cdots \ar[r] & K_\ast(\A_r)  \ar[r]   & K_\ast ( C^*(P_r(X))^\Gamma ) \ar[r]  & K_\ast \big( \frac{C^*(P_r(X))^\Gamma}{\A_r} \big) \ar[r]   & \cdots,
	}
\]
\end{small}
where all vertical lines are induced by evaluation-at-one maps and horizontal lines are exact sequences.
The analysis above implies that the third vertical line is an isomorphism when $r \to \infty$. Hence to finish the proof, it suffices to show that
\[
\lim_{r \to \infty} K_\ast(\A_{L,r}) \longrightarrow \lim_{r \to \infty} K_\ast(\A_r)
\]
is an isomorphism for $\ast=0,1$. Consider the following commutative diagram:
\begin{small}
\[
\xymatrix{
			\lim\limits_{r\to \infty} K_\ast(\A_{L,r})  \ar[r]^<<<<{\displaystyle \cong}  \ar[d] & \lim\limits_{r\to \infty}\lim\limits_{n\to \infty} K_\ast \big( C^*_L\big(P_r\big(\bigsqcup_{k \leq n} X_k\big)\big)^\Gamma \big) \ar[r]^{\displaystyle \cong} \ar[d] & \lim\limits_{n\to \infty}\lim\limits_{r\to \infty} K_\ast \big( C^*_L\big(P_r\big(\bigsqcup_{k \leq n} X_k\big)\big)^\Gamma \big) \ar[d]\\
			\lim\limits_{r\to \infty} K_\ast(\A_r)  \ar[r]^<<<<<{\displaystyle \cong}  & \lim\limits_{r\to \infty}\lim\limits_{n\to \infty} K_\ast \big( C^*\big(P_r\big(\bigsqcup_{k \leq n} X_k\big)\big)^\Gamma \big) \ar[r]^{\displaystyle \cong}  & \lim\limits_{n\to \infty}\lim\limits_{r\to \infty} K_\ast \big( C^*\big(P_r\big(\bigsqcup_{k \leq n} X_k\big)\big)^\Gamma \big),
	}
\]
\end{small}
where all vertical lines are induced by evaluation-at-one maps. Note that the right two horizontal lines are isomorphisms due to the continuity of $K$-theory and the following isomorphisms
\[
\lim\limits_{r\to \infty}\lim\limits_{n\to \infty} C^*_L\big(P_r\big(\bigsqcup_{k \leq n} X_k\big)\big)^\Gamma  \cong  \lim\limits_{n\to \infty}\lim\limits_{r\to \infty} C^*_L\big(P_r\big(\bigsqcup_{k \leq n} X_k\big)\big)^\Gamma
\]
and
\[
\lim\limits_{r\to \infty}\lim\limits_{n\to \infty} C^*\big(P_r\big(\bigsqcup_{k \leq n} X_k\big)\big)^\Gamma  \cong  \lim\limits_{n\to \infty}\lim\limits_{r\to \infty} C^*\big(P_r\big(\bigsqcup_{k \leq n} X_k\big)\big)^\Gamma
\]
(note that both limits are just increasing unions of subalgebras). For each $n\in \NN$, it follows from Proposition \ref{prop:equiv-CBC for cocompact} that
\[
\ev_\ast: \lim\limits_{r\to \infty} K_\ast \big( C^*_L\big(P_r\big(\bigsqcup_{k \leq n} X_k\big)\big)^\Gamma \big) \longrightarrow \lim\limits_{r\to \infty} K_\ast \big( C^*\big(P_r\big(\bigsqcup_{k \leq n} X_k\big)\big)^\Gamma \big)
\]
is an isomorphism. Hence we conclude the proof.
\end{proof}

Consequently, we obtain the following:

\begin{cor}\label{cor:final reduction}
To prove Theorem \ref{thm:main result equiv. CBC}, it suffices to prove the following: Let $\Gamma$ be a countable discrete group and $(X,d)$ be a discrete metric space of bounded geometry and having the form $X=\bigsqcup_{n=1}^\infty X_n$ with $d(X_n, X_m) \to \infty$ as $n+m \to \infty~ (n\neq m)$ such that $\Gamma$ acts on each $X_n$ properly and cocompactly by isometries. Assume that the action of $\Gamma$ on $X$ has controlled distortion, the quotient space $X / \Gamma$ admits a coarse embedding into Hilbert space and $\Gamma$ is a-T-menable, then the following map
\begin{equation}\label{EQ:main thm}
\ev_\ast: \lim_{r\to \infty} K_\ast \big(C^*_L(P_r(X))^\Gamma \cap \prod_{n=1}^\infty C^*_L(P_r(X_n))^\Gamma \big) \longrightarrow K_\ast \big(C^*(P_r(X))^\Gamma \cap \prod_{n=1}^\infty C^*(P_r(X_n))^\Gamma \big)
\end{equation}
induced by the evaluation-at-one map is an isomorphism for $\ast=0,1$.
\end{cor}

\section{Equivariant twisted algebras}\label{sec:twisted algebras}

In this section, we define equivariant twisted Roe and localisation algebras, which follows the constructions in \cite[Section 12.6]{willett2020higher} and originally comes from \cite[Section 5]{Yu00}.

\subsection{The Bott-Dirac operators on Euclidean spaces}\label{ssec:The Bott-Dirac operators on Euclidean spaces}
Let us start by recalling the Bott-Dirac operators. Here we only list necessary notions and facts, and guide readers to \cite[Section 12.1]{willett2020higher} for details.

Let $E$ be a real Hilbert space (also called a \emph{Euclidean space}) with even dimension $d \in \NN$. The \emph{complexified Clifford algebra} of $E$, denoted by $\Cliff(E)$, is the universal unital complex algebra containing $E$ as a real subspace and subject to the multiplicative relations $x \cdot x = \|x\|_E^2$ for all $x\in E$. It is natural to treat $\Cliff(E)$ as a graded Hilbert space (see for example \cite[Example E.2.12]{willett2020higher}), and in this case we denote it by $\H_E$.

Denote $\L _E^2$ the graded Hilbert space of square integrable functions from $E$ to $\H_E$ where the grading is inherited from $\H_E$, and $\SS_E$ the dense subspace consisting of Schwartz class functions from $E$ to $\H_E$. Fix an orthonormal basis $\{e_1,\ldots,e_d\}$ of $E$ and let $x_1,\ldots,x_d:E \to \RR$ be the corresponding coordinates. Recall that the \emph{Bott operator} $C$ and the \emph{Dirac operator} $D$ are unbounded operators on $\L _E^2$ with domain $\SS_E$ defined as:
\[
(Cu)(x) = x \cdot u(x), \quad \mbox{and} \quad (Du)(x)=\sum_{i=1}^d \hat{e_i} \cdot \frac{\partial u}{\partial x_i}(x)
\]
for $u \in \SS_E$ and $x\in E$, where $\hat{e_i}: \Cliff(E) \to \Cliff(E)$ is the operator determined by $\hat{e_i} (w)=(-1)^{\partial w} w \cdot e_i$ for any homogeneous element $w\in \Cliff(E)$.

\begin{defn}\label{defn:Bott-Dirac}
	The \emph{Bott-Dirac operator} is the unbounded operator $B:=D+C$ on $\L _E^2$ with domain $\SS_E$.
\end{defn}

Given $x\in E$, the \emph{left Clifford multiplication operator associated to $x$} is the bounded operator $c_x$ on $\L _E^2$ defined as the left Clifford multiplication by the fixed vector $x$, and the \emph{translation operator associated to $x$} is the unitary operator $V_x$ on $\L _E^2$ defined by $(V_xu)(y):=u(y-x)$. Given $s\in [1,\infty)$, the \emph{shrinking operator associated to $s$} is the unitary operator $S_s$ on $\L _E^2$ defined by $(S_s u)(y):=s^{{-d}/{2}}u(sy)$.

\begin{defn}\label{defn:Bott-Dirac modified}
	For $s\in [1,\infty)$ and $x\in E$, the \emph{Bott-Dirac operator associated to $(x,s)$} is the unbounded operator $B_{s,x}:=s^{-1}D+C-c_x$ on $\L _E^2$ with domain $\SS_E$.
\end{defn}

Note that $B_{1,0}=B$ and $B_{s,x}=s^{{-1}/{2}}~V_xS_{\sqrt{s}}BS^*_{\sqrt{s}}V_x^*$. It is also known that for each $s\in [1,\infty)$ and $x\in E$, the operator $B_{s,x}$ is unbounded, odd, essentially self-adjoint and maps $\SS_E$ to itself (see, \emph{e.g.}, \cite[Corollary 12.1.4]{willett2020higher}).

\begin{defn}\label{defn:F_{s,x}}
	Let $s\in [1,\infty)$, $x\in E$ and $B_{s,x}$ be the Bott-Dirac operator associated to $(x,s)$. Define a bounded operator on $\L_E^2$ by:
	\[
	F_{s,x}:= B_{s,x}(1+B^2_{s,x})^{-1/2}.
	\]
\end{defn}

We list several important properties of the operator $F_{s,x}$. For simplicity, denote $\chi_{x,R}:=\chi_{B(x,R)}$ for $x\in E$ and $R \geq 0$.
\begin{prop}[{\cite[Proposition 12.1.10]{willett2020higher}}]\label{prop:Psi function}
	For each $\varepsilon >0$ there exists an odd function $\Psi :\RR \rightarrow [-1,1]$ with $\Psi (t)\rightarrow 1$ as $t\rightarrow +\infty$, satisfying the following:
	\begin{enumerate}
		\item For all $s\in [1,\infty)$ and $x\in E$, we have $\|F_{s,x}-\Psi(B_{s,x})\|< \varepsilon$.
		\item There exists $R_0>0$ such that for all $s\in [1,\infty)$ and $x\in E$, we have $\ppg_E(\Psi(B_{s,x})) \le s^{-1}R_0$.
		\item For all $s\in [1,\infty)$ and $x\in E$, $\Psi(B_{s,x})^2 - 1$ is compact.
		\item For all $s\in [1,\infty)$ and $x,y\in E$, $\Psi(B_{s,x}) - \Psi(B_{s,y})$ is compact.
		\item For all $s\in [1,\infty)$ and $x,y\in E$, $\|F_{s,x}-F_{s,y}\|\le 3\|x-y\|_E$. And there exists $c>0$ such that for all $s\in [1,\infty)$ and $x,y\in E$, we have
		\[
		\|\Psi(B_{s,x}) - \Psi(B_{s,y})\| \leq c \|x-y\|_E.
		\]
		\item For all $x\in E$, the function
		$$[1,\infty)\rightarrow \B(\L_E^2),\ s \mapsto \Psi(B_{s,x})$$
		is strong-$*$ continuous.
		\item The family of functions
		\[
		[1,\infty) \to \B(\L_E^2), \ s \mapsto \Psi(B_{s,x})^2-1
		\]
		is norm equi-continuous as $x$ varies over $E$ and $s$ varies over any fixed compact subset of $[1,\infty)$.
		\item For any $r \geq 0$, the family of functions
		\[
		[1,\infty) \to \B(\L_E^2), \ s \mapsto \Psi(B_{s,x}) - \Psi(B_{s,y})
		\]
		is norm equi-continuous as $(x,y)$ varies over the elements of $E \times E$ with $|x-y| \leq r$, and $s$ varies over any fixed compact subset of $[1,\infty)$.
		\item For any $\varepsilon_1>0$, there exists $R_1 > 0$ such that for all $R\ge R_1$, $s\ge 2d$ and $x\in E$, we have
		$$\| ( \Psi(B_{s,x})^2-1 )( 1-\chi_{x,R} ) \|< \varepsilon_1.$$
		\item For any $\varepsilon_2>0, r>0$ there exists $R_2 >0$ such that for all $R\ge R_2$, $s\ge 2d$ and $x,y \in E$ with $\|x-y\|_E\le r$, we have
		$$\|(\Psi(B_{s,x})-\Psi(B_{s,y}))(1-\chi_{x,R})\|< \varepsilon_2.$$
	\end{enumerate}
	Moreover, we can require that the function $\Psi$, constants $R_0$ in (2), $c$ in (5), $R_1$ in (9) and $R_2$ in (10) are independent of the dimension $d$ of the Euclidean space $E$.
\end{prop}

The strong-$\ast$ topology used in (6) above is defined as follows: A net $(T_i)$ of bounded operators converges to a bounded operator $T$ in the \emph{strong-$\ast$ topology} if all $v$ in the underlying Hilbert space, $T_i v \to Tv$ and $T_i^* v \to T^* v$. We need the following elementary result:

\begin{lem}[{\cite[Lemma 12.3.5]{willett2020higher}}]\label{lem:12.3.5 from WY21}
Let $\mathcal{S}$ and $\mathcal{T}$ be norm bounded sets of operators on a Hilbert space $\H$ such that $\mathcal{T}$ consists only of compact operators. Equip $\mathcal{S}$ with the strong-$\ast$ topology and $\mathcal{T}$ with the norm topology. Then the product maps
\[
\mathcal{S} \times \mathcal{T} \to \mathfrak{K} \quad \mbox{and} \quad \mathcal{T} \times \mathcal{S} \to \mathfrak{K}
\]
are jointly continuous where $\mathfrak{K}$ denotes the compact operators on $\H$.
\end{lem}

\subsection{Equivariant twisted algebras}\label{ssec:equiv. twisted algebras}

Throughout the rest of this paper, thanks to Section \ref{sec:reduction}, we always assume that $(X,d)$ is a discrete metric space of bounded geometry and having the form $X=\bigsqcup_{n=1}^\infty X_n$ with $d(X_n, X_m) \to \infty$ as $n+m \to \infty~ (n\neq m)$ such that the a-T-menable group $\Gamma$ acts on each $X_n$ properly and cocompactly by isometries. Furthermore, assume that all $\Gamma$-orbit maps are uniformly coarsely equivalent and the quotient space admits a coarse embedding $\xi: X/\Gamma \to \H$ into some Hilbert space $\H$.

Denoting by $E$ the underlying Euclidean space of $H$, it is clear that the map $\xi: X/\Gamma \to E$ is also a coarse embedding. Note that each $X_n/\Gamma$ is finite, hence the image of $\xi$ restricted to $X_n/\Gamma$ sits inside a finite-dimensional Euclidean space, denoted by $E_n$. Without loss of generality, we assume that the dimension $d_n$ of $E_n$ is even. For $n\in \NN$, denote $f_n:=\xi \circ \pi|_{X_n}: X_n \to E_n$, where $\pi: X\to X/\Gamma$ is the quotient map.

For $n\in \NN$ and $r>0$, denote the Rips complex of $X_n$ with scale $r$ by $P_{r,n}:=P_r(X_n)$. The $\Gamma$-action on $X_n$ induces a proper and cocompact $\Gamma$-action on $P_{r,n}$ by isometries as explained in Section \ref{ssec: equiv CBC}.

For $x\in X_n$, we consider the open star neighbourhood $B_{x,r,n}:=St_{n,r}(x)$ of $x$ in the barycentric subdivision of $P_{r,n}$. Note that $B_{x,r,n} \cap B_{y,r,n} = \emptyset$ for $x\neq y \in X_n$, $\gamma \cdot B_{x,r,n} = B_{\gamma x,r,n}$ for $\gamma \in \Gamma$, and $B_{x,r,n} \subseteq B_{x,s,n}$ for each $r \leq s$. Taking $Y_{r,n}:=\bigcup_{x\in X_n} B_{x,r,n}$, it is clear that $Y_{r,n}$ is $\Gamma$-invariant (since the action is by isometries) and dense in $P_{r,n}$. Moreover, let $Z_{r,n}$ be the collection of points $\sum {c_i x_i}$ in $Y_{r,n}$ such that all the coefficients $c_i$ take rational values. It is clear that $Z_{r,n}$ is a $\Gamma$-invariant countable subset which is also dense in $P_{r,n}$.

For $n\in \NN$ and $r>0$, we extend the above $f_n: X_n \to E_n$ to $f_{r,n}: Y_{r,n} \to E_n$ by sending all points in $B_{x,r,n}$ to $f_n(x)$ for $x\in X_n$. It is clear that $f_{r,n}(\gamma \cdot y) = f_{r,n}(y)$ for any $\gamma \in \Gamma$ and $y\in Y_{r,n}$.

Let us fix an infinite-dimensional separable Hilbert space $\HH$. For $r>0$ and $n\in \NN$, consider the following Hilbert spaces:
\[
\H_{r,n}:=\ell^2(Z_{r,n}) \otimes \HH \otimes \ell^2(\Gamma) \quad \mbox{and} \quad \H_{r,n,E}:=\ell^2(Z_{r,n}) \otimes \HH \otimes \ell^2(\Gamma) \otimes \L _{E_n}^2.
\]
The group $\Gamma$ acts on $\H_{r,n}$ (respectively, $\H_{r,n,E}$) as follows: for any $\gamma \in \Gamma$,
\begin{eqnarray*}
U_\gamma: & \delta_z \otimes \xi \otimes \delta_g  &\mapsto  ~~\delta_{\gamma z} \otimes \xi \otimes \delta_{\gamma g}\\
(\mbox{respectively,} ~U_\gamma: & \delta_z \otimes \xi \otimes \delta_g \otimes u & \mapsto ~~\delta_{\gamma z} \otimes \xi \otimes \delta_{\gamma g} \otimes u).
\end{eqnarray*}
It is clear that $\H_{r,n}$ is an admissible $P_{r,n}$-module under the amplified multiplication representation, and similarly $\H_{r,n,E}$ is both an admissible $P_{r,n}$-module and an ample $E_n$-module. We use these modules to the build equivariant Roe algebras $C^*(\H_{r,n})^\Gamma$ and $C^*(\H_{r,n,E})^\Gamma$ of $P_{r,n}$, and the equivariant localisation algebras $C^*_L(\H_{r,n})^\Gamma$ and $C^*_L(\H_{r,n,E})^\Gamma$ of $P_{r,n}$. Moreover for $T\in \B(\H_{r,n,E})$, we write $\ppg_{P}(T)$ and $\ppg_{E}(T)$ for the propagation of $T$ with respect to the $P_{r,n}$-module structure and the $E_n$-module structure, respectively. Also denote the Hilbert spaces
\[
\H_r:=\bigoplus_{n=1}^\infty \H_{r,n} \quad \mbox{and} \quad \H_{r,E}:=\bigoplus_{n=1}^\infty \H_{r,n,E},
\]
and use these modules to build the equivariant Roe algebras $C^*(\H_r)^\Gamma$ and $C^*(\H_{r,E})^\Gamma$ of $P_r:=\bigsqcup_{n=1}^\infty P_{r,n}$, and the equivariant localisation algebras $C^*_L(\H_r)^\Gamma$ and $C^*_L(\H_{r,E})^\Gamma$ of $P_r$.

For $r\leq s$, the canonical inclusion $Z_r \to Z_s$ induces isometric inclusions of Hilbert spaces
\[
\H_{r,n} \longrightarrow \H_{s,n} \quad \mbox{and} \quad \H_{r,n,E} \longrightarrow \H_{s,n,E},
\]
and further implies inclusions of $C^*$-algebras
\[
C^*(\H_{r,n}) \longrightarrow C^*(\H_{s,n}), \quad C^*(\H_{r,n,E}) \longrightarrow C^*(\H_{s,n,E})
\]
and
\[
C^*_L(\H_{r,n}) \longrightarrow C^*_L(\H_{s,n}), \quad C^*_L(\H_{r,n,E}) \longrightarrow C^*_L(\H_{s,n,E}).
\]

For $r>0$ and $n\in \NN$, also note that although $\{B_{x,r,n}:x\in X_n\}$ does not cover $P_{r,n}$ we still have the following decomposition of Hilbert spaces:
\[
\H_{r,n} = \bigoplus_{x\in X_n} \chi_{B_{x,r,n}}\H_{r,n} \quad \mbox{and} \quad \H_{r,n,E}=\bigoplus_{x\in X_n} \chi_{B_{x,r,n}}\H_{r,n,E}
\]
since $\{B_{x,r,n}:x\in X_n\}$ covers $Z_{r,n}$. Write
\[
\H_{x,r,n}:=\chi_{B_{x,r,n}}\H_{r,n} \quad \mbox{and} \quad \H_{x,r,n,E}=\chi_{B_{x,r,n}}\H_{r,n,E}.
\]
We can represent a bounded linear operator $T$ on $\H_{r,n}$ (respectively, $\H_{r,n,E}$) as an $X_n$-by-$X_n$ matrix $(T_{x,y})_{x,y\in X_n}$, where each $T_{x,y}$ is a bounded linear operator $\H_{y,r,n} \to \H_{x,r,n}$ (respectively, $\H_{y,r,n,E} \to \H_{x,r,n,E}$). Moreover, $T$ is $\Gamma$-invariant if and only if the following diagram commutes:
\[
\xymatrix{
			\H_{y,r,n} \ar[r]^{T_{x,y}} \ar[d]_{U_\gamma} & \H_{x,r,n}  \ar[d]^{U_\gamma} \\
			\H_{\gamma y,r,n} \ar[r]^{T_{\gamma x, \gamma y}} & \H_{\gamma x,r,n}
	}
\]
for any $\gamma \in \Gamma$ and $x,y\in X_n$, \emph{i.e.}, $T_{\gamma x, \gamma y} = U_\gamma T_{x,y} U_{\gamma}^\ast =: \gamma \cdot T_{x,y}$.

It was pointed out in \cite[Inequality (12.11)]{willett2020higher} that for any $r>0$, there exists a proper non-decreasing function $\varphi_r: [0,\infty) \to [0,\infty)$ such that for any $T \in \CC[\H_{r,n,E}]^\Gamma$ we have:
\begin{equation}\label{EQ:relations between ppgs}
\ppg_P(T) -2 \leq \sup\{d(x,y): T_{x,y} \neq 0 \mbox{~where~} x,y\in X_n \} \leq \varphi_r(\ppg_P(T)).
\end{equation}
We also record the following elementary result (see, \emph{e.g.}, \cite[Lemma 12.2.4]{willett2020higher}) for later use:

\begin{lem}\label{lem:norm control of finite ppg op}
For $s,r \geq 0$, there exists an $N \in \NN$ such that for any $n\in \NN$ and any bounded operator $T=(T_{x,y})_{x,y\in X}$ on $\H_{r,n,E}$ with $P$-propagation at most $s$, we have
\[
\|T\| \leq N \cdot \sup_{x,y\in X} \|T_{x,y}\|.
\]
\end{lem}

To introduce the twisted algebras, we recall the following construction from \cite[Definition 12.3.1]{willett2020higher}:

\begin{defn}\label{defn:V-construction}
	Given $r>0$, $n\in \NN$ and $T \in \B(\L_{E_n}^2)$, define a bounded linear operator $T^V$ on $\H_{r,n,E_n}=\ell^2(Z_{r,n})\otimes \HH\otimes \ell^2(\Gamma) \otimes \L^2_{E_n}$ by the formula
	\[
	T^V: \delta_z \otimes \xi\otimes \delta_g \otimes u\mapsto \delta_z\otimes \xi\otimes \delta_g \otimes V_{f_{r,n}(z)}TV_{f_{r,n}(z)}^*u
	\]
	for $z\in Z_{r,n}$, $\xi\in \HH$, $g\in \Gamma$ and $u\in \L^2_{E_n}$, where $V_{f_{r,n}(z)}$ is the translation operator defined in Section \ref{ssec:The Bott-Dirac operators on Euclidean spaces}.
\end{defn}

Writing in the matrix representation, we have
\[
T^V_{x,y}=
\begin{cases}
~\Id_{\H_{x,r,n}} \otimes V_{f_{r,n}(x)}TV_{f_{r,n}(x)}^*, & y=x; \\
~0, & \mbox{otherwise}.
\end{cases}
\]
Since $f_{r,n}(\gamma z) = f_{r,n}(z)$ for any $z\in Z_{r,n}$ and $\gamma \in \Gamma$, it is easy to see that $\gamma \cdot T_{x,y} = T_{\gamma x, \gamma y}$. In other words, $T^V$ is invariant. Hence we obtain:

\begin{lem}\label{lem: T^V is invariant}
For any $T \in \B(\L_{E_n}^2)$, we have $T^V \in \CC[\H_{r,n,E}]^\Gamma$.
\end{lem}

Now we introduce the notion of twisted algebras, which is slightly different from \cite[Section 12.6]{willett2020higher}.

\begin{defn}\label{defn:twisted Roe}
Fix an $r>0$. Denote $\prod_{n\in \NN} C_b([1,\infty), C^*(\H_{r,n,E})^\Gamma)$ the product $C^*$-algebra of all bounded continuous functions from $[1,\infty)$ to $C^*(\H_{r,n,E})^\Gamma$ with supremum norm. Write elements of this $C^*$-algebra as a collection $(T_{n,s})_{n\in \NN,s\in [1,\infty)}$ for $T_{n,s}=(T_{n,s,x,y})_{x,y\in X_n} \in C^*(\H_{r,n,E})^\Gamma$, whose norm is
\[
\|(T_{n,s})\|=\sup_{n\in \NN,s\in [1,\infty)}\|T_{n,s}\|_{\B(\H_{r,n,E})}.
\]
Let $\AA^r(X;E)^\Gamma$ denote the $*$-subalgebra of $\prod_{n\in \NN}C_b([1,\infty),C^*(\H_{r,n,E})^\Gamma)$ consisting of elements satisfying the following conditions:
\begin{enumerate}
		\item $\sup\limits_{s\in [1,\infty),n\in \NN}\ppg_{P}(T_{n,s})<\infty$;\\[0.1cm]
		\item for each $n\in \NN$, $\lim\limits_{s\rightarrow \infty}\ppg_{E}(T_{n,s})=0$;\\[0.1cm]
		\item $\lim\limits_{R\rightarrow \infty}\sup\limits_{s\in [1,\infty),n\in \NN}\|\chi_{0,R}^VT_{n,s}-T_{n,s}\|=\lim\limits_{R\rightarrow \infty}\sup\limits_{s\in [1,\infty),n\in \NN}\|T_{n,s}\chi_{0,R}^V-T_{n,s}\|=0$;\\[0.1cm]
		\item for each $n\in \NN$ and $x,y\in X_n$, the map $s \mapsto T_{n,s,x,y}$ belongs to the subalgebra $\K(\H_{y,r,n},\H_{x,r,n}) \otimes C_b([1,\infty), \K(\L^2_{E_n}))$ of $C_b([1,\infty), \K(\H_{y,r,n},\H_{x,r,n}) \otimes\K(\L^2_{E_n}))$.
\end{enumerate}
	The \emph{equivariant twisted Roe algebra} $A^r(X;E)^\Gamma$ of $X$ is defined to be the norm-closure of $\AA^r(X;E)^\Gamma$ in $\prod_{n\in \NN}C_b([1,\infty),C^*(\H_{r,n,E})^\Gamma)$.
\end{defn}

\begin{defn}\label{defn:twisted localisation}
Define $\AA^r_L(X;E)^\Gamma$ to be the collection of uniformly continuous bounded functions $T:[1,\infty)\to \AA^r(X;E)^\Gamma$ satisfying the following: writing $T=(T_t)_{t\in [1,\infty)} = (T_{t,n,s})_{t,s\in [1,\infty), n\in \NN}$ then
\begin{enumerate}
 \item $\lim\limits_{t\to \infty} \sup\limits_{s\in [1,\infty), n\in \NN} \ppg_P(T_{t,n,s}) = 0$;\\[0.1cm]
 \item $\lim\limits_{R\rightarrow \infty}\sup\limits_{t,s\in [1,\infty),n\in \NN}\|\chi_{0,R}^VT_{t,n,s}-T_{t,n,s}\|=\lim\limits_{R\rightarrow \infty}\sup\limits_{t,s\in [1,\infty),n\in \NN}\|T_{t,n,s}\chi_{0,R}^V-T_{t,n,s}\|=0$.
\end{enumerate}
The \emph{equivariant twisted localisation algebra} $A^r_L(X;E)^\Gamma$ of $X$ is defined to be the completion of $\AA^r_L(X;E)^\Gamma$ for the norm $\|(T_t)\|:=\sup_t \|T_t\|_{A^r(X;E)^\Gamma}$.
\end{defn}

\begin{rem}\label{rem:deff between twisted algebras}
Readers might already notice that the above definition is slightly different from \cite[Definition 12.6.2]{willett2020higher}. More precisely, we weaken condition (2) to level-wise convergence and restrict the living space of each matrix entry (see condition (4) above) for later use. On the other hand, the following lemma shows that the original condition (4) in \cite[Definition 12.6.2]{willett2020higher} can be recovered.
\end{rem}

\begin{lem}\label{lem:condition (4) recovery}
Given $n\in \NN$, $r>0$ and $T \in C^*(\H_{r,n,E})^\Gamma$, we have
\[
\lim_{i\in I}\|p_i^VT-T\|=\lim_{i\in I}\|Tp_i^V-T\|=0
\]
where $\{p_i\}_{i\in I}$ is the net of finite rank projections on $\L_{E_n}^2$.
\end{lem}

\begin{proof}
We fix an operator $T \in \CC[\H_{r,n,E}]^\Gamma$ with $P$-propagation at most $R'$, and a point $x_0 \in X_n$. For $x,y\in X_n$ with $T_{x,y} \neq 0$, we have $d(x,y) \leq \varphi_r(R')=:R$ where the function $\varphi_r$ comes from Inequality (\ref{EQ:relations between ppgs}). Let $N$ be the number from Lemma \ref{lem:norm control of finite ppg op} for the parameters $R'$ and $r$, \emph{i.e.}, for any bounded operator $\tilde{T}=(\tilde{T}_{x,y})_{x,y\in X}$ on $\H_{r,n,E}$ with $P$-propagation at most $R'$, then
\[
\|\tilde{T}\| \leq N \cdot \sup_{x,y\in X} \|\tilde{T}_{x,y}\|.
\]
Recall that the $\Gamma$-action on $X_n$ is cocompact, hence there exists an $S>0$ such that $\Gamma \cdot B(x_0,S) = X_n$.

We consider a finite set of operators:
\[
\mathfrak{F}:=\{T_{x,y}: x\in B(x_0,S) \mbox{~and~} y\in B(x,R)\}.
\]
Since $T$ is locally compact, each $T_{x,y}$ is a compact operator from $\H_{y,r,n,E}$ to $\H_{x,r,n,E}$. Note that for any projection $p \in \B(\L^2_{E_n})$ and $x,y\in X_n$, we have:
\[
(p^V T - T)_{x,y} = \big( \Id_{\H_{x,r,n}} \otimes V_{f_{r,n}(x)} p V_{f_{r,n}(x)}^* - \Id_{\H_{x,r,n,E}} \big) \cdot T_{x,y},
\]
and the net
\[
\big\{\Id_{\H_{x,r,n}} \otimes V_{f_{r,n}(x)} p V_{f_{r,n}(x)}^* - \Id_{\H_{x,r,n,E}} : p \mbox{~is~a~finite~rank~projection~in~} \B(\L^2_{E_n})\big\}
\]
converges to $0$ in the strong operator topology. Hence for any $\varepsilon>0$ there exists a finite rank projection $p_\varepsilon \in \B(\L^2_{E_n})$ such that for any finite rank projection $p\geq p_\varepsilon$ and $T_{x,y} \in \mathfrak{F}$, we have
\[
\|(p^V T -T)_{x,y}\| < \frac{\varepsilon}{N}.
\]

Now for any $x', y'\in X_n$ with $d(x',y') \leq R$, there exists $\gamma\in \Gamma$ such that $\gamma x' \in B(x_0, S)$. Letting $x=\gamma x'$ and $y=\gamma y'$, we have $d(x,y) \leq R$ and hence $T_{x,y} \in \mathfrak{F}$. Moreover, we obtain:
\[
\|(p^V T -T)_{x',y'}\| = \| (p^V T -T)_{\gamma^{-1} x, \gamma^{-1} y} \| = \|\gamma^{-1} \cdot (p^V T - T) _{x,y}\| =  \|(p^V T - T) _{x,y}\| < \frac{\varepsilon}{N}.
\]
Note that the $P$-propagation of $p^V T - T$ is again at most $R'$. Hence from Lemma \ref{lem:norm control of finite ppg op} we obtain:
\[
\|p^V T - T\| \leq N \cdot \sup_{x,y \in X_n: d(x,y) \leq R} \|(p^V T -T)_{x,y}\| < N \cdot \frac{\varepsilon}{N} = \varepsilon.
\]
Therefore, we obtain that $\lim_{i\in I}\|p_i^VT-T\|=0$. Finally taking adjoints, we obtain $\lim_{i\in I}\|Tp_i^V-T\|=0$ as well and conclude the proof.
\end{proof}

Finally, we introduce the following operators:

\begin{defn}[{\cite[Section 12.3 and 12.6]{willett2020higher}}]\label{defn:F}
Fix an $r>0$. For each $n\in \NN$, $s\in [1,\infty)$ and $x\in E_n$, Definition \ref{defn:F_{s,x}} provides a bounded linear operator $F_{s,x} \in \B(\L_{E_n}^2)$, also denoted by $F_{n,s,x}$. Applying Definition \ref{defn:V-construction}, we obtain an operator
\[
F_{n,s}:=F^V_{n,s+2d_n,0} \in \B(\H_{r,n,E})^\Gamma
\]
where $d_n$ is the dimension of $E_n$. Let $F_s:=(F_{n,s})_{n\in \NN}$ be the block diagonal operator in $\prod_n \B(\H_{r,n,E})^\Gamma \subseteq \B(\H_{r,E})^\Gamma$. Finally, we define $F$ to be an element in $\prod_n \B(L^2([1,\infty);\H_{r,n,E})) \subseteq \B(L^2 ([1,\infty);\H_{r,E}))$ defined by $(F(u))(s):=F_su(s)$.
	
Similarly given $\varepsilon>0$, let $\Psi$ be a function as in Proposition~\ref{prop:Psi function} and set $F^\Psi_{n,s,x}:=\Psi(B_{n,s,x})$. Let $F^\Psi_s$ be the bounded diagonal operator on $\H_{r,E}$ defined by:
\[
F^\Psi_s:=(F^\Psi_{n,s})_{n\in \NN} \quad \mbox{where} \quad F^\Psi_{n,s}:= (F^\Psi_{n,s+2d_n,0})^V = \Psi(B_{n,s+2d_n,0})^V \in \B(\H_{r,n,E})^\Gamma.
\]
We also define $F^\Psi$ to be the element in $\prod_n \B(L^2([1,\infty);\H_{r,n,E})) \subseteq \B(L^2 ([1,\infty);\H_{r,E}))$ by $(F^\Psi(u))(s):=F^\Psi_su(s)$.
\end{defn}

\section{The index maps}\label{sec:index map}

In this section, we construct equivariant index maps (with the same notation as in Section \ref{ssec:equiv. twisted algebras}):
\begin{equation}\label{EQ: Ind_F}
\Ind_F: K_* \big( C^*(\H_r)^\Gamma \cap \prod_{n=1}^\infty C^*(\H_{r,n})^\Gamma \big) \rightarrow K_*(A^r(X;E)^\Gamma)
\end{equation}
and
\begin{equation}\label{EQ: Ind_{F_L}}
\Ind_{F_L}: K_* \big( C^*_L(\H_r)^\Gamma \cap \prod_{n=1}^\infty C^*_L(\H_{r,n})^\Gamma \big) \rightarrow K_*(A^r_L(X;E)^\Gamma),
\end{equation}
where $F$ is the operator from Definition \ref{defn:F}. We use these maps to transfer $K$-theoretic information from equivariant Roe and localisation algebras to their twisted counterparts, which allow us to prove Theorem \ref{thm:main result equiv. CBC} via local isomorphisms. This approach is mainly based on \cite[Secition 12.3 and 12.6]{willett2020higher} in the case of coarse embedding, while several changes are needed to involve group actions.

The main result of this section is the following:

\begin{prop}\label{prop:index maps}
Fix an $r>0$. With notation as in Section \ref{ssec:equiv. twisted algebras}, for each $s\in [1,\infty)$ the composition
\[
K_* \big( C^*(\H_r)^\Gamma \cap \prod_{n=1}^\infty C^*(\H_{r,n})^\Gamma \big) \stackrel{\Ind_F}{\longrightarrow} K_*(A^r(X;E)^\Gamma) \stackrel{\iota^s_*}{\longrightarrow} K_*\big( C^*(\H_{r,E})^\Gamma \cap \prod_{n=1}^\infty C^*(\H_{r,n,E})^\Gamma \big)
\]
is an isomorphism, where $\iota^s:A^r(X;E)^\Gamma \rightarrow C^*(\H_{r,E})^\Gamma \cap \prod_{n=1}^\infty C^*(\H_{r,n,E})^\Gamma$ is the evaluation map at $s$. The analogous statement holds for the equivariant localisation algebras. Moreover, we have the following commutative diagram:
\[
\xymatrix{
	K_* \big( C^*_L(\H_r)^\Gamma \cap \prod_{n=1}^\infty C^*_L(\H_{r,n})^\Gamma \big)  \ar[r]^>>>>>{\Ind_{F_L}}  \ar[d] & K_\ast ( A^r_L(X;E)^\Gamma ) \ar[r]^>>>>{\iota^s_*}  \ar[d] & K_\ast \big( C^*_L(\H_{r,E})^\Gamma \cap \prod_{n=1}^\infty C^*_L(\H_{r,n,E})^\Gamma \big)  \ar[d]\\
	K_* \big( C^*(\H_r)^\Gamma \cap \prod_{n=1}^\infty C^*(\H_{r,n})^\Gamma \big)  \ar[r]^>>>>>{\Ind_F}  & K_\ast ( A^r(X;E)^\Gamma ) \ar[r]^>>>>{\iota^s_*}  & K_\ast \big( C^*(\H_{r,E})^\Gamma \cap \prod_{n=1}^\infty C^*(\H_{r,n,E})^\Gamma \big)
	}
\]
where all vertical lines are induced by the evaluation-at-one maps.
\end{prop}

We follow the same notation as in Section \ref{ssec:equiv. twisted algebras}. Recall that $(X,d)$ is a discrete metric space of bounded geometry and having the form $X=\bigsqcup_{n=1}^\infty X_n$ with $d(X_n, X_m) \to \infty$ as $n+m \to \infty~ (n\neq m)$ such that $\Gamma$ acts on each $X_n$ properly and cocompactly by isometries. For each $n\in \NN$, we have a map $f_n: X_n \to E_n$ coming from the uniformly coarse embedding of $X_n/\Gamma$ into some Euclidean space $E_n$ of even dimension $d_n$.

To construct the index map $\Ind_F$, we need a series of lemmas. Let us fix an $r>0$ throughout the rest of this section.

\begin{lem}\label{lem:F multiplier}
The operator $F$ is a self-adjoint, norm one, odd operator in the multiplier algebra of $A^r(X;E)^\Gamma$.
\end{lem}

\begin{proof}
	The operator $F$ is self-adjoint, norm one and odd since each $F_{n,s,x}$ is. Given $\varepsilon>0$, let $\Psi:\RR \to [-1,1]$ be a function as in Proposition \ref{prop:Psi function} for this $\varepsilon$. Then Proposition \ref{prop:Psi function}(1) implies:
	\[
	\|F - F^\Psi\| \leq \sup_{n\in \NN, s\in [1,\infty)} \|F^V_{n,s,0} - \Psi(B_{n,s,0})^V\| \leq \sup_{n\in \NN, s\in [1,\infty)} \sup_{x\in X_n} \|F_{n,s,f_n(x)} - \Psi(B_{n,s,f_n(x)})\| \leq \varepsilon.
	\]
Hence it suffices to show that $(T_{n,s}) F^\Psi=(T_{n,s}F^\Psi_{n,s})$ belongs to $\AA^r(X;E)^\Gamma$ for any $(T_{n,s}) \in \AA^r(X;E)^\Gamma$.

We first claim that for each $n\in \NN$ the map $s\mapsto T_{n,s}F^\Psi_{n,s}$ is norm continuous. In fact, this can be proved using the same argument as in the proof of \cite[Lemma 12.3.6]{willett2020higher} thanks to Lemma \ref{lem:condition (4) recovery}. Here we provide a direct proof.

Fix $n\in \NN$ and $x_0 \in X_n$. Assume that each $T_{n,s}$ has $P$-propagation at most $R'$, and set $R:=\varphi_r(R')$ where the function $\varphi_r$ comes from Inequality (\ref{EQ:relations between ppgs}). Since the $\Gamma$-action on $X_n$ is cocompact, there exists an $S>0$ such that $\Gamma \cdot B(x_0,S) = X_n$.  For each $x,y\in X_n$, it follows from Proposition \ref{prop:Psi function}(6) and Lemma \ref{lem:12.3.5 from WY21} that the map $s\to T_{n,s,x,y}\cdot F^\Psi_{n,s+2d_n,f_n(y)} = (T_{n,s}F^\Psi_{n,s})_{x,y}$ is norm continuous. On the other hand, for any $x, y\in X_n$ with $d(x,y) \leq R$, there exists $\gamma\in \Gamma$ such that $\bar{x}:=\gamma x \in B(x_0, S)$ and hence $\bar{y}:=\gamma y \in B(\bar{x},R)$. Moreover, we have
\[
(T_{n,s}F^\Psi_{n,s})_{x,y} = (T_{n,s}F^\Psi_{n,s})_{\gamma^{-1} \bar{x}, \gamma^{-1} \bar{y}} = \gamma^{-1} \cdot (T_{n,s}F^\Psi_{n,s}) _{\bar{x},\bar{y}}.
\]
Note that the set $\{(\bar{x},\bar{y}): \bar{x}\in B(x_0, S) \mbox{~and~}\bar{y} \in B(\bar{x},R)\}$ is finite, hence the family $\{s \mapsto (T_{n,s}F^\Psi_{n,s})_{x,y}: x,y\in X_n\}$ is uniformly norm continuous. This concludes the claim due to Lemma \ref{lem:norm control of finite ppg op}.

For conditions (1)-(3) in Definition \ref{defn:twisted Roe}, it suffices to note that the $P$-propagation of $F^\Psi_{n,s}$ is $0$ while the $E$-propagation of $F^\Psi_{n,s}$ is uniformly bounded (in both $s$ and $n$) and (uniformly) tends to $0$ (in $n$) as $s \to \infty$ by Proposition \ref{prop:Psi function}(2). Finally for condition (4), note that for each $n\in \NN$ and $x,y\in X_n$ we have
\[
s \mapsto (T_{n,s}F^\Psi_{n,s})_{x,y}=T_{n,s,x,y} \cdot \big(\Id_{\H_{y,r,n}} \otimes F^\Psi_{n,s+2d_n,f_n(y)} \big)
\]
Since the map $s\mapsto T_{n,s,x,y}$ belongs to $\K(\H_{y,r,n},\H_{x,r,n}) \otimes C_b([1,\infty), \K(\L^2_{E_n}))$, it follows from Proposition \ref{prop:Psi function}(6) and Lemma \ref{lem:12.3.5 from WY21} again that the map $s \mapsto (T_{n,s}F^\Psi_{n,s})_{x,y}$ belongs to $\K(\H_{y,r,n},\H_{x,r,n}) \otimes C_b([1,\infty), \K(\L^2_{E_n}))$ as well. Hence we conclude the proof.
\end{proof}

\begin{lem}\label{lem:multiplier}
Considered as represented on $L^2([1,\infty)) \otimes\H_{r,E}$ via the amplification of identity, $C^*(\H_r)^\Gamma \cap \prod_{n=1}^\infty C^*(\H_{r,n})^\Gamma$ is a subalgebra of the multiplier algebra of $A^r(X;E)^\Gamma$.
\end{lem}

\begin{proof}
It suffices to show that $(S_nT_{n,s}) \in \AA^r(X;E)^\Gamma$ for any $(T_{n,s}) \in \AA^r(X;E)^\Gamma$ and $(S_n)\in \CC[\H_r]^\Gamma \cap \prod_{n=1}^\infty C^*(\H_{r,n})^\Gamma$. It is clear that the map $s\mapsto S_nT_{n,s}$ is norm-continuous and bounded for each $n\in \NN$.

For conditions (1) and (2) in Definition \ref{defn:twisted Roe}, it suffices to note that $S_n$ has uniformly finite $P$-propagation (in $n$) and $E$-propagation $0$. Condition (4) follows from the fact that $S_n$ is constant in $s$.  Finally for condition (3), it is clear that
\[
\lim\limits_{R\rightarrow \infty}\sup\limits_{s\in [1,\infty),n\in \NN}\|S_nT_{n,s}\chi_{0,R}^V-S_nT_{n,s}\|=0.
\]
On the other hand, set
\[
R_0:= \sup_n\big\{|f_n(x)-f_n(y)|: x,y\in X_n \mbox{~with~} d(x,y) \leq \ppg_P(S_n)\big\}.
\]
This is clear that $R_0$ is finite. For any $n\in \NN$, $x,y\in X_n$ and $R\geq R_0$, we have:
\begin{eqnarray*}
(\chi_{0,R}^V \cdot S_n  \cdot \chi_{0,R-R_0}^V)_{x,y} &=& \chi_{f_n(x),R} \cdot S_{n,x,y} \cdot \chi_{f_n(y), R-R_0} = S_{n,x,y} \cdot \chi_{f_n(x),R} \cdot \chi_{f_n(y), R-R_0} \\
&=& S_{n,x,y} \cdot \chi_{f_n(y), R-R_0} = (S_n  \cdot \chi_{0,R-R_0}^V)_{x,y}.
\end{eqnarray*}
In other words, we obtain
\[
\chi_{0,R}^V \cdot S_n  \cdot \chi_{0,R-R_0}^V = S_n  \cdot \chi_{0,R-R_0}^V,
\]
which implies that
\[
\|\chi_{0,R}^V  S_n  T_{n,s} - S_n T_{n,s}\| \leq \|\chi_{0,R}^V\cdot S_n \cdot (\chi_{0,R-R_0}^V T_{n,s}-T_{n,s})\| + \|S_n \cdot (\chi_{0,R-R_0}^V T_{n,s}-T_{n,s})\|.
\]
This tends to $0$ as $R \to \infty$ uniformly (in $s$ and $n$) by assumption.
\end{proof}

Regarding $C^*(\H_r)^\Gamma \cap \prod_{n=1}^\infty C^*(\H_{r,n})^\Gamma$ as a subalgebra in $\B(L^2([1,\infty)) \otimes\H_{r,E})$ as in Lemma \ref{lem:multiplier}, we have the following:

\begin{lem}\label{lem:commutator}
For any $(S_n) \in C^*(\H_r)^\Gamma \cap \prod_{n=1}^\infty C^*(\H_{r,n})^\Gamma$, we have $[(S_n),F] \in A^r(X;E)^\Gamma$.
\end{lem}

\begin{proof}
Fix an $(S_n) \in \CC[\H_r]^\Gamma \cap \prod_{n=1}^\infty C^*(\H_{r,n})^\Gamma$. From Proposition~\ref{prop:Psi function}(1), it suffices to show that
\[
\left(s\mapsto [S_n,F_{n,s}^\Psi]\right)_n
\]
belongs to $\AA^r(X;E)^\Gamma$ for any $\Psi$ as in Proposition~\ref{prop:Psi function}.

First note that for any $n\in \NN$ and $x,y\in X_n$, we have
\[
([S_n,F_{n,s}^\Psi])_{x,y} = S_{n,x,y} \otimes \big( F^\Psi_{n,s+2d_n,f_n(y)} - F^\Psi_{n,s+2d_n,f_n(x)} \big).
\]
Hence for each $n\in \NN$, the function $s\mapsto [S_n,F_{n,s}^\Psi]$ is bounded and norm continuous by Proposition \ref{prop:Psi function}(8) and Lemma \ref{lem:norm control of finite ppg op}. Also condition (4) in Definition \ref{defn:twisted Roe} holds by Proposition \ref{prop:Psi function}(4) and (8). Moreover, condition (1) and (2) holds since $F_{n,s}^\Psi$ has $P$-propagation $0$, and $S_n$ has $E$-propagation $0$ together with Proposition \ref{prop:Psi function}(2).

Finally for condition (3), note that
\[
(\chi_{0,R}^V \cdot [S_n,F_{n,s}^\Psi])_{x,y} = S_{n,x,y} \otimes \chi_{f_n(x),R}\big( F^\Psi_{n,s+2d_n,f_n(y)} - F^\Psi_{n,s+2d_n,f_n(x)} \big).
\]
Hence condition (3) follows from Proposition \ref{prop:Psi function}(10), finite $P$-propagation of $(S_n)$ and Lemma \ref{lem:norm control of finite ppg op}.
\end{proof}

\begin{lem}\label{lem:projection}
For any projection $(p_n) \in C^*(\H_r)^\Gamma \cap \prod_{n=1}^\infty C^*(\H_{r,n})^\Gamma$, the element
\[
\left( s\mapsto ( p_n F_{n,s} p_n )^2 - p_n \right)_{n\in \NN}
\]
is in $(p_n) A^r(X;E)^\Gamma (p_n)$.
\end{lem}

\begin{proof}
From Lemma~\ref{lem:commutator}, it suffices to show that $\left( s\mapsto p_n (F_{n,s})^2 - p_n \right)_n$ is in $A^r(X;E)^\Gamma$. Moreover, we only need to show that
\[
\left( s\mapsto q_n (F_{n,s}^\Psi)^2 - q_n \right)_n
\]
is in $\AA^r(X;E)^\Gamma$ for any $\Psi$ as in Proposition~\ref{prop:Psi function}, where $(q_n)$ is a finite propagation approximation to $(p_n)$. For each $n\in \NN$, it follows from Proposition \ref{prop:Psi function}(7) and Lemma \ref{lem:norm control of finite ppg op} that the function $s \mapsto q_n(F_{n,s}^\Psi)^2 - q_n$ is bounded and continuous. It is routine to check condition (1)-(4) in Definition \ref{defn:twisted Roe} for this map using Proposition \ref{prop:Psi function}(3) and (9), the finite propagation of $(q_n)$ together with Lemma \ref{lem:norm control of finite ppg op}.
\end{proof}

Now we are in the position to construct the index map $\Ind_F$ in (\ref{EQ: Ind_F}). It follows from a standard construction in $K$-theories (see, \emph{e.g.}, \cite[Definitoin 2.8.5]{willett2020higher}):

\begin{defn}\label{defn:index definition}
	Let $\H=\H^{+}\oplus \H^{-}$ be a graded Hilbert space with grading operator $U$ (\emph{i.e.}, $U$ is a self-adjoint unitary operator in $\B(\H)$ such that $\H^{\pm}$ coincides with the $(\pm 1)$-eigenspace of $U$), and $A$ be a $C^*$-subalgebra of $\B(\H)$ such that $U$ is in the multiplier algebra of $A$. Let $F\in \B(\H)$ be an odd operator of the form
	\[ F =
	\begin{pmatrix}
	0&V \\
	W&0
	\end{pmatrix}
	\]
	for some operators $V: \H^{-} \rightarrow \H^{+}$ and $W: \H^{+} \rightarrow \H^{-} $. Suppose $F$ satisfies:
	\begin{itemize}
		\item $F$ is in the multiplier algebra of $A$;
		\item $F^2-1$ is in $A$.
	\end{itemize}
	Then we define the \emph{index class} $\Ind[F] \in K_0(A)$ of $F$ to be
	\[
	\Ind[F] :=
	\begin{bmatrix}
	(1-VW)^2&V(1-WV) \\
	W(2-VW)(1-VW)&WV(2-WV)
	\end{bmatrix}
	-
	\begin{bmatrix}
	0&0 \\
	0&1
	\end{bmatrix}.
	\]
\end{defn}

Combining Lemma \ref{lem:F multiplier}$\sim$Lemma \ref{lem:projection}, we obtain that for each projection $(p_n) \in C^*(\H_r)^\Gamma \cap \prod_{n=1}^\infty C^*(\H_{r,n})^\Gamma$ the operator $( (p_n)F(p_n) )$ is an odd self-adjoint operator on the graded Hilbert space $\bigoplus_n p_n(L^2 ( [1,\infty), \H_{r,n,E} ) )$ satisfying:
\begin{itemize}
	\item $((p_n)F(p_n)) $ is in the multiplier algebra of $(p_n) A^r(X;E)^\Gamma (p_n)$;
	\item $( (p_n)F(p_n) )^2 - (p_n) $ is in $(p_n) A^r(X;E)^\Gamma (p_n)$.
\end{itemize}
Hence Definition~\ref{defn:index definition} produces an index class in $K_0((p_n) A^r(X;E)^\Gamma (p_n))$. Composing with the $K_0$-map induced by the inclusion $(p_n) A^r(X;E)^\Gamma (p_n) \hookrightarrow A^r(X;E)^\Gamma$, we get an element in $K_0( A^r(X;E)^\Gamma )$, denoted by $\Ind_F[(p_n)]$. An argument analogous to \cite[Lemma~12.3.11]{willett2020higher} provides a well-defined homomorphism
\[
\Ind_F: K_0 \big( C^*(\H_r)^\Gamma \cap \prod_{n=1}^\infty C^*(\H_{r,n})^\Gamma \big) \longrightarrow K_0( A^r(X;E)^\Gamma ).
\]

Passing to suspension and applying the above construction pointwise, we can also define $\Ind_F$ on the level of $K_1$-groups. Similarly, we can also deal with the localisation case by applying the above construction pointwise in the parameter $t$. Consequently, we obtain the following:

\begin{prop}\label{prop:index maps defn}
The process above provides well-defined homomorphisms:
\[
\Ind_F: K_\ast \big( C^*(\H_r)^\Gamma \cap \prod_{n=1}^\infty C^*(\H_{r,n})^\Gamma \big) \rightarrow K_\ast( A^r(X;E)^\Gamma )
\]
and
\[
\Ind_{F_L}: K_\ast \big( C^*_L(\H_r)^\Gamma \cap \prod_{n=1}^\infty C^*_L(\H_{r,n})^\Gamma \big) \rightarrow K_\ast( A^r_L(X;E)^\Gamma )
\]
for $\ast=0,1$, which are called the \emph{equivariant index maps} associated to $F$.
\end{prop}

Finally, we prove Proposition \ref{prop:index maps}.

\begin{proof}[Proof of Proposition \ref{prop:index maps}]
The proof follows the outline of \cite[Proposition~12.3.13 and Proposition~12.6.3]{willett2020higher}. We will only focus on the case of equivariant (twisted) Roe algebras, while the localisation case follows from the same argument applied pointwise. Throughout the proof, we fix an $s \in [1,\infty)$.

For each $n \in \NN$, we define a map $\kappa_n: E_n\rightarrow E_n$ by
\[
\kappa_n (x)=
\begin{cases}
~\frac{x}{|x|}(|x|-1), & \mbox{if } |x|\ge 1;  \\
~0, & \mbox{otherwise}
\end{cases}
\]
and a sequence of maps:
\[
F_n^{(k)}: \H_{r,n,E} \rightarrow \H_{r,n,E}, \quad \delta_z\otimes \xi\otimes \delta_g \otimes u \mapsto \delta_z\otimes \xi\otimes \delta_g \otimes F_{n,s+2d_n,\kappa_n^k(f_{r,n}(z))}u
\]
for $k \in \NN \cup \{\infty\}$ and $z\in Z_{r,n}$, where $\kappa_n^{\infty}(v):=0$ for all $v\in E_n$. Denote
\[
F^{(k)}:=(F_n^{(k)})_n \in \prod_n \B(\H_{r,n,E})^\Gamma \subseteq \B(\H_{r,E})^\Gamma.
\]
We note that for the fixed $s$ at the beginning, we have that $F^{(0)}=F_s$ from Definition~\ref{defn:F} and $F^{(\infty)} = (\Id_{\H_{r,n}}\otimes F_{n,s+2d_n,0})_n$.
Quite analogous to the construction in Proposition~\ref{prop:index maps defn}, we obtain:
\[
\Ind_{F^{(k)}}: K_*\big( C^*(\H_r)^\Gamma \cap \prod_{n=1}^\infty C^*(\H_{r,n})^\Gamma \big) \rightarrow K_*\big( C^*(\H_{r,E})^\Gamma \cap \prod_{n=1}^\infty C^*(\H_{r,n,E})^\Gamma \big)
\]
for each $k\in \NN \cup \{\infty\}$. It is clear that $\Ind_{F^{(0)}}= \iota^s_* \circ \Ind_F$. In the following, we will show that $\Ind_{F^{(0)}}= \Ind_{F^{(\infty)}}$ and that $\Ind_{F^{(\infty)}}$ is an isomorphism to conclude the proof. We will only focus on the case of $K_0$, and the case of $K_1$ can be handled using a standard suspension argument.

First we show that $\Ind_{F^{(0)}}= \Ind_{F^{(\infty)}}$. Let
\[
\H_{r,n,E,\infty}:=(\H_{r,n,E})^{\oplus \infty} = \H_{r,n} \otimes (\L^2_{E_n})^{\oplus \infty} ~\mbox{for~all~}n\in \NN \quad \mbox{and} \quad \H_{r,E,\infty}:=\bigoplus_{n}\H_{r,n,E,\infty}.
\]
Denote $C^*(\H_{r,E,\infty})^\Gamma$ the corresponding equivariant Roe algebra. It is known that the ``top-left corner inclusion'':
\[
\iota: C^*(\H_{r,E})^\Gamma \cap \prod_{n=1}^\infty C^*(\H_{r,n,E})^\Gamma \rightarrow C^*(\H_{r,E,\infty})^\Gamma \cap \prod_{n=1}^\infty C^*(\H_{r,n,E,\infty})^\Gamma,  \quad T\mapsto
	\begin{pmatrix}
	T &0&\cdots \\
	0&0& \\
	\vdots &  &\ddots
	\end{pmatrix}
\]
induces an isomorphism on $K$-theory. To show that $\Ind_{F^{(0)}}= \Ind_{F^{(\infty)}}$, it suffices to show that $\iota_\ast \circ\Ind_{F^{(0)}}= \iota_\ast \circ\Ind_{F^{(\infty)}}$.

Without loss of generality, it suffices to show that $\iota_\ast \circ\Ind_{F^{(0)}}[(p_n)]= \iota_\ast \circ\Ind_{F^{(\infty)}}[(p_n)]$ for any projection $(p_n)\in C^*(\H_{r,E})^\Gamma \cap \prod_{n=1}^\infty C^*(\H_{r,n,E})^\Gamma$. For each $k\in \NN \cup \{\infty\}$, it follows from Definition~\ref{defn:index definition} that the class $\Ind_{F^{(k)}}[(p_n)]$ can be represented by a concrete difference of projections, say
\[
[(p_n^{(k)})]-[(q_n)],
\]
where $(q_n)$ is independent of $k$ and we have
\[
(p_n^{(k)})-(q_n)\in M_2\left( (p_n) \cdot \big( C^*(\H_{r,E})^\Gamma \cap \prod_{n=1}^\infty C^*(\H_{r,n,E})^\Gamma \big) \cdot (p_n) \right).
\]
Now we consider the projections
\[
	\begin{pmatrix}
	(p_n^{(0)}) &0 &0 &\cdots \\
	0& (p_n^{(1)}) &0 & \\
	0& 0& (p_n^{(2)}) & \\
	\vdots & & &\ddots
	\end{pmatrix}
	\quad \mbox{and} \quad
	\begin{pmatrix}
	(p_n^{(\infty)}) &0 &0 &\cdots \\
	0& (p_n^{(\infty)}) &0 & \\
	0& 0& (p_n^{(\infty)}) & \\
	\vdots & & &\ddots
	\end{pmatrix}
\]
in the multiplier algebra of $M_2\big( C^*(\H_{r,E,\infty})^\Gamma \cap \prod_{n=1}^\infty C^*(\H_{r,n,E,\infty})^\Gamma \big)$. For any compact subset $K \subseteq P_r$ there exists $k\in \NN$ such that $\kappa^k(f_{r,n}(K \cap P_{r,n}))=0$ for any $n\in \NN$ (in fact there are only finitely many $n\in \NN$ satisfying $K \cap P_{r,n} \neq \emptyset$). Hence it is easy to see that the difference of these projections is in $M_2\big( C^*(\H_{r,E,\infty})^\Gamma \cap \prod_{n=1}^\infty C^*(\H_{r,n,E,\infty})^\Gamma \big)$. Therefore the formal difference
\[
	\begin{bmatrix}
	\begin{pmatrix}
	(p_n^{(0)}) &0 &0 &\cdots \\
	0& (p_n^{(1)}) &0 & \\
	0& 0& (p_n^{(2)}) & \\
	\vdots & & &\ddots
	\end{pmatrix}
	\end{bmatrix}
	-
	\begin{bmatrix}
	\begin{pmatrix}
	(p_n^{(\infty)}) &0 &0 &\cdots \\
	0& (p_n^{(\infty)}) &0 & \\
	0& 0& (p_n^{(\infty)}) & \\
	\vdots & & &\ddots
	\end{pmatrix}
	\end{bmatrix}
\]
defines a class in $K_0\big( C^*(\H_{r,E,\infty})^\Gamma \cap \prod_{n=1}^\infty C^*(\H_{r,n,E,\infty})^\Gamma \big)$, denoted by $a$.

On the other hand, there is a class $b\in K_0\big( C^*(\H_{r,E,\infty})^\Gamma \cap \prod_{n=1}^\infty C^*(\H_{r,n,E,\infty})^\Gamma \big)$ defined by the formal difference
\[
	\begin{bmatrix}
	\begin{pmatrix}
	(p_n^{(0)}) &0 &0 &\cdots \\
	0& 0 &0 & \\
	0& 0& 0 & \\
	\vdots & & &\ddots
	\end{pmatrix}
	\end{bmatrix}
	-
	\begin{bmatrix}
	\begin{pmatrix}
	(p_n^{(\infty)}) &0 &0 &\cdots \\
	0& 0 &0 & \\
	0& 0& 0 & \\
	\vdots & & &\ddots
	\end{pmatrix}
	\end{bmatrix}.
\]
We claim that $a+b=a$, whence $b=0$. This implies that $\Ind_{F^{(0)}}[(p_n)]= \Ind_{F^{(\infty)}}[(p_n)]$. In fact, for each $k \in \NN$ we consider the path $(F^{(k),r})_{r\in [0,1]} = \big( (F_n^{(k),r}) \big)_{r\in [0,1]}$ in $\B(\H_{r,E})^\Gamma$ defined by
\[
F_n^{(k),r}: \delta_z\otimes \xi\otimes \delta_g \otimes u \mapsto \delta_z\otimes \xi\otimes \delta_g \otimes F_{n,s+2d_n,(1-r)\kappa_n^k(f_{r,n}(z))+r\kappa_n^{k+1}(f_{r,n}(z))}u \quad \mbox{on each } \H_{r,n,E}.
\]
It follows from Proposition \ref{prop:Psi function}(5) that $\|F^{(k),r} - F^{(k),r'}\|\le 3|r-r'|$ for each $n\in \NN$. Hence $a$ is homotopic to
\[
	\begin{bmatrix}
	\begin{pmatrix}
	(p_n^{(1)}) &0 &0 &\cdots \\
	0& (p_n^{(2)}) &0 & \\
	0& 0& (p_n^{(3)}) & \\
	\vdots & & &\ddots
	\end{pmatrix}
	\end{bmatrix}
	-
	\begin{bmatrix}
	\begin{pmatrix}
	(p_n^{(\infty)}) &0 &0 &\cdots \\
	0& (p_n^{(\infty)}) &0 & \\
	0& 0& (p_n^{(\infty)}) & \\
	\vdots & & &\ddots
	\end{pmatrix}
	\end{bmatrix}.
\]
which implies that $a+b=a$.

Finally we explain that $\Ind_{F^{(\infty)}}$ is an isomorphism. For each $n\in \NN$, note that  $F^{(\infty)}_n = \Id_{\H_{r,n}}\otimes F_{n,s+2d_n,0}$. Let $p_{n,0}$ be the projection onto the kernel of $F_{n,s+2d_n,0}$, which is one-dimensional. Recall that $F_{n,s+2d_n,0}=g((s+2d_n)^{-1}D+C)$ where $g(x)=x(1+x^2)^{-1/2}$. Now consider the path $(F_{n,s+2d_n,0}^t)_{t\in[1,\infty]}$ defined by
\[
F_{n,s+2d_n,0}^t:=g\big(t\big(\frac{D}{s+2d_n}+C\big) \big).
\]
This defines a homotopy between $F_{n,s+2d_n,0}$ and $F_{n,s+2d_n,0}^\infty$ (for each $n\in \NN$), which decomposes with respect to the grading as
\[
F_{n,s+2d_n,0}^\infty=
	 \left(
	\begin{matrix}
	0& 1 \\
	1-p_{n,0} &0
	\end{matrix}
	\right) .
\]
Hence for each $n$, this homotopy provides that
\[
\Ind_{F_n^{(\infty)}}[p_n] = [p_n\otimes p_{n,0}],
\]
for any projection $(p_n)\in C^*(\H_r)^\Gamma \cap \prod_{n=1}^\infty C^*(\H_{r,n})^\Gamma$. Since the above homotopies are \emph{not} uniformly continuous with respect to $n$, we need an extra argument (which is called a ``stacking argument'' in the proof of \cite[Proposition 12.6.3]{willett2020higher}) to conclude that
\[
\Ind_{F^{(\infty)}}[(p_n)] = [(p_n\otimes p_{n,0})].
\]


For the convenience to readers, here we provide more details. Without loss of generality, for any projection $(p_n)\in C^*(\H_r)^\Gamma \cap \prod_{n=1}^\infty C^*(\H_{r,n})^\Gamma$ we assume that
\[
\Ind_{F^{(\infty)}}[(p_n)] = [(q_n)] - [(q_n^{(0)})],
\]
where $(q_n)$ and $(q_n^{(0)})$ are projections in $C^*(\H_{r,E})^\Gamma \cap \prod_{n=1}^\infty C^*(\H_{r,n,E})^\Gamma$. For each $n\in \NN$, the path $F_{n,s+2d_n,0}^t$ above provides a homotopy of projections $\{H_{n,t}\}_{t\in [0,1]}$ such that $H_{n,0} = q_n$. Hence for each $n\in \NN$, there exists $\delta_n>0$ such that for any $t,s \in [0,1]$ with $|t-s| < \delta_n$, we have $\|H_{n,t} - H_{n,s}\| < \frac{1}{2}$. For each $n\in \NN$, take an $K_n \in \NN$ such that $\frac{1}{K_n} < \delta_n$ and an isometry
\[
V_n: \H_{r,n,E} \longrightarrow \H_{r,n,E}^{\oplus K_n}, \quad v \mapsto (v,0,\cdots,0).
\]
They give rise to an isometry
\[
V=(V_n)_n: \H_{r,E} \longrightarrow \H'_{r,E}:= \bigoplus_n \H_{r,n,E}^{\oplus K_n}.
\]
Note that $V$ equivariantly covers the identity map on $P_r$, hence the following homomorphism
\[
\Phi:=\Ad_V: C^*(\H_{r,E})^\Gamma \cap \prod_{n=1}^\infty C^*(\H_{r,n,E})^\Gamma \longrightarrow C^*(\H'_{r,E})^\Gamma\cap \prod_{n=1}^\infty \B(\H_{r,n,E}^{\oplus K_n})^\Gamma
\]
gives rise to an isomorphism on $K$-theories (see, \emph{e.g.}, \cite[Theorem 5.2.6]{willett2020higher} for more details). Therefore, it suffices to show that
\[
\Phi_\ast (\Ind_{F^{(\infty)}}[(p_n)]) = \Phi_\ast([(p_n\otimes p_{n,0})]).
\]

Note that for each $n\in \NN$, we have
\[
\begin{bmatrix}
\begin{pmatrix}
	q_n &0 &\cdots & 0\\
	0& 0 &\cdots & 0\\
	\vdots & \vdots & \ddots & \vdots \\
	0 &  0 & \cdots &0
	\end{pmatrix}
	\end{bmatrix}
	=
	\begin{bmatrix}
	\begin{pmatrix}
   H_{n,0} &0  &\cdots &0\\
   0& H_{n,\frac{1}{K_n}}  &\cdots &0 \\
   \vdots& \vdots &\ddots & \vdots \\
   0& 0 & \cdots &H_{n,1}
   \end{pmatrix}
 \end{bmatrix}
 -
 \begin{bmatrix}
	\begin{pmatrix}
   0 &0  &\cdots &0\\
   0& H_{n,\frac{1}{K_n}}  &\cdots &0 \\
   \vdots& \vdots &\ddots & \vdots \\
   0& 0 & \cdots &H_{n,1}
   \end{pmatrix}
 \end{bmatrix}.
\]
On the other hand, note that
\[
\left\|~ \begin{pmatrix}
   0 &0  &\cdots &0\\
   0& H_{n,\frac{1}{K_n}}  &\cdots &0 \\
   \vdots& \vdots &\ddots & \vdots \\
   0& 0 & \cdots &H_{n,1}
   \end{pmatrix}
   -
   \begin{pmatrix}
   0 &0  &\cdots &0\\
   0& H_{n,0}  &\cdots &0 \\
   \vdots& \vdots &\ddots & \vdots \\
   0& 0 & \cdots &H_{n,\frac{K_n-1}{K_n}}
   \end{pmatrix}
    ~ \right\| < \frac{1}{2}
\]
Hence for each $n\in \NN$ there exists a homotopy connecting these two matrices, which are \emph{uniformly} continuous with respect to $n$. Following by a rotation homotopy between
\[
\begin{pmatrix}
   0 &0  &\cdots &0\\
   0& H_{n,0}  &\cdots &0 \\
   \vdots& \vdots &\ddots & \vdots \\
   0& 0 & \cdots &H_{n,\frac{K_n-1}{K_n}}
   \end{pmatrix}
   \quad \mbox{and} \quad
\begin{pmatrix}
  H_{n,0} &0  &\cdots &0 & 0\\
  0& H_{n,\frac{1}{K_n}}  &\cdots &0 &0 \\
  \vdots& \vdots &\ddots & \vdots & \vdots \\
  0& 0 & \cdots &H_{n,\frac{K_n-1}{K_n}} &0 \\
  0& 0 & \cdots &0  &0
 \end{pmatrix},
\]
we obtain that
\[
\begin{bmatrix}
\begin{pmatrix}
	q_n &0 &\cdots & 0\\
	0& 0 &\cdots & 0\\
	\vdots & \vdots & \ddots & \vdots \\
	0 &  0 & \cdots &0
	\end{pmatrix}
	\end{bmatrix}
	=
	\begin{bmatrix}
\begin{pmatrix}
	0& \cdots & 0 &0\\
	\vdots & \ddots &\vdots &\vdots\\
	0 & \cdots  & 0 & 0 \\
	0  & \cdots &  0 &H_{n,1}
	\end{pmatrix}
	\end{bmatrix}
	=
	\begin{bmatrix}
    \begin{pmatrix}
	H_{n,1} &0 &\cdots & 0\\
	0& 0 &\cdots & 0\\
	\vdots & \vdots & \ddots & \vdots \\
	0 &  0 & \cdots &0
	\end{pmatrix}
	\end{bmatrix}.
\]
Note that these homotopies are \emph{uniformly} continuous with respect to $n$, hence from the construction of $\Phi$ we obtain that
\[
\Phi_\ast (\Ind_{F^{(\infty)}}[(p_n)]) = \Phi_\ast([(p_n\otimes p_{n,0})]),
\]
which concludes the proof.
\end{proof}

\section{Isomorphisms between $K$-theories of equivariant twisted algebras}\label{sec:local isom}

In this section, we study the $K$-theories of equivariant twisted algebras $A^r(X;E)^\Gamma$ and $A^r_L(X;E)^\Gamma$ introduced in Section~\ref{ssec:equiv. twisted algebras}. First recall that for $r>0$, there is the evaluation-at-one homomorphism
\[
\mathrm{ev}: A^r_L(X;E)^\Gamma \longrightarrow A^r(X;E)^\Gamma, \quad (T_t) \mapsto T_1.
\]
Also recall that for $r \leq s$, the inclusion $Z_r \to Z_s$ induces isometric embeddings of Hilbert spaces
\[
\H_{r,n} \longrightarrow \H_{s,n} \quad \mbox{and} \quad \H_{r,n,E} \longrightarrow \H_{s,n,E}
\]
for each $n\in \NN$. These maps give rise to the following commutative diagram:
\[
\xymatrix{
			A^r_L(X;E)^\Gamma \ar[r]^{\mathrm{ev}} \ar[d] & A^r(X;E)^\Gamma  \ar[d] \\
			A^s_L(X;E)^\Gamma \ar[r]^{\mathrm{ev}} & A^s(X;E)^\Gamma
	}
\]
for each $r \leq s$, which further induces a homomorphism between direct limits when taking $r \to +\infty$.

The main result of this section is the following:

\begin{thm}\label{thm:iso. of twisted algebras in $K$-theory}
The evaluation-at-one map induces the following isomorphism in $K$-theories:
\[
\mathrm{ev}_\ast: \lim_{r\to \infty} K_\ast (A^r_L(X;E)^\Gamma) \longrightarrow \lim_{r\to \infty} K_\ast (A^r(X;E)^\Gamma)
\]
for $\ast=0,1$.
\end{thm}

The proof follows the outline of that in \cite[Section 12.4]{willett2020higher}, and the main ingredient is to use appropriate Mayer-Vietoris arguments for twisted algebras (Proposition~\ref{prop:M_V for twisted algebra}). This allows us to chop the space into easily-handled pieces, on which we apply the result from \cite{HK01} on the Baum-Connes conjecture with coefficients for a-T-menable groups.

Let us start with some more notions. By saying that $(F_n)_{n \in \NN}$ is a sequence of closed subsets in $(E_n)_{n \in \NN}$, we mean that each $F_n$ is a closed subset of $E_n$. Firstly, we define the following subalgebras associated to $(F_n)$:

\begin{defn}\label{defn:twisted ideal}
Fix an $r>0$. For a sequence of closed subsets $(F_n)$ in $(E_n)$, we define $\AA^r_{(F_n)}(X;E)^\Gamma$ to be the set of elements $(T_{n,s}) \in \AA^r(X;E)^\Gamma$ satisfying: for each $n$ and $\varepsilon >0$ there exists $s_{n, \varepsilon }>0$ such that for $s\ge s_{n, \varepsilon }$ we have:
\[
\supp _{E} (T_{n,s}) \subseteq \Nd _\varepsilon (F_n) \times \Nd _\varepsilon (F_n).
\]
Denote by $A^r_{(F_n)}(X;E)^\Gamma$ the $C^*$-algebra of the norm closure of $\AA^r_{(F_n)}(X;E)^\Gamma$ in $A^r(X;E)^\Gamma$.

Denote by $\AA^r_{L,(F_n)}(X;E)^\Gamma$ the set of elements $(T_t)$ in $\AA^r_L(X;E)^\Gamma$ satisfying that $T_t \in \AA^r_{(F_n)}(X;E)^\Gamma$ for each $t\in [1,\infty)$. Define $A^r_{L,(F_n)}(X;E)^\Gamma$ to be the $C^*$-algebra of the completion of $\AA^r_{L,(F_n)}(X;E)^\Gamma$ with respect to the norm $\|(T_t)\|=\sup_t \|T_t\|$. Equivalently, the algebra $A^r_{L,(F_n)}(X;E)^\Gamma$ consists of elements $(T_t)$ in $A^r_L(X;E)^\Gamma$ such that each $T_t$ belongs to $A^r_{(F_n)}(X;E)^\Gamma$.
\end{defn}

The following lemma is straightforward, hence we omit the proof.

\begin{lem}\label{lem:twisted ideals are ideals}
With the same notation as above, $A^r_{(F_n)}(X;E)^\Gamma$ is an ideal in $A^r(X;E)^\Gamma$ and $A^r_{L,(F_n)}(X;E)^\Gamma$ is an ideal in $A^r_L(X;E)^\Gamma$.
\end{lem}

Moreover, we have the following.

\begin{lem}\label{lem:M-V preparation for twisted ideals}
Fix an $r>0$. Let $(F_n)$ and $(G_n)$ be two sequences of compact subsets in $(E_n)$. Then
\[
A^r_{(F_n)}(X;E)^\Gamma \cap A^r_{(G_n)}(X;E)^\Gamma = A^r_{(F_n \cap G_n)}(X;E)^\Gamma
\]
and
\[
A^r_{(F_n)}(X;E)^\Gamma + A^r_{(G_n)}(X;E)^\Gamma = A^r_{(F_n \cup G_n)}(X;E)^\Gamma.
\]
The same holds for the localisation case.
\end{lem}

\begin{proof}
We first prove the case of equivariant twisted Roe algebras. For the first equation, note that $A^r_{(F_n \cap G_n)}(X;E)^\Gamma \subseteq A^r_{(F_n)}(X;E)^\Gamma \cap A^r_{(G_n)}(X;E)^\Gamma$ holds trivial. The converse comes from a $C^*$-algebraic fact that the intersection of two ideals coincides with their product together with a basic fact for metric space: For a compact metric space $K$, a closed cover $(C,D)$ of $K$ and $\varepsilon>0$, there exists $\delta>0$ such that $\Nd_\delta(C) \cap \Nd_\delta(D) \subseteq \Nd_\varepsilon(C\cap D)$.
	
For the second equation, note that $A^r_{(F_n)}(X;E)^\Gamma + A^r_{(G_n)}(X;E)^\Gamma \subseteq A^r_{(F_n \cup G_n)}(X;E)^\Gamma$ holds trivially. As for the converse, we fix a $(T_{n,s})\in \AA^r_{(F_n \cup G_n)}(X;E)^\Gamma$. For $n\in \NN$ there is a strictly increasing sequence $(s_{n,k})_{k\in \NN}$ in $[1,\infty)$ tending to infinity such that for $s\ge s_{n,k}$, we have
\[
\supp_{E}(T_{n,s}) \subseteq \Nd_{\frac{1}{k+1}}(F_n\cup G_n) \times \Nd_{\frac{1}{k+1}}(F_n \cup G_n).
\]
For each $n\in \NN$ and $s\in [1,\infty)$, we construct an operator $W_{n,s}\in \B(\H_{r,n,E})^\Gamma$ as follows:
\[
W_{n,s}=
	\begin{cases}
	~\chi_{\Nd_{1}(F_n)}, & \mbox{if~} 1\le s \le s_{n,1};  \\[0.3cm]
	~\frac{s_{n,k+1}-s}{s_{n,k+1} - s_{n,k}}\chi_{\Nd_{\frac{1}{k}}(F_n)} + \frac{s-s_{n,k}}{s_{n,k+1} - s_{n,k}}\chi_{\Nd_{\frac{1}{k+1}}(F_n)}, & \mbox{for~} s_{n,k} \leq s \leq s_{n,k+1} \mbox{~where~} k\in \NN.
	\end{cases}
\]
Then the map $s \mapsto W_{n,s}$ is in $C_b([1,\infty), \B(\H_{r,n,E})^\Gamma)$. Moreover, it is clear that $(s \mapsto W_{n,s})_n$ is also in the multiplier algebra of $A^r(X;E)^\Gamma$. Now we consider:
\begin{align*}
(T_{n,s}) &= (W_{n,s}) (T_{n,s}) + (1-W_{n,s}) (T_{n,s}) \\
	&= (W_{n,s}) (T_{n,s}) + (1-W_{n,s}) (T_{n,s}) (W_{n,s}) +(1-W_{n,s}) (T_{n,s}) (1-W_{n,s}) .
\end{align*}
It is clear that $(W_{n,s}) (T_{n,s})$ and $(1-W_{n,s}) (T_{n,s}) (W_{n,s})$ are in $A^r_{(F_n)}(X;E)^\Gamma$. Also note that from the construction above, for each $n\in \NN$ and $s \geq s_{n,k}$ we have:
\[
\supp_{E}((1-W_{n,s}) T_{n,s} (1-W_{n,s})) \subseteq \Nd_{\frac{1}{k+1}}( G_n) \times \Nd_{\frac{1}{k+1}}( G_n).
\]
Hence we obtain that $A^r_{(F_n)}(X;E)^\Gamma + A^r_{(G_n)}(X;E)^\Gamma$ is dense in $A^r_{(F_n \cup G_n)}(X;E)^\Gamma$, which concludes the proof.

For the localisation case, we apply the above argument pointwise to obtain the first equation. Concerning the second, note that the above $(W_{n,s})$ might not vary continuously with respect to the parameter $t$. To get around the issue, we need an approximation argument.

Fix a $(T_t) = (T_{t,n,s})\in \AA^r_{L,(F_n \cup G_n)}(X;E)^\Gamma$. Note that the map $t \mapsto T_t$ is uniformly continuous, hence for any $\varepsilon>0$ there is a $\delta>0$ such that for any $t',t'' \in [1,\infty)$ with $|t' - t''| \leq \delta$, we have $\|T_{t'} - T_{t''}\| \leq \varepsilon$. For each $p\in \NN$ we set $t_p = 1+p\delta$, and also set $t_{-1} = 1$.
We construct another element $(T^\varepsilon_t) = (T^\varepsilon_{t,n,s})\in \AA^r_{L,(F_n \cup G_n)}(X;E)^\Gamma$ by setting $T^\varepsilon_{t_p} = T_{t_p}$ for each $p \in \NN$ and do linear combination between $t_p$ and $t_{p+1}$. It is clear that $\|(T_t) - (T^\varepsilon_t)\| \leq 2\varepsilon$. Since $\varepsilon$ is arbitrarily chosen, it suffices to show that each $(T^\varepsilon_t)$ is in $A^r_{L,(F_n)}(X;E)^\Gamma + A^r_{L,(G_n)}(X;E)^\Gamma$.

For each $n\in \NN$ and $p\in \NN \cup \{0\}$, there is a strictly increasing sequence $(s_{n,k,p})_{k\in \NN}$ in $[1,\infty)$ tending to infinity such that for $s\ge s_{n,k,p}$ and $t=t_{p-1}, t_p, t_{p+1}$ we have
\[
\supp_{E}(T_{t,n,s}) \subseteq \Nd_{\frac{1}{k+1}}(F_n\cup G_n) \times \Nd_{\frac{1}{k+1}}(F_n \cup G_n).
\]
For each $n\in \NN$, $p\in \NN \cup \{0\}$ and $s\in [1,\infty)$, we construct an operator $W_{t_p,n,s}\in \B(\H_{r,n,E})^\Gamma$ as follows:
\[
W_{t_p,n,s}=
	\begin{cases}
	~\chi_{\Nd_{1}(F_n)}, & \mbox{if~} 1\le s \le s_{n,1,p};  \\[0.3cm]
	~\frac{s_{n,k+1}-s}{s_{n,k+1} - s_{n,k}}\chi_{\Nd_{\frac{1}{k}}(F_n)} + \frac{s-s_{n,k}}{s_{n,k+1} - s_{n,k}}\chi_{\Nd_{\frac{1}{k+1}}(F_n)}, & \mbox{for~} s_{n,k,p} \leq s \leq s_{n,k+1,p} \mbox{~where~} k\in \NN.
	\end{cases}
\]
Then the map $s \mapsto W_{t_p,n,s}$ is in $C_b([1,\infty), \B(\H_{r,n,E})^\Gamma)$. Moreover, it is clear that for each $p\in \NN \cup \{0\}$, $(s \mapsto W_{t_p,n,s})_n$ is also in the multiplier algebra of $A^r(X;E)^\Gamma$. Now we define a path $(W_t)_{t\in [1,\infty)}$ by setting $W_{t_p} = (W_{t_p,n,s})$ for each $p\in \NN \cup \{0\}$, and do linear combination between $t_p$ and $t_{p+1}$. It is clear that $t \mapsto W_t = (W_{t,n,s})$ is uniformly continuous and bounded, and hence it is in the multiplier algebra of $A^r_L(X;E)^\Gamma$. Now consider:
\begin{align*}
(T^\varepsilon_{t,n,s}) &= (W_{t,n,s}) (T^\varepsilon_{t,n,s}) + (1-W_{t,n,s}) (T^\varepsilon_{t,n,s}) \\
	&= (W_{t,n,s}) (T^\varepsilon_{t,n,s}) + (1-W_{t,n,s}) (T^\varepsilon_{t,n,s}) (W_{t,n,s}) +(1-W_{t,n,s}) (T^\varepsilon_{t,n,s}) (1-W_{t,n,s}).
\end{align*}
It is clear that $t\mapsto W_t \cdot T^\varepsilon_t$ and $t \mapsto (1-W_t) \cdot  T^\varepsilon_t \cdot W_t$ are in $A^r_{L,(F_n)}(X;E)^\Gamma$. Note that for any $p\in \NN \cup \{0\}$ and $s \leq s_{n,k+1,p}$, we have
\[
\supp_E(1-W_{t_p,n,s}) \subseteq (\Nd_{\frac{1}{k+1}}(F_n)^c) \times (\Nd_{\frac{1}{k+1}}(F_n)^c).
\]
Setting $s'_{n,k,p}$ to be the minimal of $s_{n,k,p}$ and $s_{n,k,p+1}$, it follows that for any $t\in [t_p,t_{p+1}]$ and $s \in [s'_{n,k,p}, s'_{n,k+1,p}]$, we have
\[
\supp_E (1-W_{t,n,s}) = \supp_E\big( 1 - (\lambda W_{t_p,n,s} + (1-\lambda) W_{t_{p+1},n,s}) \big) \subseteq (\Nd_{\frac{1}{k+1}}(F_n)^c) \times (\Nd_{\frac{1}{k+1}}(F_n)^c)
\]
and
\[
\supp_E (T^\varepsilon_{t,n,s}) = \supp_E \big( \lambda T_{t_p,n,s} + (1-\lambda)T_{t_{p+1},n,s} \big) \subseteq \Nd_{\frac{1}{k+1}}(F_n\cup G_n) \times \Nd_{\frac{1}{k+1}}(F_n \cup G_n),
\]
where $\lambda$ is the combinatorial parameter determined by $t\in [t_p,t_{p+1}]$. These inequalities imply that
\[
\supp_{E}\big((1-W_{t,n,s}) T^\varepsilon_{t,n,s} (1-W_{t,n,s})\big) \subseteq \Nd_{\frac{1}{k+1}}( G_n) \times \Nd_{\frac{1}{k+1}}( G_n).
\]
Finally note that the sequence $(s'_{n,k,p})_{k\in \NN}$ tends to infinity, hence we obtain that $(T^\varepsilon_t) \in A^r_{L,(F_n)}(X;E)^\Gamma + A^r_{L,(G_n)}(X;E)^\Gamma$ as required.
\end{proof}

Consequently, we obtain the following Mayer-Vietoris sequences for twisted algebras:
\begin{prop}\label{prop:M_V for twisted algebra}
Let $(F_n)$ and $(G_n)$ be two sequences of compact subsets in $(E_n)$ and fix an $r>0$. Then we have the following six-term exact sequence:
\begin{small}
\[
	\xymatrix{
				K_0\big(A_{(F_n\cap G_n)}^r(X;E)^\Gamma\big) \ar[r]  &  K_0\big(A_{(F_n)}^r(X;E)^\Gamma\big) \oplus K_0\big(A_{(G_n)}^r(X;E)^\Gamma\big) \ar[r] & K_0\big(A_{(F_n\cup G_n)}^r(X;E)^\Gamma\big) \ar[d]  \\
				K_1\big(A_{(F_n\cup G_n)}^r(X;E)^\Gamma\big) \ar[u]  & K_1\big(A_{(F_n)}^r(X;E)^\Gamma\big) \oplus K_1\big(A_{(G_n)}^r(X;E)^\Gamma\big) \ar[l] & K_1\big(A_{(F_n\cap G_n)}^r(X;E)^\Gamma\big) \ar[l]
	}
\]
\end{small}
The same holds in the case of equivariant twisted localisation algebras. Furthermore, we have the following commutative diagram:\\
\begin{small}
		\centerline{
			\xymatrix{
				\cdots \ar[r] & K_*\big(A_{L,(F_n\cap G_n)}^r(X;E)^\Gamma\big) \ar[r]  \ar[d] &  K_\ast\big(A_{L,(F_n)}^r(X;E)^\Gamma\big) \oplus K_1\big(A_{L,(G_n)}^r(X;E)^\Gamma\big) \ar[d]\ar[r] & K_\ast\big(A_{L,(F_n\cup G_n)}^r(X;E)^\Gamma\big) \ar[d]\ar[r]  & \cdots \\
				\cdots \ar[r] & K_*\big(A_{(F_n\cap G_n)}^r(X;E)^\Gamma\big) \ar[r]  & K_\ast\big(A_{(F_n)}^r(X;E)^\Gamma\big) \oplus K_\ast\big(A_{(G_n)}^r(X;E)^\Gamma\big) \ar[r] & K_\ast\big(A_{(F_n\cup G_n)}^r(X;E)^\Gamma\big) \ar[r] & \cdots
		}}
\end{small}
where the vertical maps are induced by the evaluation-at-one homomorphisms.
\end{prop}

Proposition \ref{prop:M_V for twisted algebra} allows us to chop the space into small pieces, on which we have the following ``local isomorphism'' result. Recall that a family $\{Z_i\}_{i\in I}$ of subspaces in a metric space $Z$ is \emph{mutually $R$-separated} for some $R>0$ if $d(Z_i,Z_j) >R$ for $i \neq j$.

\begin{prop}\label{prop:local isom. in K-theory}
Let $(F_n)$ be a sequence of closed subsets in $(E_n)$ such that $F_n = \bigsqcup_{j=1}^{\infty}F_j^{(n)}$ for a mutually $3$-separated family $\{F_j^{(n)}\}_{j}$, and there exist $R>0$ and $y^{(n)}_j\in X_n/\Gamma$ such that $F^{(n)}_j \subseteq B(\xi(y^{(n)}_j);R)$. Then the homomorphism induced by the evaluation-at-one map
\[
\mathrm{ev}_\ast: \lim_{r\to \infty} K_\ast\big( A_{L, (F_n)}^r(X;E)^\Gamma \big) \longrightarrow \lim_{r\to \infty} K_\ast\big( A_{(F_n)}^r(X;E)^\Gamma \big)
\]
is an isomorphism for $\ast=0,1$.		
\end{prop}

The proof of Proposition \ref{prop:local isom. in K-theory} is technical and divided into several steps. Before we present the detailed proof, let us first use it to conclude the proof of Theorem \ref{thm:iso. of twisted algebras in $K$-theory}. To achieve that, we need an extra lemma from \cite[Lemma 12.4.5]{willett2020higher}:

\begin{lem}\label{lem:decomposition}
For any $s>0$, there exist $M \in \NN$ and decompositions
\[
X_n/\Gamma=Y_{n,1}\sqcup Y_{n,2}\sqcup \cdots \sqcup Y_{n,M} \quad \mbox{for all } n\in \NN,
\]
such that the family $\big\{\overline{B(\xi(y);s)}\big\}_{y\in Y_{n,i}}$ is mutually $3$-separated for each $n\in\NN$ and $i=1,2,\ldots,M$.
\end{lem}

\begin{proof}[Proof of Theorem \ref{thm:iso. of twisted algebras in $K$-theory}]
Given $s>0$, let $M\in \NN$ and $\{Y_{n,i}\}_{n\in\NN,1\le i\le M}$ be provided by Lemma~\ref{lem:decomposition}. Setting $W^s_n:=\Nd_s(\xi(Y_n)) = \Nd_s(f_n(X_n))$ and $W_{n,i}^s:=\bigsqcup_{y\in Y_{n,i}}\overline{B(\xi(y);s)}$, we have $W^s_n=\bigcup_{i=1}^M W_{n,i}^s$. For each $i$ applying Proposition~\ref{prop:local isom. in K-theory} to the sequence of subsets $(W_{n,i}^s)_n$, we obtain that the homomorphism induced by the evaluation-at-one map
\[
\mathrm{ev}_\ast: \lim_{r\to \infty} K_\ast\big( A_{L, (W_{n,i}^s)}^r(X;E)^\Gamma \big) \longrightarrow \lim_{r\to \infty} K_\ast\big( A_{(W_{n,i}^s)}^r(X;E)^\Gamma \big)
\]
is an isomorphism for each $i$. Applying the Mayer-Vietoris sequence from Proposition~\ref{prop:M_V for twisted algebra} $(M-1)$-times (and Proposition~\ref{prop:local isom. in K-theory} again to deal with the intersection) together with the Five Lemma, we obtain that the homomorphism induced by the evaluation-at-one map
\[
\mathrm{ev}_\ast: \lim_{r\to \infty} K_\ast\big( A_{L, (W^s_n)}^r(X;E)^\Gamma \big) \longrightarrow \lim_{r\to \infty} K_\ast\big( A_{(W^s_n)}^r(X;E)^\Gamma \big)
\]
is an isomorphism. Finally note that condition $(3)$ in Definition~\ref{defn:twisted Roe} and condition $(2)$ in Definition \ref{defn:twisted localisation} imply that for each $r>0$, we have
\[
A^r(X;E)^\Gamma=\lim_{s\rightarrow \infty}A^r_{(W^s_n)}(X;E)^\Gamma \quad \mbox{and} \quad A^r_L(X;E)^\Gamma=\lim_{s\rightarrow \infty}A^r_{L,(W^s_n)}(X;E)^\Gamma.
\]
Hence we conclude the proof.
\end{proof}

The rest of this section is devote to the proof of Proposition \ref{prop:local isom. in K-theory}. First let us introduce some more notation. Fix an $r>0$. Let $(F_n)$ and $(G_n)$ be sequences of closed subsets in $(E_n)$. Define:
\[
A^r(X;(G_n))^\Gamma:=(\Id_{L^2([1,\infty))} \otimes \Id_{\H_{r,n}} \otimes \chi_{G_n})_n \cdot  A^r(X;E)^\Gamma \cdot (\Id_{L^2([1,\infty))} \otimes \Id_{\H_{r,n}} \otimes \chi_{G_n})_n
\]
and
\[
A^r_{(F_n)}(X;(G_n))^\Gamma:= A^r(X;(G_n))^\Gamma \cap A^r_{(F_n)}(X;E)^\Gamma.
\]
Moreover, given a sequence of $\Gamma$-invariant subspaces $O_n \subseteq X_n~ (n\in \NN)$ we define:
\[
A^r( (O_n) ; (G_n) )^\Gamma:=\big(\Id_{L^2([1,\infty))} \otimes \chi_{P_r(O_n)} \otimes \Id_{\L^2_{E_n}}\big)_n \cdot A^r( X;(G_n) )^\Gamma \cdot \big(\Id_{L^2([1,\infty))} \otimes \chi_{P_r(O_n)} \otimes \Id_{\L^2_{E_n}}\big)_n
\]
and
\[
A^r_{(F_n) }( (O_n) ; (G_n) )^\Gamma:=\big(\Id_{L^2([1,\infty))} \otimes \chi_{P_r(O_n)} \otimes \Id_{\L^2_{E_n}}\big)_n \cdot  A^r_{(F_n)}(X;(G_n))^\Gamma \cdot \big(\Id_{L^2([1,\infty))} \otimes \chi_{P_r(O_n)} \otimes \Id_{\L^2_{E_n}}\big)_n.
\]
Define the $C^*$-algebras $A^r_L(X;(G_n))^\Gamma$ and $A^r_L( (O_n) ; (G_n) )^\Gamma$ to be those consisting of elements $(T_t) \in A^r_L(X;E)^\Gamma$ such that each $T_t$ belongs to $A^r(X;(G_n))^\Gamma$ and $A^r( (O_n) ; (G_n) )^\Gamma$, respectively. Also define the $C^*$-algebras $A^r_{L,(F_n)}(X;(G_n))^\Gamma$ and $A^r_{L,(F_n) }( (O_n) ; (G_n) )^\Gamma$ to be those consisting of elements $(T_t) \in A^r_{L,(F_n)}(X;E)^\Gamma$ such that each $T_t$ belongs to $A^r_{(F_n)}(X;(G_n))^\Gamma$ and $A^r_{(F_n) }( (O_n) ; (G_n) )^\Gamma$, respectively.

On the other hand, we define
\[
A^r_0(X;E)^\Gamma:=\big\{(T_{n,s})\in A^r(X;E)^\Gamma: \lim\limits_{s\rightarrow \infty}T_{n,s} = 0 \mbox{~for~each~}n\in \NN\big\}
\]
and
\[
A^r_0(X;(G_n))^\Gamma:= A^r_0(X;E)^\Gamma \cap A^r(X;(G_n))^\Gamma
\]
It is clear that $A^r_0(X;(G_n))^\Gamma$ is an ideal in $A^r(X;(G_n))^\Gamma$, and also an ideal in $A^r_{(F_n)}(X;(G_n))^\Gamma$ for any sequence of closed subsets $(F_n)$. Define the algebras $A^r_{L,0}(X;E)^\Gamma$ and $A^r_{L,0}(X;(G_n))^\Gamma$ to be those consisting of elements $(T_t) \in A^r_L(X;E)^\Gamma$ such that each $T_t$ belongs to $A^r_0(X;E)^\Gamma$ and $A^r_0(X;(G_n))^\Gamma$, respectively.

\begin{lem}\label{lem:K-zero for tend to zero algebra}
Fix an $r>0$. Then we have:
\[
	K_\ast (A^r_0(X;E)^\Gamma) = 0 \quad \mbox{and} \quad K_\ast(A^r_0(X;(G_n))^\Gamma) = 0.
\]
And we also have:
\[
	K_\ast (A^r_{L,0}(X;E)^\Gamma) = 0 \quad \mbox{and} \quad K_\ast(A^r_{L,0}(X;(G_n))^\Gamma) = 0.
\]
\end{lem}

\begin{proof}
Here we only prove $K_\ast (A^r_0(X;E)^\Gamma) = 0$ since the other is similar, and the case of localisation algebras holds by using the same argument pointwise. Let
\[
\H_{r,n,E,\infty} := (\H_{r,n,E})^{\oplus \infty} \quad \mbox{and} \quad  \H_{r,E,\infty} := \bigoplus_n \H_{r,n,E,\infty}.
\]
Using these admissible modules we construct the equivariant twisted Roe algebra, denoted by $A^r_\infty(X;E)^\Gamma$, and similarly construct $A^{r}_{0,\infty}(X;E)^\Gamma$. For each $n\in \NN$, define an isometry
\[
V_n : \H_{r,n,E} \rightarrow \H_{r,n,E,\infty},\quad v\mapsto (v,0,0,\cdots).
\]
Then $(V_n)$ induces the ``top-left corner inclusion''
\[
	\Ad_{(V_n)} : A^{r}_{0}(X;E)^\Gamma \rightarrow A^{r}_{0,\infty}(X;E)^\Gamma, \quad (T_{n,s}) \mapsto \big( T_{n,s} \oplus \bigoplus_{k=1}^{\infty}0 \big),
\]
which induces the identity map on the $K$-groups.
	
It is also easy to check that the formula
\[
(\alpha_n) : (T_{n,s})_{n,s} \rightarrow \big( 0 \oplus \bigoplus_{k=1}^{\infty}T_{n,s+k} \big)_{n,s} \mathrm{\ where \ }
	\alpha_n : (T_{n,s})_s \rightarrow \big( 0 \oplus \bigoplus_{k=1}^{\infty}T_{n,s+k} \big)_s
\]
and
\[
(\beta_n) : (T_{n,s})_{n,s} \rightarrow \big( \bigoplus_{k=0}^{\infty}T_{n,s+k+1} \big)_{n,s} \mathrm{\ where \ }
	\beta_n : (T_{n,s})_s \rightarrow \big( \bigoplus_{k=0}^{\infty}T_{n,s+k+1} \big)_s
\]
give two well-defined homomorphisms from $ A^r_{0}(X;E)^\Gamma$ to $A^r_{0,\infty}(X;E)^\Gamma$. Let
\[
S_n : \H_{r,n,E,\infty} \rightarrow \H_{r,n,E,\infty}, \quad  (v_0,v_1,v_2,\dots) \mapsto (0,v_0,v_1,\dots)
\]
be the shift operator. It is clear that $(S_n)_n$ is in the multiplier algebra of $A^r_{0,\infty}(X;E)^\Gamma$, which builds a conjugation between $(\alpha_n)$ and $(\beta_n)$. More precisely, for each $n\in \NN$ we have
\[
\alpha_n\big( (T_{n,s})_s \big)  = S_n \cdot \beta_n\big( (T_{n,s})_s \big) \cdot S_n^*.
\]
Applying \cite[Proposition 2.7.5]{willett2020higher}, we have
\[
(\alpha_n)_* = (\beta_n)_* .
\]
	
On the other hand, notice that
\[
\big(\Ad_{(V_n)} + (\alpha_n) \big)\big( (T_{n,s})_{n,s} \big) = \big( \bigoplus_{k=0}^{\infty}T_{n,s+k} \big)_{n,s}.
\]
Since $\Ad_{(V_n)}$ and $(\alpha_n)$ have orthogonal images and $\big(\Ad_{(V_n)} + (\alpha_n) \big)\big( (T_{n,s})_{n,s}\big)$ is homotopic to $(\beta_n)\big((T_{n,s})_{n,s}\big)$, we have
\[
(\beta_n)_* = (\Ad_{(V_n)} + (\alpha_n) )_* = (\Ad_{(V_n)})_* + (\alpha_n)_*.
\]
Consequently we obtain that $(\Ad_{(V_n)})_* = 0$, which concludes the proof.
\end{proof}

Now we move back to Proposition \ref{prop:local isom. in K-theory}. Fix an $r>0$. Recall that $F_n = \bigsqcup_{j=1}^{\infty}F_j^{(n)}$ where $\{F_j^{(n)}\}_{j}$ is a mutually $3$-separated family, and there exist $R>0$ and $y^{(n)}_j\in X_n/\Gamma$ such that $F^{(n)}_j \subseteq B(\xi(y^{(n)}_j);R)$. Also recall that the $\Gamma$-action on $\bigsqcup_n X_n$ has controlled distortion with certain fundamental domain $\D$. For each $n$ and $j$, denote $x^{(n)}_j$ the unique point in $X_n \cap \D$ such that $\pi(x^{(n)}_j) = y^{(n)}_j$. Since each $X_n/\Gamma$ is finite, then for each $n\in \NN$ there are only finitely many $j$'s such that $F_j^{(n)}$ is non-empty. Denote the set of such $j$ by $J_n$. Taking $G_j^{(n)}=\Nd_{1}(F_j^{(n)})$ for each $n$ and $j$,
we define the ``restricted product'':
\begin{align*}
\prod_{j}^{res}A^r_{(F_j^{(n)})}( X;(G_j^{(n)}) )^\Gamma :&= \left(\prod_{j}A^r_{(F_j^{(n)}) }( X;(G_j^{(n)}) )^\Gamma \right) \cap A^r_{(F_n)}(X;E)^\Gamma\\
& = \left\{ (T_{n,s}) \in A^r_{(F_n)}(X;E)^\Gamma: \supp_E(T_{n,s}) \subseteq \bigsqcup_{j\in J_n} G_j^{(n)} \times G_j^{(n)} \right\}.
\end{align*}
Also define the algebra $\prod_{j}^{res}A^r_{L,(F_j^{(n)})}( X;(G_j^{(n)}) )^\Gamma$ to be those consisting of elements $(T_t) \in A^r_L(X;E)^\Gamma$ such that each $T_t$ belongs to $\prod_{j}^{res}A^r_{(F_j^{(n)})}( X;(G_j^{(n)}) )^\Gamma$.
We have the following lemma, whose proof is again from a standard Eilenberg Swindle argument (as in Lemma \ref{lem:K-zero for tend to zero algebra}), hence omitted.

\begin{lem}\label{lem:K-zero for tend to zero algebra 2}
With the same notations as above, the algebras
\[
\prod_{j}^{res}A^r_0( X;(G_j^{(n)}) )^\Gamma:=\big(\prod_{j}^{res}A^r_{(F_j^{(n)})}( X;(G_j^{(n)}) )^\Gamma \big) \cap A^r_{0}(X;E)^\Gamma
\]
and
\[
\prod_{j}^{res}A^r_{L,0}( X;(G_j^{(n)}) )^\Gamma:=\big(\prod_{j}^{res}A^r_{L,(F_j^{(n)})}( X;(G_j^{(n)}) )^\Gamma \big) \cap A^r_{L,0}(X;E)^\Gamma
\]
have trivial $K$-theories.
\end{lem}

The following lemma is a key step in the proof of Proposition \ref{prop:local isom. in K-theory}:

\begin{lem}\label{lem:cluster axiom}
With the same notation as above, the following inclusions
\[
\iota: \prod_{j}^{res}A^r_{(F_j^{(n)})}( X;(G_j^{(n)}) )^\Gamma \hookrightarrow A^r_{(F_n)}(X;E)^\Gamma
\]
and
\[
\iota_L: \prod_{j}^{res}A^r_{L,(F_j^{(n)})}( X;(G_j^{(n)}) )^\Gamma \hookrightarrow A^r_{L,(F_n)}(X;E)^\Gamma
\]
induce isomorphisms in $K$-theory.
\end{lem}

\begin{proof}
We only prove the first, and the second can be proved using the same argument pointwise. The proof follows the outline of \cite[Theorem 6.4.20]{willett2020higher}.
	
Consider the following quotient algebras:
\[
A^r_{(F_n),Q}(X;E)^\Gamma := \frac{A^r_{(F_n)}(X;E)^\Gamma}{A^r_0(X;E)^\Gamma} \quad \mbox{and} \quad \prod_{j}^{res,Q}A^r_{(F_j^{(n)})}( X;(G_j^{(n)}) )^\Gamma:= \frac{\prod\limits_{j}^{res}A^r_{(F_j^{(n)})}( X;(G_j^{(n)}) )^\Gamma}{\prod\limits_{j}^{res}A^r_0( X;(G_j^{(n)}) )^\Gamma}.
\]
It follows from Lemma \ref{lem:K-zero for tend to zero algebra} and \ref{lem:K-zero for tend to zero algebra 2} that the quotient maps
\[
A^r_{(F_n)}(X;E)^\Gamma \rightarrow A^r_{(F_n),Q}(X;E)^\Gamma \quad \mbox{and} \quad \prod_{j}^{res}A^r_{(F_j^{(n)})}( X;(G_j^{(n)}) )^\Gamma  \rightarrow \prod_{j}^{res,Q}A^r_{(F_j^{(n)})}( X;(G_j^{(n)}) )^\Gamma
\]
induce isomophisms in $K$-theories.

It is clear that the inclusion $\iota$ induces a $*$-homomorphism:
\[
\iota_Q: \prod_{j}^{res,Q}A^r_{(F_j^{(n)})}( X;(G_j^{(n)}) )^\Gamma \longrightarrow A^r_{(F_n),Q}(X;E)^\Gamma.
\]
We also define a map
\[
\gamma :A^r_{(F_n)}(X;E)^\Gamma \longrightarrow \prod_{j}^{res}A^r_{(F_j^{(n)})}( X;(G_j^{(n)}) )^\Gamma \quad \mbox{by} \quad (T_{n,s})\mapsto \prod_{j}(\chi_{G_j^{(n)}}T_{n,s}\chi_{G_j^{(n)}}),
\]
which induces a $*$-homomorphism
\[
\gamma_Q: A^r_{(F_n),Q}(X;E)^\Gamma \longrightarrow \prod_{j}^{res,Q}A^r_{(F_j^{(n)})}( X;(G_j^{(n)}) )^\Gamma.
\]
We claim that the composition $\iota_Q\circ \gamma_Q$ equals to the identity map on the quotient algebra, \emph{i.e.}, the map $\iota \circ \gamma$ equals to the identity map modulo $A^r_0(X;E)^\Gamma$. Given $(T_{n,s})\in A^r_{(F_n)}(X;E)^\Gamma$, then for each $n \in \NN$ there exists $s_n \in [1,\infty)$ such that for any $s>s_n$ we have $\supp_E(T_{n,s}) \subseteq \bigsqcup_j \big( G_j^{(n)} \times G_j^{(n)}\big)$. Hence for $s>s_n$, we obtain:
\[
\sum_{j} \chi_{G_j^{(n)}}T_{n,s}\chi_{G_j^{(n)}} - T_{n,s}  =  0,
\]
which implies that
\[
( T_{n,s} ) - \sum_{j} ( \chi_{G_j^{(n)}}T_{n,s}\chi_{G_j^{(n)}} ) \in A^r_0(X;E)^\Gamma.
\]
On the other hand, it is clear that $\gamma\circ \iota$ is the identity map and descends to identity map $\gamma_Q\circ \iota_Q$. Hence $\iota_Q$ induces an isomorphism in $K$-theory, which implies that the inclusion $i$ induces an isomorphism in $K$-theory.
\end{proof}


Using the same notation as above, we consider the following commutative diagram:
\[
\xymatrix{
			A^r_{L,(F_n)}(X;E)^\Gamma \ar[r]^{\mathrm{ev}} & A^r_{(F_n)}(X;E)^\Gamma  \\
			\prod\limits_j^{res}A^r_{L,(F_j^{(n)})}( X;(G_j^{(n)}) )^\Gamma \ar[r]^{\mathrm{ev}} \ar@{^(->}[u] & \prod\limits_j^{res}A^r_{(F_j^{(n)})}( X;(G_j^{(n)}) )^\Gamma. \ar@{^(->}[u]
	}
\]
It follows from Lemma \ref{lem:cluster axiom} that vertical maps induce isomorphisms in $K$-theory.
Also note that condition (3) in Definition~\ref{defn:twisted Roe} implies that
\begin{eqnarray*}
\prod_{j}^{res}A^r_{(F_j^{(n)})}( X;(G_j^{(n)}) )^\Gamma &=&\lim_{m\rightarrow \infty}\prod_{j}^{res}A^r_{ (F_j^{(n)}) } \big( (\pi^{-1}(B(y_j^{(n)},m)));(G_j^{(n)}) \big)^\Gamma \\
&=& \lim_{m\rightarrow \infty}\prod_{j}^{res}A^r_{ (F_j^{(n)}) } \big( (\Nd_m(\Gamma \cdot x_j^{(n)}));(G_j^{(n)}) \big)^\Gamma
\end{eqnarray*}
and condition (2) in Definition \ref{defn:twisted localisation} implies that
\begin{eqnarray*}
\prod_{j}^{res}A^r_{L,(F_j^{(n)})}( X;(G_j^{(n)}) )^\Gamma &=&\lim_{m\rightarrow \infty}\prod_{j}^{res}A^r_{L, (F_j^{(n)}) } \big( (\pi^{-1}(B(y_j^{(n)},m)));(G_j^{(n)}) \big)^\Gamma \\
&=& \lim_{m\rightarrow \infty}\prod_{j}^{res}A^r_{L, (F_j^{(n)}) } \big( (\Nd_m(\Gamma \cdot x_j^{(n)}));(G_j^{(n)}) \big)^\Gamma.
\end{eqnarray*}
On the other hand, note that
\[
\lim_{r\to \infty} \lim_{m\rightarrow \infty}\prod_{j}^{res}A^r_{ (F_j^{(n)}) } \big( (\Nd_m(\Gamma \cdot x_j^{(n)}));(G_j^{(n)}) \big)^\Gamma = \lim_{m\rightarrow \infty} \lim_{r\to \infty}\prod_{j}^{res}A^r_{ (F_j^{(n)}) } \big( (\Nd_m(\Gamma \cdot x_j^{(n)}));(G_j^{(n)}) \big)^\Gamma
\]
and
\[
\lim_{r\to \infty} \lim_{m\rightarrow \infty}\prod_{j}^{res}A^r_{L, (F_j^{(n)}) } \big( (\Nd_m(\Gamma \cdot x_j^{(n)}));(G_j^{(n)}) \big)^\Gamma = \lim_{m\rightarrow \infty} \lim_{r\to \infty}\prod_{j}^{res}A^r_{L, (F_j^{(n)}) } \big( (\Nd_m(\Gamma \cdot x_j^{(n)}));(G_j^{(n)}) \big)^\Gamma
\]
since these limits are just direct union of increasing subalgebras.

Consequently, in order to conclude Proposition \ref{prop:local isom. in K-theory}, it suffices to prove the following:

\begin{prop}\label{prop:local isomorphism reduction}
For each fixed $m$, the evaluation-at-one map
\[
\mathrm{ev}: \lim_{r\to \infty}\prod_{j}^{res}A^r_{L, (F_j^{(n)}) } \big( (\Nd_m(\Gamma \cdot x_j^{(n)}));(G_j^{(n)}) \big)^\Gamma \longrightarrow \lim_{r\to \infty}\prod_{j}^{res}A^r_{ (F_j^{(n)}) } \big( (\Nd_m(\Gamma \cdot x_j^{(n)}));(G_j^{(n)}) \big)^\Gamma
\]
induces an isomorphism in $K$-theory.
\end{prop}

Roughly speaking, Proposition \ref{prop:local isomorphism reduction} follows from a family version of the result that the Baum-Connes conjecture with coefficients holds for a-T-menable groups due to Higson and Kasparov \cite{HK01}.  For convenience to readers, we provide more details here.



From now on, let us fix an $m \geq 0$. To simplify the notation, denote an index set $\Lambda:=\{(n,j): n\in \NN, j\in J_n\}$. For each $\lambda=(n,j) \in \Lambda$, denote $x_\lambda = x_j^{(n)}$, $X_\lambda:=\Nd_m(\Gamma \cdot x_\lambda)$, $F_\lambda:=F_j^{(n)}$, $G_\lambda:=G_j^{(n)}$ and $\L^2_\lambda:=L^2(G_\lambda, \Cliff(E_n))$. Also denote the Hilbert spaces
\[
\H_{r,\lambda}:=\ell^2(Z_{r,n} \cap P_r(X_\lambda)) \otimes \HH \otimes \ell^2(\Gamma)
\]
and
\[
\H_{r,\lambda,E}:=\ell^2(Z_{r,n} \cap P_r(X_\lambda)) \otimes \HH \otimes \ell^2(\Gamma) \otimes \L^2_\lambda
\]
for $\lambda=(n,j) \in \Lambda$. Note that both $\H_{r,\lambda}$ and $\H_{r,\lambda,E}$ are $\Gamma$-invariant.

It is clear that $\H_{r,\lambda}$ is an admissible $P_r(X_\lambda)$-module, and $\H_{r,\lambda,E}$ is both an admissible $P_r(X_\lambda)$-module and an ample $G_\lambda$-module. We use them to build the equivariant Roe algebras $C^*(\H_{r,\lambda})^\Gamma$ and $C^*(\H_{r,\lambda,E})^\Gamma$ of $P_r(X_\lambda)$.

For $x\in X_\lambda$, write $B_{x,r,\lambda} = B_{x,r,n} \cap P_r(X_\lambda)$ and
\[
\H_{x,r,\lambda}:=\chi_{B_{x,r,\lambda}}\H_{r,\lambda}, \quad \H_{x,r,\lambda,E}=\chi_{B_{x,r,\lambda}}\H_{r,\lambda,E}.
\]
Again we represent a bounded operator $T$ on $\H_{r,\lambda}$ (respectively, $\H_{r,\lambda,E}$) as an $X_\lambda$-by-$X_\lambda$ matrix $(T_{x,y})_{x,y\in X_\lambda}$ where each $T_{x,y}$ is a bounded operator $\H_{y,r,\lambda} \to \H_{x,r,\lambda}$ (respectively, $\H_{y,r,\lambda,E} \to \H_{x,r,\lambda,E}$).

\begin{defn}\label{defn:twisted Roe version 2}
Fix an $r>0$. Let $\prod_{\lambda \in \Lambda} C_b([1,\infty), C^*(\H_{r,\lambda,E})^\Gamma)$ denote the product $C^*$-algebra of all bounded continuous functions from $[1,\infty)$ to $C^*(\H_{r,\lambda,E})^\Gamma$ with supremum norm. Write elements of this $C^*$-algebra as a collection $(T_{\lambda,s})_{\lambda \in \Lambda,s\in [1,\infty)}$ for $T_{\lambda,s}=(T_{\lambda,s,x,y})_{x,y \in X_{\lambda}} \in C^*(\H_{r,\lambda,E})^\Gamma$, whose norm is
\[
\|(T_{\lambda,s})\|=\sup_{\lambda \in \Lambda,s\in [1,\infty)}\|T_{\lambda,s}\|.
\]
Let $\CC^r[(X_\lambda); (F_\lambda)]^\Gamma$ denote the $*$-subalgebra of $\prod_{\lambda \in \Lambda} C_b([1,\infty), C^*(\H_{r,\lambda,E})^\Gamma)$ consisting of elements satisfying the following conditions:
\begin{enumerate}
		\item $\sup\limits_{s\in [1,\infty),\lambda \in \Lambda}\ppg_{P}(T_{\lambda,s})<\infty$;\\[0.1cm]
		\item for $\lambda \in \Lambda$, $\lim\limits_{s\rightarrow \infty}\ppg_{E}(T_{\lambda,s})=0$;\\[0.1cm]
		\item for $\lambda \in \Lambda$ and $x,y\in X_{\lambda}$, the map $s \mapsto T_{\lambda,s,x,y}$ belongs to the subalgebra $\K(\H_{y,r,\lambda},\H_{x,r,\lambda}) \otimes C_b\left([1,\infty), \K(\L^2_\lambda)\right)$;\\[0.1cm]
		\item for $\varepsilon>0$ and $\lambda \in \Lambda$, there exists an $s_{\lambda,\varepsilon} \in [1,\infty)$ such that for any $s>s_{\lambda,\varepsilon}$ we have:
		\[
		\supp_{E} (T_{\lambda,s}) \subseteq \Nd_\varepsilon(F_\lambda) \times \Nd_\varepsilon(F_\lambda).
		\]
\end{enumerate}
Denote $C^{*,r}( (X_\lambda); (F_\lambda))^\Gamma$ the norm-closure of $\CC^r[(X_\lambda); (F_\lambda)]^\Gamma$ in $\prod\limits_{\lambda \in \Lambda} C_b([1,\infty), C^*(\H_{r,\lambda,E})^\Gamma)$.

Also define $\CC^r_L[(X_\lambda); (F_\lambda)]^\Gamma$ to be the collection of uniformly continuous bounded functions $(T_t)$ from $[1,\infty)$ to $\CC^r[(X_\lambda); (F_\lambda)]^\Gamma$ such that the $P$-propagation of $(T_t)=(T_{t,\lambda,s})$ tends to zero as $t\to +\infty$. Denote $C^{*,r}_L( (X_\lambda); (F_\lambda))^\Gamma$ the completion of $\CC^r_L[(X_\lambda); (F_\lambda)]^\Gamma$ with respect to the norm $\|(T_t)\|:=\sup_t \|T_t\|$.
\end{defn}

For later use, let us record the following lemma. The proof follows directly from the bounded geometry of $X_\lambda$, hence omitted.
\begin{lem}\label{lem:condition (3) in twisted Roe 2}
Condition (3) in Definition \ref{defn:twisted Roe version 2} is equivalent to the following: for $\lambda \in \Lambda$ and bounded Borel subsets $K_1,K_2 \subseteq P_r(X_{\lambda})$, the map $s \mapsto \chi_{K_1}T_{\lambda,s}\chi_{K_2}$ belongs to the algebra $\K(\H_{r,\lambda}) \otimes C_b([1,\infty), \K(\L^2_\lambda))$.
\end{lem}

The following result is our motivation to introduce the above algebras:

\begin{lem}\label{lem:isom btw twisted algebras}
We have the following natural isomorphisms between $C^*$-algebras:
\begin{enumerate}[(1)]
 \item $\prod_{j}^{res}A^r_{ (F_j^{(n)}) } \big( (\Nd_m(\Gamma \cdot x_j^{(n)}));(G_j^{(n)}) \big)^\Gamma \cong C^{*,r}\big( (X_\lambda); (F_\lambda)\big)^\Gamma$;
 \item $\prod_{j}^{res}A^r_{L, (F_j^{(n)}) } \big( (\Nd_m(\Gamma \cdot x_j^{(n)}));(G_j^{(n)}) \big)^\Gamma \cong C^{*,r}_L\big( (X_\lambda); (F_\lambda)\big)^\Gamma$.
\end{enumerate}
\end{lem}

\begin{proof}
We only prove the first isomorphism, while the second follows by the same argument pointwise. Since $J_n$ is finite, $s\mapsto T_{n,s}$ is norm-continuous if and only if $s \mapsto T_{\lambda,s}$ is norm-continuous for $\lambda=(n,j) \in \Lambda$ where $T_{\lambda,s} = \chi_{G_\lambda}T_{n,s}\chi_{G_\lambda}$.

For $(T_{n,s}) \in \prod_{j}^{res}\AA^r_{ (F_j^{(n)}) } \big( (\Nd_m(\Gamma \cdot x_j^{(n)}));(G_j^{(n)}) \big)^\Gamma$, condition (3) in Definition \ref{defn:twisted Roe} says that
\begin{equation}\label{EQ:condition 3 in twisted Roe}
\lim\limits_{R'\rightarrow \infty}\sup_{s\in [1,\infty),\lambda \in \Lambda}\|\chi_{0,R'}^VT_{\lambda,s}-T_{\lambda,s}\|=\lim\limits_{R'\rightarrow \infty}\sup_{s\in [1,\infty),\lambda \in \Lambda}\|T_{\lambda,s}\chi_{0,R'}^V-T_{\lambda,s}\|=0.
\end{equation}
Since $f_n(\Nd_m(\Gamma \cdot x_j^{(n)})) \subseteq B(\xi(y_j^{(n)}), \rho_+(m))$ where $\rho_+$ is from the coarse embedding $\xi$, (\ref{EQ:condition 3 in twisted Roe}) is equivalent to the following:
\[
\lim_{R'\to \infty} \sup_{s\in [1,\infty),\lambda \in \Lambda} \|\chi_{B(\xi(y_j^{(n)}), R')} T_{\lambda,s} - T_{\lambda,s}\| = \lim_{R'\to \infty} \sup_{s\in [1,\infty),\lambda \in \Lambda} \|T_{\lambda,s}\chi_{B(\xi(y_j^{(n)}), R')} - T_{\lambda,s}\| = 0.
\]
On the other hand, note that for any $\lambda \in \Lambda$ and $s\in [1,\infty)$ we have:
\[
\supp_{E} (T_{\lambda,s}) \subseteq G_\lambda \times G_\lambda \subseteq B(\xi(y_j^{(n)}), R+1) \times B(\xi(y_j^{(n)}), R+1)
\]
where $R$ is the constant given in the assumption of Proposition \ref{prop:local isom. in K-theory}.
Hence (\ref{EQ:condition 3 in twisted Roe}) holds for $(T_{n,s})$, which concludes the proof.
\end{proof}

Consequently, to prove Proposition \ref{prop:local isomorphism reduction}, it suffices to show that the evaluation-at-one map:
\[
\mathrm{ev}: \lim_{r\to \infty} C^{*,r}_L\big( (X_\lambda); (F_\lambda)\big)^\Gamma \longrightarrow \lim_{r\to \infty}C^{*,r}\big( (X_\lambda); (F_\lambda)\big)^\Gamma
\]
induces an isomorphism in $K$-theory. To achieve, we need an extra version of the twisted algebras built on Hilbert modules.

For $\lambda \in \Lambda$, denote the $C^*$-algebra $B_\lambda$ as the norm closure in $C_b([1,\infty), \K(\L^2_\lambda))$ consisting of operators $T'=(T'_s)$ satisfying the following:
\begin{enumerate}
		\item $\lim\limits_{s\rightarrow \infty}\ppg_{E}(T'_s)=0$; \\[0.05cm]
		\item for $\varepsilon>0$ there exists an $s_{\lambda,\varepsilon} \in [1,\infty)$ such that for any $s>s_{\lambda,\varepsilon}$ we have:
		\[
		\supp_{E} (T'_s) \subseteq \Nd_\varepsilon(F_\lambda) \times \Nd_\varepsilon(F_\lambda).
		\]
\end{enumerate}
Consider the Hilbert $B_\lambda$-module:
\[
\H_{r,\lambda} \otimes B_\lambda = \ell^2(Z_{r,n} \cap P_r(X_\lambda)) \otimes \HH \otimes \ell^2(\Gamma) \otimes B_\lambda,
\]
and denote the $C^*$-algebra of adjointable morphisms on $\H_{r,\lambda} \otimes B_\lambda$ by $\L(\H_{r,\lambda} \otimes B_\lambda)$. For $T \in \L(\H_{r,\lambda} \otimes B_\lambda)$, we define its $P$-propagation as in Definition \ref{defn: Roe algebra} and also denote by $\ppg_P(T)$. Denote $\K(\H_{r,\lambda} \otimes B_\lambda)$ the $C^*$-algebra of compact morphisms on $\H_{r,\lambda} \otimes B_\lambda$. The $\Gamma$-action on $\H_{r,\lambda}$ extends to a $\Gamma$-action on $\H_{r,\lambda} \otimes B_\lambda$ by adjointable morphisms. Denote the set of $\Gamma$-invariant morphisms by $\L(\H_{r,\lambda} \otimes B_\lambda)^\Gamma$.

We consider the following:

\begin{defn}\label{defn:twisted Roe module language}
With the same notation as above, define $\A^r[ (X_\lambda); (B_\lambda) ]^\Gamma$ to be the $\ast$-subalgebra in $\prod_{\lambda \in \Lambda} \L(\H_{r,\lambda} \otimes B_\lambda)^\Gamma$ consisting of elements $T=(T_\lambda)_\lambda$ satisfying the following conditions:
\begin{enumerate}
  \item $\sup_{\lambda \in \Lambda}\ppg_{P}(T_{\lambda})<\infty$;\\[0.05cm]
  \item for $\lambda \in \Lambda$, $T_\lambda$ is \emph{locally compact} in the sense that for any bounded Borel subset $K \subseteq P_r(X_{\lambda})$, both $\chi_K T_\lambda$ and $T_\lambda \chi_K$ belong to $\K(\H_{r,\lambda} \otimes B_\lambda) \cong \K(\H_{r,\lambda}) \otimes B_\lambda$.
\end{enumerate}
Denote $\A^r( (X_\lambda); (B_\lambda) )^\Gamma$ the norm closure of $\A^r[ (X_\lambda); (B_\lambda) ]^\Gamma$ in $\prod_{\lambda \in \Lambda} \L(\HH_{r,\lambda})^\Gamma$.

Also define $\A^r_L( (X_\lambda); (B_\lambda) )^\Gamma$ to be the closure of the collection of uniformly continuous bounded functions $(T_t)$ from $[1,\infty)$ to $\A^r[(X_\lambda); (B_\lambda)]^\Gamma$ such that the $P$-propagation of $(T_t)$ tends to zero as $t\to +\infty$.
\end{defn}

We provide the following matrix version of elements in $\A^r[ (X_\lambda); (B_\lambda) ]^\Gamma$ for later use. For $\lambda \in \Lambda$, we have the following decomposition of Hilbert $B_\lambda$-modules:
\[
\H_{r,\lambda} \otimes B_\lambda = \bigoplus_{x\in X_\lambda} \H_{x,r,\lambda} \otimes B_\lambda.
\]
Given an $X_\lambda$-by-$X_\lambda$ matrix $S=(S_{x,y})_{x,y\in X_\lambda}$ with $S_{x,y} \in \K(\H_{y,r,\lambda},\H_{x,r,\lambda}) \otimes B_\lambda$ and finite propagation, we consider the map (using the same notation)
\[
S: \H_{r,\lambda} \otimes B_\lambda \longrightarrow \H_{r,\lambda} \otimes B_\lambda
\]
by matrix multiplication. It is easy to check that $S$ is an adjointable morphism on $\H_{r,\lambda} \otimes B_\lambda$ which is locally compact. Also it is obvious that any locally compact $S \in \L(\H_{r,\lambda} \otimes B_\lambda)$ with finite propagation comes from such an $X_\lambda$-by-$X_\lambda$ matrix.

In conclusion, elements in $\A^r[ (X_\lambda); (B_\lambda) ]^\Gamma$ can be written in the form of $S=(S_\lambda)$ where $S_\lambda=(S_{\lambda,x,y})_{x,y\in X_\lambda}$ is an $X_\lambda$-by-$X_\lambda$ matrix with matrix entry $S_{\lambda,x,y} \in \K(\H_{y,r,\lambda},\H_{x,r,\lambda}) \otimes B_\lambda$ such that $(S_\lambda)_\lambda$ has uniformly finite propagation. The converse holds as well.

The following lemma allows us to turn Proposition \ref{prop:local isomorphism reduction} into the setting of Hilbert modules.

\begin{lem}\label{lem:isom btw twisted Hilbert module}
For each $r>0$, we have natural $C^*$-isomorphisms:
\[
C^{*,r}\big( (X_\lambda); (F_\lambda)\big)^\Gamma \cong \A^r\big( (X_\lambda); (B_\lambda) \big)^\Gamma \quad \mbox{and} \quad C^{*,r}_L\big( (X_\lambda); (F_\lambda)\big)^\Gamma \cong \A^r_L\big( (X_\lambda); (B_\lambda) \big)^\Gamma.
\]
\end{lem}

\begin{proof}
First we define a map
\begin{eqnarray*}
\Theta: \CC^r[(X_\lambda); (F_\lambda)]^\Gamma &\longrightarrow & \prod_{\lambda \in \Lambda}\B\big( L^2([1,\infty)) \otimes \H_{r,\lambda} \otimes \L^2_\lambda\big)^\Gamma
\end{eqnarray*}
by
\[
\big(\Theta((T_{\lambda,s}))_\lambda(f \otimes \xi )\big)(s) = T_{\lambda,s}(f(s)\xi),
\]
where $f\in L^2([1,\infty))$ and $\xi \in \H_{r,\lambda} \otimes \L^2_\lambda$. It is clear that $\Theta$ is an injective $\ast$-homomorphism. Combining with the following isomorphism
\[
\B\big( L^2([1,\infty)) \otimes \H_{r,\lambda} \otimes  \L^2_\lambda \big)^\Gamma  \cong \B\big( \H_{r,\lambda} \otimes L^2([1,\infty), \L^2_\lambda)\big)^\Gamma
\]
for $\lambda \in \Lambda$, the image of $\Theta$ consists of elements $(T_\lambda)_\lambda \in \prod\limits_{\lambda \in \Lambda}\B\big( \H_{r,\lambda} \otimes L^2([1,\infty), \L^2_\lambda) \big)^\Gamma$ satisfying the following:
\begin{enumerate}
 \item $\sup_{\lambda \in \Lambda}\ppg_{P}(T_{\lambda})<\infty$;\\[0.05cm]
 \item for $\lambda \in \Lambda$ and $x,y\in X_\lambda$, the matrix entry $T_{\lambda,x,y}$ belongs to $\K(\H_{y,r,\lambda},\H_{x,r,\lambda}) \otimes B_\lambda$.
\end{enumerate}
Note that here we use the fact that for each $\lambda \in \Lambda$, the operator $T_\lambda$ can be determined by finitely many $\chi_K T \chi_{K'}$ where $K$ and $K'$ are bounded Borel subsets in $P_r(X_\lambda)$ since $T_\lambda$ is $\Gamma$-equivariant and the action on $P_r(X_\lambda)$ is cocompact (using an argument similar to that in the proof of Lemma \ref{lem:condition (4) recovery}).

Noting that the image of $\Theta$ coincides with elements in $\A^r[ (X_\lambda); (B_\lambda) ]^\Gamma$ using the matrix form introduced above, we obtain a $\ast$-isomorphism
\[
\hat{\Theta}: \CC^r[(X_\lambda); (F_\lambda)]^\Gamma \longrightarrow \A^r[ (X_\lambda); (B_\lambda) ]^\Gamma.
\]
Also note that $B_\lambda$ can be faithfully represented on $L^2([1,\infty), \L^2_\lambda)$. Hence it follows from the Hilbert module theory that $\hat{\Theta}$ is also isometric, which can be extended to a $C^*$-isomorphism (using the same notation)
\[
\hat{\Theta}: C^{*,r}\big( (X_\lambda); (F_\lambda)\big)^\Gamma \cong \A^r\big( (X_\lambda); (B_\lambda) \big)^\Gamma.
\]
Applying $\hat{\Theta}$ pointwise, we obtain a required isomorphism between twisted localisation algebras. Therefore, we conclude the proof.
\end{proof}

Therefore, to prove Proposition \ref{prop:local isomorphism reduction}, it suffices to prove the following:

\begin{prop}\label{prop: local isom for appendix}
The homomorphism
\begin{equation}\label{EQ:app_ev}
\mathrm{ev}_\ast: \lim_{r\to \infty} K_\ast\big(\A^r_L\big( (X_\lambda); (B_\lambda) \big)^\Gamma \big) \longrightarrow \lim_{r\to \infty} K_\ast \big( \A^r\big( (X_\lambda); (B_\lambda) \big)^\Gamma \big)
\end{equation}
induced by the evaluation-at-one map is an isomorphism for $\ast =0,1$.
\end{prop}

Readers might already notice that Proposition \ref{prop: local isom for appendix} is just a reformulation of a family version of the Baum-Connes conjecture with coefficients for the a-T-menable group $\Gamma$ (see, \emph{e.g.}, \cite[Section 3]{KY12}), which holds thanks to a ``uniform version'' of the proof by Higson and Kasparov \cite{HK01}. This is well-known to experts, however, we cannot find an explicit proof in literature. For convenience to readers and also for completeness, we provide a detailed proof in Appendix \ref{sec:App A} using an approach slightly different from the original one for a single space \cite{HK01}.

Consequently, we finish the proof of Proposition \ref{prop:local isomorphism reduction}, and hence conclude Theorem \ref{thm:iso. of twisted algebras in $K$-theory}.

\section{Proof of Theorem \ref{thm:main result equiv. CBC}}\label{sec:pf of main thm}

In this final section, we finish the proof of the main result.

\begin{proof}[Proof of Theorem~\ref{thm:main result equiv. CBC}]
Consider the following commutative diagram
\[
\xymatrix{
 \lim\limits_{r\rightarrow \infty}K_* \big( C^*_L(\H_r)^\Gamma \cap \prod_{n=1}^\infty C^*_L(\H_{r,n})^\Gamma \big) \ar[r] \ar[d]^>>>>>{\Ind_{F_L}} &  \lim\limits_{r\rightarrow \infty}K_* \big( C^*(\H_r)^\Gamma \cap \prod_{n=1}^\infty C^*(\H_{r,n})^\Gamma \big) \ar[d]^>>>>>{\Ind_F}\\
 \lim\limits_{r\rightarrow \infty}K_\ast ( A^r_L(X;E)^\Gamma ) \ar[r] \ar[d]^>>>>{\iota^s_*} &  \lim\limits_{r\rightarrow \infty}K_\ast ( A^r(X;E)^\Gamma ) \ar[d]^>>>>{\iota^s_*}\\
 \lim\limits_{r\rightarrow \infty}K_\ast \big( C^*_L(\H_{r,E})^\Gamma \cap \prod_{n=1}^\infty C^*_L(\H_{r,n,E})^\Gamma \big)\ar[r] &  \lim\limits_{r\rightarrow \infty}K_\ast \big( C^*(\H_{r,E})^\Gamma \cap \prod_{n=1}^\infty C^*(\H_{r,n,E})^\Gamma \big),
}
\]
where the vertical maps come from Proposition~\ref{prop:index maps} and all horizon maps are induced by evaluation-at-one maps. From Proposition~\ref{prop:index maps} again, the compositions of vertical maps are isomorphisms. The middle horizon map is an isomorphism by Theorem \ref{thm:iso. of twisted algebras in $K$-theory}, and it is clear that the upper horizon map identifies with the bottom horizon one. Therefore we obtain that both of the upper and bottom horizon maps are isomorphisms using diagram chasing. Finally combining with Corollary~\ref{cor:final reduction}, we conclude the proof.
\end{proof}

\appendix

\section{Proof of Proposition \ref{prop: local isom for appendix}}\label{sec:App A}

In this appendix, we prove a family version of the Baum-Connes conjecture with coefficients for a-T-menable groups, and equivalently concludes the proof of Proposition \ref{prop: local isom for appendix}. Here we use a slightly different approach from the original proof for a single space due to Higson and Kasparov \cite{HK01}, which also involves the localisation technique introduced in \cite{Yu97} and \cite{Yu00} (see also \cite{KY12} and \cite{FW16}).

Throughout this appendix, let us fix another separable infinite-dimensional Hilbert space $H$ (which is a different notation from the fixed Hilbert space $\HH$ in Section \ref{ssec:equiv. twisted algebras}). Also fix a left-invariant proper metric $d_\Gamma$ on the group $\Gamma$.




Let us start with some notation. First recall that by definition, $\Gamma$ being a-T-menable means that $\Gamma$ admits a metrically proper action on $H$ by isometries. Using the Mazur-Ulam Theorem \cite{MU32}, there exists a unitary representation $\pi: \Gamma \to \U(H)$ and a 1-cocycle $b: \Gamma \to H$ such that $\gamma \cdot v = \pi(\gamma) v +b(\gamma)$ and $\lim_{\gamma \to \infty} \|b(\gamma)\| = +\infty$. Here $b$ is a \emph{1-cocycle} means that $b(\gamma_1\gamma_2) = \pi(\gamma_1)b(\gamma_2) + b(\gamma_1)$ for any $\gamma_1,\gamma_2 \in \Gamma$. It is easy to see that the map $b: \Gamma \to H$ is a coarse embedding.

For each $\gamma \in \Gamma$ and $k\in \NN$, we define a finite-dimensional Euclidean affine subspace $W_k(\gamma)$ in $H$ as follows:
\begin{equation}\label{EQ:W_k(gamma)}
W_k(\gamma):= b(\gamma) + \mathrm{span}_{\CC}\left\{b(\gamma') - b(\gamma): d_\Gamma(\gamma', \gamma) \leq k^2\right\}.
\end{equation}
It is straightforward to check that $\gamma' \cdot W_k(\gamma) = W_k(\gamma'\gamma)$. Moreover, we define:
\[
W(\gamma):=\bigcup_{k\in \NN} W_k(\gamma).
\]
Since $b:\Gamma \to H$ is a coarse embedding, the space $V:=W(\gamma)$ is independent of $\gamma \in \Gamma$. Note that $b(1_\Gamma)=0$, hence $V$ is a $\Gamma$-invariant countably infinite-dimensional linear subspace in $H$. Without loss of generality, we assume that $V$ is dense in $H$.

We recall an algebra associated to $V$ introduced by Hison, Kasparov and Trout in \cite{HKT98}. Let $V_a$ be a finite-dimensional affine subspaces of $V$. Denote by $V_a^0$ the finite-dimensional linear subspace of $V$ consisting of differences of elements in $V_a$. Let $\Cliff(V_a^0)$ be the complexified Clifford algebra of $V_a^0$, and $\C(V_a):=C_0(V_a, \Cliff(V^0_a))$ be the graded $C^*$-algebra of continuous functions from $V_a$ to $\Cliff(V_a^0)$ which vanish at infinity. Let $\S:=C_0(\RR)$, graded according to odd and even functions. Define the graded tensor product
\[
\Ac(V_a):=\S ~ \hat{\otimes} ~ \C(V_a).
\]

For two finite-dimensional affine subspaces $V_a \subseteq V_b$ in $V$, we have a decomposition $V_b=V_{ba}^0+V_a$, where $V_{ba}^0$ is the orthogonal complement of $V_a^0$ in $V_b^0$. For each $v_b \in V_b$, we have a corresponding decomposition $v_b=v_{ba}+v_a$, where $v_{ba}\in V_{ba}^0$ and $v_a\in V_a$. Every function $h$ on $V_a$ can be extended to a function $\tilde{h}$ on $V_b$ by the formula $\tilde{h}(v_{ba}+v_a)=h(v_a)$.

\begin{defn}\label{defn:direct system of affine space}
For $V_a\subseteq V_b$, denote by $C_{V_b,V_a}:V_b\rightarrow \Cliff(V_b^0)$ the function $v_b\mapsto v_{ba}\in  \Cliff(V_b^0)$, where $v_{ba}$ is regarded as an element in $\Cliff(V_b^0)$ via the inclusion $V_{ba}^0\subset \Cliff(V_b^0)$. Let $X$ be the unbounded multiplier of $\S$ with degree one given by the function $t\mapsto t$.
Define a $*$-homomorphism $\beta_{V_b,V_a}:\Ac(V_a)\rightarrow \Ac(V_b)$ by the formula
\[
\beta_{V_b,V_a}(g\hat{\otimes}h)=g(X\hat{\otimes}1+1\hat{\otimes}C_{V_b,V_a})(1\hat{\otimes}\tilde{h})
\]
for $g\in \S$ and $h\in \C(V_a)$, where $g(X\hat{\otimes}1+1\hat{\otimes}C_{V_b,V_a})$ is defined by functional calculus.
\end{defn}

The maps $\beta_{V_b,V_a}$ in Definition \ref{defn:direct system of affine space} make the collection $\{\Ac(V_a)\}$ into a directed system as $V_a$ ranges over the set of all finite-dimensional affine subspaces of $V$. We define the $C^*$-algebra $\Ac(V)$ as the associated direct limit:
\[
\Ac(V):=\lim_{\longrightarrow}\Ac(V_a).
\]
We denote by $\beta_{V,V_a}: \Ac(V_a) \to \Ac(V)$ the associated homomorphism.

Endow the set $\RR_+ \times H$ with the weakest topology for which the projection $\RR_+ \times H \to H$ is weakly continuous and the function $(t,h) \mapsto t^2 + \|h\|^2$ on $\RR_+ \times H$ is continuous. It is clear that $\RR_+ \times H$ is a locally compact Hausdorff space with this topology. Also note that for $v\in H$ and $r>0$, the ball
\[
B(v,r):=\{(t,h) \in \RR_+ \times H: t^2 + \|h-v\|^2 < r^2\}
\]
is open. For each finite-dimensional affine subspace $V_a \subset V$, the center of $\Ac(V_a)$ contains $C_0(\RR_+ \times V_a)$ (by even extension). For $V_a \subset V_b$, the map $\beta_{ba}$ takes $C_0(\RR_+ \times V_a)$ into $C_0(\RR_+ \times V_b)$. It is clear that the $C^*$-algebra $\lim\limits_{\longrightarrow}C_0(\RR_+ \times V_a)$ is $\ast$-isomorphic to $C_0(\RR_+ \times H)$, where the direct limit is over the directed set of all finite-dimensional affine subspaces $V_a \subset V$. Hence the center of $\Ac(V)$ contains $C_0(\RR_+ \times H)$.

The $(\RR_+ \times H)$-support of an element $a \in \Ac(V)$, denoted by $\supp_{\RR_+ \times H}(a)$, is defined to be the complement of all $(t,h) \in \RR_+ \times H$ for which there exists $g\in C_0(\RR_+ \times H)$ such that $ag=0$ and $g(t,h) \neq 0$.

Now we consider actions on these algebras induced by the $\Gamma$-action on $H$. Let $V_a$ be a finite-dimensional affine subspace in $V$. For any $\gamma\in \Gamma$, the unitary $\pi(\gamma)$ naturally induces a homomorphism $\Cliff(V_a^0) \to \Cliff(\pi(\gamma)V_a^0)$, also denoted by $\pi(\gamma)$. This induces a homomorphism
\[
\gamma: \C(V_a) \longrightarrow \C(\gamma V_a)
\]
by
\[
\gamma(h)(v):=\pi(\gamma)h(\gamma^{-1}v)
\]
for $h\in \C(V_a)$ and $v\in \gamma V_a$, which further induces a homomorphism
\[
\gamma: \Ac(V_a) \longrightarrow \Ac(\gamma V_a) \quad \mbox{by} \quad g\hat{\otimes} h \mapsto g \hat{\otimes}  \gamma(h).
\]
Recall from \cite[Lemma 4.6]{FW16} that the following diagram commutes:
\begin{equation}\label{EQ:commutative diagram}
\xymatrix@C=4em{
			\Ac(V_a) \ar[r]^{\beta_{V_b,V_a}} \ar[d]_{\gamma} & \Ac(V_b)  \ar[d]^{\gamma} \\
			\Ac(\gamma V_a) \ar[r]^{\beta_{\gamma V_b, \gamma V_a}} & \Ac(\gamma V_b).
	}
\end{equation}
Consequently, we obtain a $\Gamma$-action on $\Ac(V)$, which makes $\Ac(V)$ into a $\Gamma$-$C^*$-algebra. Moreover, the $\Gamma$-$C^*$-algebra $\Ac(V)$ is proper in the following sense.

\begin{defn}[\cite{GHT00}]\label{defn:proper algebra}
A $\Gamma$-$C^*$-algebra $\Ac$ is called \emph{proper} if there exists a second countable, locally compact, proper $\Gamma$-space $Z$, and an equivariant $\ast$-homomorphism from $C_0(Z)$ into the center of the multiplier algebra of $\Ac$ such that $C_0(Z)\Ac$ is dense in $\Ac$. In this case, we also say that $\Ac$ is \emph{proper over $Z$}.
\end{defn}

\begin{lem}[{\cite[Proposition 4.9]{HK01}}]\label{lem:A(V) is proper}
The $\Gamma$-$C^*$-algebra $\Ac(V)$ is proper over $\RR_+ \times H$.
\end{lem}

Before we introduce the twisted algebras with $\Ac(V)$-coefficients, let us simplify some notation. Denote
\[
\beta_{k',k}(\gamma):=\beta_{W_{k'}(\gamma), W_k(\gamma)}: \Ac(W_k(\gamma)) \longrightarrow \Ac(W_{k'}(\gamma))
\]
for $k<k'$ and $\gamma \in \Gamma$, and
\[
\beta_{k}(\gamma):=\beta_{V, W_k(\gamma)}: \Ac(W_k(\gamma)) \longrightarrow \Ac(V).
\]
We also write
\[
\beta(\gamma):=\beta_0(\gamma):  \S ~\hat{\otimes}~ \CC \cong \Ac(W_0(\gamma)) \longrightarrow \Ac(V).
\]
It follows from \cite[Lemma 4.6]{FW16} that Diagram (\ref{EQ:commutative diagram}) commutes, which implies the following:

\begin{lem}\label{lem:commut beta}
For $k \in \NN$ and $\gamma' ,\gamma \in \Gamma$, we have $\gamma' \beta_{k,0}(\gamma) = \beta_{k,0}(\gamma' \gamma)$. Hence we obtain that $\gamma' \beta(\gamma) = \beta(\gamma' \gamma)$.
\end{lem}

Now let us introduce twisted algebras for the $C^*$-algebras $\A^r( (X_\lambda); (B_\lambda) )^\Gamma$ and $\A^r_L( (X_\lambda); (B_\lambda) )^\Gamma$ introduced in Definition \ref{defn:twisted Roe module language}. We follow the same notation from Section \ref{sec:local isom}. For each $r>0$, recall that we have the Hilbert $B_\lambda$-module:
\[
\H_{r,\lambda} \otimes B_\lambda = \ell^2(Z_{r,n} \cap P_r(X_\lambda)) \otimes \HH \otimes \ell^2(\Gamma) \otimes B_\lambda.
\]
We also consider the following Hilbert $B_\lambda ~\hat{\otimes}~ \Ac(V)$-module:
\[
H_{r,\lambda}:=(\H_{r,\lambda} \otimes B_\lambda) ~\hat{\otimes} ~\Ac(V) = \H_{r,\lambda} \otimes (B_\lambda ~\hat{\otimes} ~\Ac(V) ),
\]
with the extended $\Gamma$-action. For $T \in \L(H_{r,\lambda})$, its $P$-propagation can be defined as before and also denoted by $\ppg_P(T)$.

For $x\in X_\lambda$, recall that $B_{x,r,\lambda} = B_{x,r,n} \cap P_r(X_\lambda)$ and we write
\[
H_{x,r,\lambda}:=\chi_{B_{x,r,\lambda}}H_{r,\lambda} = \H_{x,r,\lambda} \otimes (B_\lambda ~\hat{\otimes} ~\Ac(V) ).
\]
Under this decomposition, any operator $S \in \L(H_{r,\lambda})$ can be written in the matrix form $S=(S_{x,y})_{x,y\in X_\lambda}$ where $S_{x,y} \in \K(\H_{y,r,\lambda},\H_{x,r,\lambda}) \otimes (B_\lambda ~\hat{\otimes} ~\Ac(V))$. We can also define the $(\RR_+ \times H)$-support for each $S_{x,y}$ as above and denote by $\supp_{\RR_+ \times H}(S_{x,y})$.

Recall that for each $\lambda$, we already fixed a point $x_\lambda \in \D \cap X_\lambda$ where $\D$ is the fundamental domain from the assumption of controlled distortion. Consider the amplified orbit map at $x_\lambda$ (which is the composition of the orbit map with an inclusion)
\[
\psi_\lambda: \Gamma \to \Gamma x_\lambda \hookrightarrow X_\lambda, \quad \gamma \mapsto \gamma \cdot x_\lambda.
\]
By the assumption of controlled distortion, the family $\{\psi_\lambda\}_{\lambda \in \Lambda}$ is uniformly coarsely equivalent. For each $\lambda \in \Lambda$, we choose a coarse inverse $\phi_\lambda: X_\lambda \to \Gamma$ to $\psi_\lambda$ such that the family $\{\phi_\lambda\}_\lambda$ is uniformly coarsely equivalent as well.


\begin{defn}\label{defn:double twisted Roe}
Define $\A^r[(X_\lambda); (B_\lambda \hat{\otimes} \Ac(V))]^\Gamma$ to be the $\ast$-subalgebra in $\prod_{\lambda \in \Lambda} \L(H_{r,\lambda})^\Gamma$ consisting of elements $T=(T_\lambda)_\lambda$ (writing $T_\lambda = (T_{\lambda,x,y})_{x,y\in X_\lambda}$) satisfying the following conditions:
\begin{enumerate}
  \item $\sup_{\lambda \in \Lambda}\ppg_{P}(T_{\lambda})<\infty$;\\[0.05cm]
  \item for $\lambda \in \Lambda$ and bounded Borel subset $B \subseteq P_r(X_{\lambda})$, both $\chi_B T_\lambda$ and $T_\lambda \chi_B$ belong to $\K(H_{r,\lambda})$;\\[0.05cm]
  \item there exists a compact subset $K$ of $\RR_+\times H$ such that $\supp_{\RR_+ \times H} (T_{\lambda,x,y})$ is contained in $\phi_\lambda(x) \cdot K$ for any $\lambda\in \Lambda$ and $x,y\in X_\lambda$;\\[0.05cm]
  \item there exists an integer $N\in \NN$ such that
  \[
  T_{\lambda,x,y} \in \big(\Id \hat{\otimes}\beta_N(\phi_\lambda(x))\big)\left((\K(\H_{y,r,\lambda},\H_{x,r,\lambda}) \otimes B_\lambda)~\hat{\otimes}~\Ac(W_N(\phi_\lambda(x)))\right)
  \]
  for any $\lambda\in \Lambda$ and $x,y\in X_\lambda$;\\[0.05cm]
  \item there exists $c>0$ such that if $T_{\lambda,x,y}=\big(\Id \hat{\otimes}\beta_N(\phi_\lambda(x))\big)(T'_{\lambda,x,y})$ then for any $Y=(s,h)\in \RR\times W_N(\phi_\lambda(x))$ with norm $1$, the derivative $D_Y(T'_{\lambda,x,y})$ in the direction of $Y$ of the function
\[
T'_{\lambda,x,y}:\RR\times W_N(\gamma)\rightarrow \left(\K(\H_{y,r,\lambda},\H_{x,r,\lambda}) \otimes B_\lambda\right)~\hat{\otimes}~ \Cliff(W_N^0(\phi_\lambda(x)))
\]
exists and $\|D_Y(T'_{\lambda,x,y})\|\leq c$ for all $x,y\in X_\lambda$.
\end{enumerate}
Define $\A^r((X_\lambda); (B_\lambda \hat{\otimes} \Ac(V)))^\Gamma$ to be the norm closure of $\A^r[(X_\lambda); (B_\lambda \hat{\otimes} \Ac(V))]^\Gamma$ in $\prod_{\lambda \in \Lambda} \L(H_{r,\lambda})$.

Also define $\A^r_L[(X_\lambda); (B_\lambda \hat{\otimes} \Ac(V))]^\Gamma$ to be the collection of uniformly continuous bounded functions $(T_t)$ from $[1,\infty)$ to $\A^r[(X_\lambda); (B_\lambda \hat{\otimes} \Ac(V))]^\Gamma$ such that the $P$-propagation of $(T_t)$ tends to zero as $t\to +\infty$ and the parameters in condition (3), (4) and (5) above can be chosen to be independent of $t$. Define $\A^r_L((X_\lambda); (B_\lambda \hat{\otimes} \Ac(V)))^\Gamma$ to be the norm closure of $\A^r_L[(X_\lambda); (B_\lambda \hat{\otimes} \Ac(V))]^\Gamma$.
\end{defn}


We aim to prove the following result:

\begin{prop}\label{prop:local isomorphism with A(V)}
The homomorphism
\[
\ev_\ast: \lim_{r\to \infty} K_\ast \big(\A^r_L\big((X_\lambda); (B_\lambda \hat{\otimes} \Ac(V))\big)^\Gamma\big) \longrightarrow \lim_{r\to \infty} K_\ast \big(\A^r\big((X_\lambda); (B_\lambda \hat{\otimes} \Ac(V))\big)^\Gamma \big)
\]
induced by the evaluation-at-one map is an isomorphism for $\ast=0,1$.
\end{prop}

The proof, has its root in \cite[Section 6]{Yu00}, relies on a Mayer-Vietoris argument on the proper $C^*$-algebra $\Ac(V)$ to chop the twisted algebras into smaller pieces. First we need to introduce some notation.

Given a $\Gamma$-invariant open subset $O \subseteq \RR_+ \times H$, we denote
\[
\Ac(V)_O:=\overline{C_0(O) \Ac(V)},
\]
which is a two-sided $\ast$-ideal in $\Ac(V)$. Define $\A^r[(X_\lambda); (B_\lambda \hat{\otimes} \Ac(V)_O)]^\Gamma$ to be the subalgebra of $\A^r[(X_\lambda); (B_\lambda \hat{\otimes} \Ac(V))]^\Gamma$ consisting of $T=(T_{\lambda,x,y})$ such that $\supp_{\RR_+ \times H} (T_{\lambda,x,y}) \subseteq O$, \emph{i.e.}, $T_{\lambda,x,y} \in \K(\H_{y,r,\lambda},\H_{x,r,\lambda}) \otimes B_\lambda ~\hat{\otimes}~ \Ac(V)_O$ for any $\lambda \in \Lambda$ and $x,y\in X_\lambda$. Define $\A^r((X_\lambda); (B_\lambda \hat{\otimes} \Ac(V)_O))^\Gamma$ to be the norm closure of $\A^r[(X_\lambda); (B_\lambda \hat{\otimes} \Ac(V)_O)]^\Gamma$.
It is clear that $\A^r((X_\lambda); (B_\lambda \hat{\otimes} \Ac(V)_O))^\Gamma$ is a two-sided $\ast$-ideal in $\A^r((X_\lambda); (B_\lambda \hat{\otimes} \Ac(V)))^\Gamma$.

Similarly, define $\A_L^r((X_\lambda); (B_\lambda \hat{\otimes} \Ac(V)_O))^\Gamma$ to be the norm closure of the subalgebra in $\A^r_L[(X_\lambda); (B_\lambda \hat{\otimes} \Ac(V))]^\Gamma$ consisting of elements $T=(T_{t,\lambda,x,y})$ such that $\supp_{\RR_+ \times H} (T_{t,\lambda,x,y}) \subseteq O$ for any $\lambda \in \Lambda$, $t\in [1,\infty)$ and $x,y\in X_\lambda$. It is also clear that $\A_L^r((X_\lambda); (B_\lambda \hat{\otimes} \Ac(V)_O))^\Gamma$ is a two-sided $\ast$-ideal in $\A_L^r((X_\lambda); (B_\lambda \hat{\otimes} \Ac(V)))^\Gamma$.



Concerning the proper $\Gamma$-action on $\RR_+ \times H$, it is known (see, \emph{e.g.}, \cite[Appendix A.2]{willett2020higher}) that  for each $(t,v) \in \RR_+ \times H$, there exists an open precompact neighbourhood $U$ of $(t,v)$ such that
\[
\Gamma \cdot U = \Gamma \times_F U := \bigsqcup_{i \in I} \gamma_i \cdot U,
\]
where $F$ is the stabiliser of $(t,v)$ (which is a finite subgroup since the action is proper) and $\gamma_i$ runs over a set of representatives of the left cosets $\{\gamma F: \gamma \in \Gamma\}$. In this case, we say that $\Gamma \cdot U$ is a \emph{slice} of the action.

The following result is a key step to attack Proposition \ref{prop:local isomorphism with A(V)}.

\begin{lem}\label{lem:local isomorphism with A(V) slice case}
For a slice $O=\Gamma \cdot U$, we have that
\[
\ev_*: \lim_{r\to \infty} K_*\left(\A_L^r\big((X_\lambda); (B_\lambda \hat{\otimes} \Ac(V)_O)\big)^\Gamma\right) \longrightarrow \lim_{r\to \infty} K_*\left(\A^r\big((X_\lambda); (B_\lambda \hat{\otimes} \Ac(V)_O)\big)^\Gamma\right)
\]
is an isomorphism for $\ast = 0,1$.
\end{lem}

\begin{proof}
By assumption, we write
\[
O=\Gamma \cdot U = \bigsqcup_{i \in I} \gamma_i \cdot U
\]
where $U$ is an open precompact subset in $\RR_+ \times H$, $F$ is a finite subgroup such that $F \cdot U = U$, and $\gamma_i$ runs over a set of representatives of the left cosets $\{\gamma F: \gamma \in \Gamma\}$. Fix $i_0 \in I$ such that $\gamma_{i_0} U=U$ (\emph{i.e.}, $\gamma_{i_0} \in F$).

Given $T=(T_{\lambda,x,y}) \in \A^r[(X_\lambda); (B_\lambda \hat{\otimes} \Ac(V)_O)]^\Gamma$, we write $T_{\lambda,x,y} = (T_{\lambda,x,y,i})_{i\in I}$ where $T_{\lambda,x,y,i} = T_{\lambda,x,y} \cdot \chi_{\gamma_i U}$. For $\gamma \in \Gamma$, we have $\gamma \cdot T_{\lambda,x,y} = T_{\lambda,\gamma x,\gamma y}$. Moreover, each $\gamma$ induces a permutation $\sigma_\gamma: I \to I$ such that $\gamma \gamma_i U = \gamma_{\sigma_\gamma(i)} U$. Hence we obtain:
\begin{equation}\label{EQ:reduction}
T_{\lambda,x,y,\sigma_\gamma^{-1}(i)} = \gamma^{-1} \cdot T_{\lambda,\gamma x, \gamma y, i} \quad  \mbox{for} \quad  i\in I.
\end{equation}
Therefore, $T_{\lambda,x,y}$ can be determined by the set $\{T_{\lambda,\gamma x, \gamma y, i_0}: \gamma\in \Gamma \}$.

Recall that there exists a compact $K \subset \RR_+ \times H$ such that $\supp_{\RR_+ \times H} (T_{\lambda,x,y})$ is contained in $\phi_\lambda(x) \cdot K$. Since $K$ is bounded and the action is proper, there exists a finite subset $I_0 \subset I$ such that $K\cap (\gamma_i U) = \emptyset$ only for $i \notin I_0$. Hence
\[
\supp_{\RR_+ \times H} (T_{\lambda,x,y}) ~\subseteq~ \bigsqcup_{i\in I_0} ~\phi_\lambda(x)\gamma_i \cdot  U \quad \mbox{~for~any~}  x,y\in X_\lambda.
\]
Assuming that $T_{\lambda,\gamma x, \gamma y, i_0} \neq 0$, then there exists $i\in I_0$ such that $\gamma_{i_0} \in \phi_\lambda(\gamma x) \gamma_i F$. Recall that $\phi_\lambda(\gamma x) x_\lambda$ is uniformly close to $\gamma x$ and $T_\lambda$ has uniformly finite propagation, hence there exists $R>0$ such that $\gamma x, \gamma y \in B(x_\lambda,R)$. In conclusion, we obtain that $T_{\lambda,x,y}$ can be determined by the set $\{T_{\lambda,x', y', i_0}: x',y' \in B(x_\lambda,R)\}$.

For $R>0$, define $\A^r((F \cdot B(x_\lambda,R)); (B_\lambda \hat{\otimes} \Ac(V)_U))^F$ and $\A^r_L((F \cdot B(x_\lambda,R)); (B_\lambda \hat{\otimes} \Ac(V)_U))^F$ similar to those in Definition \ref{defn:double twisted Roe} except that here we only require operators are $F$-invariant (instead of $\Gamma$-invariant) and their $(\RR_+ \times H)$-supports are in $U$. Also define a map
\[
\Psi_R: \A^r[(F \cdot B(x_\lambda,R)); (B_\lambda \hat{\otimes} \Ac(V)_U)]^F \longrightarrow \A^r[(X_\lambda); (B_\lambda \hat{\otimes} \Ac(V)_O)]^\Gamma
\]
by setting the image of $T=(T_{\lambda, x',y'})$ to be
\[
\Psi_R(T)_{\lambda,x,y,i}:=
\begin{cases}
	~\gamma^{-1} \cdot T_{\lambda,\gamma x, \gamma y}, & \mbox{if~} \exists~ \gamma \in \Gamma: \gamma x, \gamma y \in F \cdot B(x_\lambda,R) \mbox{~and~} \sigma_\gamma(i) = i_0;  \\[0.3cm]
	~0, & \mbox{otherwise},
	\end{cases}
\]
inspired by Equation (\ref{EQ:reduction}).
We claim that $\Psi_R(T)$ is well-defined. In fact, if there exists another $\gamma' \in \Gamma$ such that $\sigma_{\gamma'}(i) = i_0$ and $\gamma' x, \gamma' y \in F \cdot B(x_\lambda,R)$, then it follows from the definition that
\[
\gamma \gamma_i \cdot U = U = \gamma' \gamma_i \cdot U,
\]
which implies that $\gamma' = g \gamma$ for some $g\in F$. Note that in this case $\gamma x \in F \cdot B(x_\lambda,R)$ implies $\gamma' x \in F \cdot B(x_\lambda,R)$ as well. Moreover, the $F$-invariance of $T$ implies that
\[
(\gamma')^{-1} \cdot T_{\lambda, \gamma' x, \gamma' y} = \gamma^{-1} g^{-1} \cdot T_{\lambda, g \cdot \gamma x, g \cdot \gamma y} = \gamma^{-1} \cdot T_{\lambda, \gamma x, \gamma y}.
\]
Hence we conclude the claim.

Now we check that condition (1)-(5) in Definition \ref{defn:double twisted Roe} holds for $\Psi_R(T)$. First note that $\Psi_R(T)_{\lambda,x,y} \neq 0$ implies that there exists $\gamma \in \Gamma$ such that $\gamma x, \gamma y \in F\cdot B(x_\lambda,R)$. This implies that $\Psi_R(T)$ has uniformly bounded propagation. For the $\Gamma$-invariance property, we need to check that
\[
\hat{\gamma} \cdot \Psi(T)_{\lambda,x,y,\sigma_{\hat{\gamma} }^{-1}(i)} = \Psi(T)_{\lambda,\hat{\gamma}  x, \hat{\gamma}  y,i} \mbox{~for~any~} \hat{\gamma}  \in \Gamma, x,y\in X_\lambda \mbox{~and~} i \in I.
\]
By definition, taking $\gamma \in \Gamma$ such that $\gamma x, \gamma y \in F \cdot B(x_\lambda,R)$ and $\sigma_\gamma(\sigma_{\hat{\gamma} }^{-1}(i)) = i_0$ then
\[
\Psi(T)_{\lambda,x,y,\sigma_{\hat{\gamma} }^{-1}(i)} = \gamma^{-1} \cdot T_{\lambda, \gamma x, \gamma y}.
\]
Note that $\sigma_\gamma(\sigma_{\hat{\gamma} }^{-1}(i)) = i_0$ implies that $\gamma \hat{\gamma}^{-1} \gamma_i \in F$. Since $T$ is $F$-invariant, we obtain
\begin{equation}\label{EQ: key lemma 1}
\gamma_i^{-1} \hat{\gamma} \cdot \Psi(T)_{\lambda,x,y,\sigma_{\hat{\gamma} }^{-1}(i)} = \gamma_i^{-1}\hat{\gamma} \gamma^{-1} \cdot T_{\lambda, \gamma x, \gamma y} = T_{\lambda, \gamma_i^{-1}\hat{\gamma} x, \gamma_i^{-1}\hat{\gamma} y}.
\end{equation}
On the other hand, $\gamma_i^{-1} \hat{\gamma} x= (\gamma_i^{-1} \hat{\gamma} \gamma^{-1}) \gamma x \in F \cdot B(x_\lambda, R)$ and similarly $\gamma_i^{-1} \hat{\gamma} y \in F \cdot B(x_\lambda, R)$. Also note that $\sigma_{\gamma_i}(i_0) = i$, hence we have:
\begin{equation}\label{EQ: key lemma 2}
\Psi(T)_{\lambda,\hat{\gamma}  x, \hat{\gamma}  y,i} = \gamma_i \cdot T_{\lambda, \gamma_i^{-1}\hat{\gamma} x, \gamma_i^{-1}\hat{\gamma} y}.
\end{equation}
Combining Equation (\ref{EQ: key lemma 1}) and (\ref{EQ: key lemma 2}), we conclude that each $\Psi(T)_\lambda$ is $\Gamma$-invariant.

Condition (2) follows from the fact that for each $\lambda,x,y$, there exist only finitely many $\gamma \in \Gamma$ such that $\gamma x, \gamma y \in F\cdot B(x_\lambda, R)$. This implies that there exists only finitely many $i\in I$ such that $\Psi_R(T)_{\lambda,x,y,i} \neq 0$, hence provides condition (2). Condition (4) and (5) hold trivially due to the construction, and finally we check condition (3).

Note that $\Phi_R(T)_{\lambda,x,y,i} \neq 0$ implies that there exists $\gamma \in \Gamma$ such that $\gamma x \in F \cdot B(x_\lambda, R)$ and $\sigma_\gamma(i) = i_0$. Hence we have $\gamma \gamma_i \in F$. Also note that $\gamma \phi_\lambda(x) x_\lambda$ is uniformly close to $\gamma x$, which is uniformly close to $x_\lambda$. Thanks to the uniformly coarse embedding of orbit maps, we obtain that $\gamma \phi_\lambda(x)$ is uniformly close to the identity, which further implies that $\gamma_i$ is uniformly close to $\phi_\lambda(x)$. Hence we obtain:
\[
\supp_{\RR_+ \times H}(\Psi_R(T)_{\lambda,x,y}) \subseteq B(\phi_\lambda(x), R') \cdot U
\]
for some uniform constant $R'>0$. Since $U$ is precompact, there exists a compact subset $K \subset \RR_+ \times H$ such that $B(\phi_\lambda(x), R') \cdot U \subseteq \phi_\lambda(x) \cdot K$ for any $x\in X_\lambda$, which concludes condition (3).

In conclusion, we obtain that $\Psi_R$ is a well-defined map. It is also straightforward to check that $\Psi_R$ is a $\ast$-homomorphism (details are omitted here). Furthermore, we claim that $\Psi_R$ is isometric. Indeed, given $T \in \A^r[(F \cdot B(x_\lambda,R)); (B_\lambda \hat{\otimes} \Ac(V)_U)]^F$ we have
\[
\|\Psi_R(T)\| = \sup_{\lambda\in \Lambda, i\in I} \|(\Psi_R(T)_{\lambda,x,y,i})_{x,y}\|
\]
Direct calculation shows that
\[
\Psi_R(T)_{\lambda,x,y,i} =
\begin{cases}
	~\gamma_i \cdot T_{\lambda, \gamma_i^{-1} x, \gamma_i^{-1} y}, & \mbox{if~} x,y\in \gamma_i F \cdot B(x_\lambda,R);  \\[0.3cm]
	~0, & \mbox{otherwise}
	\end{cases}
\]
It follows that $\|(\Psi_R(T)_{\lambda,x,y,i})\| = \|T_\lambda\|$ for each $i\in I$, which implies that $\|\Psi_R\|$ is isometric. Therefore, the map $\Psi_R$ can be extended to an isometric $\ast$-homomorphism (still denoted by $\Psi_R$)
\[
\Psi_R: \A^r\big((F \cdot B(x_\lambda,R)); (B_\lambda \hat{\otimes} \Ac(V)_U)\big)^F \longrightarrow \A^r\big((X_\lambda); (B_\lambda \hat{\otimes} \Ac(V)_O)\big)^\Gamma.
\]

Taking direct limits, we obtain a homomorphism
\[
\Psi: \lim_{R\to \infty}\A^r\big((F \cdot B(x_\lambda,R)); (B_\lambda \hat{\otimes} \Ac(V)_U)\big)^F \longrightarrow \A^r\big((X_\lambda); (B_\lambda \hat{\otimes} \Ac(V)_O)\big)^\Gamma.
\]
The analysis at the beginning shows that $\Psi$ is surjective, hence a $\ast$-isomorphism. Similarly on the localisation level, we also have a $\ast$-isomorphism
\[
\Psi_L: \lim_{R\to \infty}\A^r_L\big((F \cdot B(x_\lambda,R)); (B_\lambda \hat{\otimes} \Ac(V)_U)\big)^F \longrightarrow \A^r_L\big((X_\lambda); (B_\lambda \hat{\otimes} \Ac(V)_O)\big)^\Gamma.
\]
Note that
\[
\lim_{r\to \infty}\lim_{R\to \infty}\A^r\big((F \cdot B(x_\lambda,R)); (B_\lambda \hat{\otimes} \Ac(V)_U)\big)^F = \lim_{R\to \infty}\lim_{r\to \infty}\A^r\big((F \cdot B(x_\lambda,R)); (B_\lambda \hat{\otimes} \Ac(V)_U)\big)^F
\]
and
\[
\lim_{r\to \infty}\lim_{R\to \infty}\A^r_L\big((F \cdot B(x_\lambda,R)); (B_\lambda \hat{\otimes} \Ac(V)_U)\big)^F = \lim_{R\to \infty}\lim_{r\to \infty}\A^r_L\big((F \cdot B(x_\lambda,R)); (B_\lambda \hat{\otimes} \Ac(V)_U)\big)^F,
\]
since all of these limits are closures of unions of increasing algebras. Hence it suffices to prove that for any $R>0$, the following is an isomorphism for $\ast=0,1$:
\begin{small}
\[
\ev_\ast: \lim_{r\to \infty} K_\ast \left( \A^r_L\big((F \cdot B(x_\lambda,R)); (B_\lambda \hat{\otimes} \Ac(V)_U)\big)^F \right)  \longrightarrow \lim_{r\to \infty} K_\ast \left( \A^r\big((F \cdot B(x_\lambda,R)); (B_\lambda \hat{\otimes} \Ac(V)_U)\big)^F \right).
\]
\end{small}
Also note that for sufficiently large $r$, the Rips complex $P_r(F \cdot B(x_\lambda,R))$ is just a simplex which is $F$-invariant, denoted by $\Delta_\lambda$. Hence it suffices to prove
\[
\ev_\ast: \lim_{r\to \infty} K_\ast \left( \A^r_L\big((\Delta_\lambda); (B_\lambda \hat{\otimes} \Ac(V)_U)\big)^F \right)  \longrightarrow \lim_{r\to \infty} K_\ast \left( \A^r\big((\Delta_\lambda); (B_\lambda \hat{\otimes} \Ac(V)_U)\big)^F \right)
\]
is an isomorphism for $\ast=0,1$. Since $\{\Delta_\lambda\}_{\lambda}$ is $F$-equivariantly uniformly coarsely equivalent to $\{x_\lambda^0\}_\lambda$ where $x_\lambda^0$ is the barycenter of $\Delta_\lambda$, we obtain
\[
K_\ast \left( \A^r\big((\Delta_\lambda); (B_\lambda \hat{\otimes} \Ac(V)_U)\big)^F \right) \cong K_\ast \left( \A^r\big((\{x_\lambda^0\}); (B_\lambda \hat{\otimes} \Ac(V)_U)\big)^F \right).
\]
Also $\{\Delta_\lambda\}_{\lambda}$ is $F$-equivariantly uniformly strongly Lipschitz homotopic equivalent to $\{x_\lambda^0\}_\lambda$, hence a similar argument as for \cite[Proposition 3.7]{Yu97} shows that
\[
K_\ast \left( \A^r_L\big((\Delta_\lambda); (B_\lambda \hat{\otimes} \Ac(V)_U)\big)^F \right) \cong K_\ast \left( \A^r_L\big((\{x_\lambda^0\}); (B_\lambda \hat{\otimes} \Ac(V)_U)\big)^F \right).
\]
Therefore, it remains to check that
\[
\ev_\ast: \lim_{r\to \infty} K_\ast \left( \A^r_L\big((\{x_\lambda^0\}); (B_\lambda \hat{\otimes} \Ac(V)_U)\big)^F \right)  \longrightarrow \lim_{r\to \infty} K_\ast \left( \A^r\big((\{x_\lambda^0\}); (B_\lambda \hat{\otimes} \Ac(V)_U)\big)^F \right)
\]
is an isomorphism. This follows directly from \cite[Lemma 12.4.3]{willett2020higher}, hence we conclude the proof.
\end{proof}

\begin{proof}[Proof of Proposition \ref{prop:local isomorphism with A(V)}]
First we claim that for any bounded open subset $B \subset \RR_+ \times H$, the map
\[
\ev_*: \lim_{r\to \infty} K_*\left(\A_L^r\big((X_\lambda); (B_\lambda \hat{\otimes} \Ac(V)_{\Gamma \cdot B})\big)^\Gamma\right) \longrightarrow \lim_{r\to \infty} K_*\left(\A^r\big((X_\lambda); (B_\lambda \hat{\otimes} \Ac(V)_{\Gamma \cdot B})\big)^\Gamma\right)
\]
is an isomorphism for $\ast=0,1$. In fact, we can use finitely many slices to cover $\Gamma \cdot B$ since $\bar{B}$ is compact. For two slices $W_1=\Gamma \times_{F_1} U_1$ and $W_2=\Gamma \times_{F_2} U_2$, note that
\[
W_1 \cap W_2 = \left( \bigsqcup_{i\in I} \gamma_i \cdot U_1 \right) \cap W_2 =\bigsqcup_{i\in I} \left( \gamma_i U_1  \cap W_2 \right) = \bigsqcup_{i\in I} \gamma_i \cdot \left(  U_1  \cap W_2 \right).
\]
On the other hand for any $\gamma \in F_1$, we have
\[
\gamma \cdot (U_1 \cap W_2) = \gamma U_1 \cap \gamma W_2 = U_1 \cap W_2.
\]
Hence we obtain that
\[
W_1 \cap W_2 = \Gamma \times_{F_1} (U_1 \cap W_2),
\]
which is also a slice. Therefore, Lemma \ref{lem:local isomorphism with A(V) slice case} together with a Mayer-Vietoris argument (similar to \cite[Lemma 6.3]{Yu00}) concludes the claim.

Finally, for each $n\in \NN$ we set $B_n$ to be the open ball in $\RR_+ \times H$ with center $(0,0)$ and radius $n$. From condition (3) in Definition \ref{defn:double twisted Roe}, we obtain:
\[
\A^r\big((X_\lambda); (B_\lambda \hat{\otimes} \Ac(V))\big)^\Gamma = \lim_{n\to \infty} \A^r\big((X_\lambda); (B_\lambda \hat{\otimes} \Ac(V)_{\Gamma \cdot B_n})\big)^\Gamma
\]
and
\[
\A_L^r\big((X_\lambda); (B_\lambda \hat{\otimes} \Ac(V))\big)^\Gamma = \lim_{n\to \infty} \A_L^r\big((X_\lambda); (B_\lambda \hat{\otimes} \Ac(V)_{\Gamma \cdot B_n})\big)^\Gamma.
\]
Hence we finish the proof thanks to the claim above.
\end{proof}

Our next aim is to construct the Bott and Dirac maps. First let us introduce the Bott maps. For each $\lambda \in \Lambda$, we choose a fundamental domain $\D_\lambda \subseteq X_\lambda$ for the $\Gamma$-action on $X_\lambda$ such that $x_\lambda \in \D_\lambda$ and
\[
D=\sup_{\lambda\in \Lambda} \diam(\D_\lambda) < +\infty.
\]
Note that $\D_\lambda$ might not coincide with $\D \cap X_\lambda$ where $\D$ is the fixed fundamental domain from the condition of controlled distortion\footnote{We would like to point out that $\D$ can be modified to satisfy that $\sup_{\lambda \in \Lambda} \diam(\D \cap X_\lambda)$ is finite, whence we can choose $D_\lambda$ to be $\D \cap X_\lambda$. However, this is not necessary for us to continue.}.

\begin{defn}\label{defn:Bott map}
For $t \geq 1$, we define a map
\[
\beta_t: \S ~\hat{\otimes}~ \A^r[ (X_\lambda); (B_\lambda) ]^\Gamma \longrightarrow  \A^r((X_\lambda); (B_\lambda \hat{\otimes} \Ac(V)))^\Gamma
\]
by setting
\[
\beta_t\left( g ~\hat{\otimes}~(T_{\lambda,x,y}) \right)_{\lambda,\gamma x, y}:= \frac{1}{|\Gamma_x|}\sum_{\gamma'\in \Gamma_x}  T_{\lambda,\gamma x, y}~\hat{\otimes}~\beta(\gamma \gamma')(g_t)
\]
for $x\in \D_\lambda, y\in X_\lambda$ and $\gamma\in \Gamma$, where $\Gamma_x$ is the stabiliser at $x$.
\end{defn}

\begin{lem}\label{lem:Bott well-defined}
For each $t \geq 1$, the map $\beta_t$ is well-defined.
\end{lem}

\begin{proof}
If $\gamma_1 x = \gamma_2 x$, then there exists $\hat{\gamma} \in \Gamma_x$ such that $\gamma_2 = \gamma_1 \cdot \hat{\gamma}$. Hence:
\begin{align*}
\beta_t\left( g \hat{\otimes}(T_{\lambda,x,y}) \right)_{\lambda,\gamma_2 x, y} &= \frac{1}{|\Gamma_x|}\sum_{\gamma'\in \Gamma_x}  T_{\lambda,\gamma_2 x, y}\hat{\otimes}\beta(\gamma_2 \gamma')(g_t) = \frac{1}{|\Gamma_x|}\sum_{\gamma'\in \Gamma_x}  T_{\lambda,\gamma_1 x, y}\hat{\otimes}\beta(\gamma_1 \hat{\gamma} \gamma')(g_t)\\
&= \frac{1}{|\Gamma_x|}\sum_{\gamma''\in \Gamma_x}  T_{\lambda,\gamma_1 x, y}\hat{\otimes}\beta(\gamma_1 \gamma'')(g_t)= \beta_t\left( g \hat{\otimes}(T_{\lambda,x,y}) \right)_{\lambda,\gamma_1 x, y}.
\end{align*}
Hence $\beta_t$ is a well-defined map.

We need to check that the range of $\beta_t$ is contained in $\A^r((X_\lambda); (B_\lambda \hat{\otimes} \Ac(V)))^\Gamma$. Given $g ~\hat{\otimes}~(T_{\lambda,x,y}) \in \S ~\hat{\otimes}~ \A^r[ (X_\lambda); (B_\lambda) ]^\Gamma$, we take a sequence of bounded smooth functions $g_n$ on $\RR$ with bounded derivatives and a sequence of even smooth functions $h_n$ on $\RR$ with range in $[0,1]$, support in $[-n,n]$ and $\|h_n'\| \leq 1$ such that $\|g_n h_n - g\|_\infty \to 0$ as $n\to \infty$. We define $\beta_t^{(n)}\left( g ~\hat{\otimes}~(T_{\lambda,x,y}) \right)$ by setting
\[
\beta_t^{(n)}\left( g ~\hat{\otimes}~(T_{\lambda,x,y}) \right)_{\lambda,\gamma x, y}:= \frac{1}{|\Gamma_x|}\sum_{\gamma'\in \Gamma_x}  T_{\lambda,\gamma x, y}~\hat{\otimes}~\beta(\gamma \gamma')((g_n h_n)_t)
\]
for $x\in \D_\lambda, y\in X_\lambda$ and $\gamma\in \Gamma$. Thanks to the finite propagation of $(T_{\lambda,x,y})$, we have
\[
\left\|\beta_t^{(n)}\left( g ~\hat{\otimes}~(T_{\lambda,x,y}) \right) - \beta_t\left( g ~\hat{\otimes}~(T_{\lambda,x,y}) \right)\right\| \to 0 \quad \mbox{as} \quad n \to \infty.
\]
Hence it suffices to show that $\beta_t^{(n)}\left( g ~\hat{\otimes}~(T_{\lambda,x,y}) \right)$ belongs to $\A^r[(X_\lambda); (B_\lambda \hat{\otimes} \Ac(V))]^\Gamma$.

Let us fix an $n\in \NN$. It is clear that the $P$-propagation of $\beta_t^{(n)}\left( g ~\hat{\otimes}~(T_{\lambda,x,y}) \right)$ is the same as that of $(T_{\lambda,x,y})$, which deduces condition (1). For $\hat{\gamma}, \gamma \in \Gamma$, we have:
\begin{align*}
\beta_t^{(n)}\left( g ~\hat{\otimes}~ (T_{\lambda,x,y}) \right)_{\lambda,\hat{\gamma} \gamma x, \hat{\gamma} y} &= \frac{1}{|\Gamma_x|}\sum_{\gamma'\in \Gamma_x}  T_{\lambda,\hat{\gamma} \gamma x, \hat{\gamma} y} ~\hat{\otimes}~\beta(\hat{\gamma} \gamma \gamma')((g_n h_n)_t)\\
& = \frac{1}{|\Gamma_x|}\sum_{\gamma'\in \Gamma_x} \hat{\gamma} \cdot T_{\lambda,\gamma x, y} ~\hat{\otimes}~ \hat{\gamma}\cdot\left(\beta(\gamma \gamma')((g_n h_n)_t)\right)\\
&= \hat{\gamma} \cdot \left(\beta_t^{(n)}\left( g ~\hat{\otimes}~ (T_{\lambda,x,y}) \right)_{\lambda,\gamma x, y}\right)
\end{align*}
where we use Lemma \ref{lem:commut beta} in the second equation. Hence each $\beta_t^{(n)}\left( g \hat{\otimes}(T_{\lambda,x,y}) \right)_{\lambda}$ is $\Gamma$-invariant. It is also clear that condition (2) holds.

Concerning condition (3): note that $\beta(\gamma \gamma')((g_n h_n)_t) = \beta(\gamma \gamma')((g_n)_t) \cdot \beta(\gamma \gamma')((h_n)_t)$.
Hence the $(\RR_+ \times H)$-support of $\beta_t\left( g ~\hat{\otimes}~(T_{\lambda,x,y}) \right)_{\lambda,\gamma x, y}$ is contained in that of $\beta(\gamma \gamma')(h_n)$. Since $h_n$ is even, it follows from direct calculation that the $(\RR_+ \times H)$-support of $\beta(\gamma \gamma')((h_n)_t)$ is contained in:
\[
[-nt,nt] \times \left\{h\in H: \|h-b(\gamma \gamma')\| \leq nt\right\}.
\]
On the other hand, it is straightforward to check that $\gamma \gamma'$ and $\phi_\lambda(\gamma x)$ are uniformly close. Hence there exists a constant $K_{n,t}>0$ such that
\[
[-nt,nt] \times \left\{h\in H: \|h-b(\gamma \gamma')\| \leq nt\right\} \subseteq \phi_\lambda(\gamma x) \cdot B(0,K_{n,t}),
\]
where $B(0,K_{n,t})=\left\{(s,h) \in \RR_+ \times H: s^2 + \|h\|^2 < K_{n,t}^2\right\}$ with compact closure. Hence we conclude condition (3).

For condition (4), it follows from the above paragraph that $\gamma\gamma'$ is uniformly close to $\phi_\lambda(\gamma x)$. Hence there exists a constant $N>0$ such that $b(\gamma \gamma')$ belongs to $W_N(\phi_\lambda(\gamma x))$. It follows that
\[
\beta(\gamma \gamma')((g_n h_n)_t) = \beta_N(\phi_\lambda(\gamma x)) \circ \beta_{W_N(\phi_\lambda(\gamma x)), W_0(b(\gamma \gamma'))}((g_n h_n)_t),
\]
which concludes condition (4). Finally condition (5) follows from the choice of $g_n$ and $h_n$, hence we finish the proof.
\end{proof}

Applying $\beta_t$ pointwise, we obtain a map on the localisation level:
\[
\beta_{L,t}: \S ~\hat{\otimes}~ \A_L^r[ (X_\lambda); (B_\lambda) ]^\Gamma \longrightarrow  \A_L^r((X_\lambda); (B_\lambda \hat{\otimes} \Ac(V)))^\Gamma.
\]

The following lemma was originally from \cite[Lemma 7.6]{Yu00}, and an equivariant version was proved in \cite[Lemma 5.8]{FW16}. Hence we omit the proof here.

\begin{lem}
The maps $\beta_t, \beta_{L,t}$ can be extended to asymptotic morphisms
\[
\beta: \S ~\hat{\otimes}~ \A^r\big( (X_\lambda); (B_\lambda) \big)^\Gamma \longrightarrow  \A^r\big((X_\lambda); (B_\lambda \hat{\otimes} \Ac(V))\big)^\Gamma
\]
and
\[
\beta_L: \S ~\hat{\otimes}~ \A_L^r\big( (X_\lambda); (B_\lambda) \big)^\Gamma \longrightarrow  \A_L^r\big((X_\lambda); (B_\lambda \hat{\otimes} \Ac(V))\big)^\Gamma
\]
\end{lem}

Now we move on to the Dirac map. Let $H$ and $V$ be as above. For a finite-dimensional affine subspace $V_a$ in $V$, denote $\L^2(V_a)=L^2(V_a, \Cliff(V_a^0))$ the graded infinite-dimensional Hilbert space of square-integrable maps from $V_a$ to the complex Clifford algebra of $V_a^0$.

For finite-dimensional affine subspaces $V_a\subseteq V_b$ in $V$, recall that we have the orthogonal decomposition
$$V_{ba}^0=V_b^0\ominus V_a^0.$$
Define a unit vector $\xi_0\in \mathcal {L}^2(V_{ba}^0)$ by
$$\xi_0(w)=\pi^{-\frac{\text{dim}V_{ba}}{4}}\exp (-\frac{1}{2}\|w\|^2).$$
Then we regard $\mathcal {L}^2(V_a)$ as a subspace of $\mathcal {L}^2(V_b)$ via the isometric inclusion:
\begin{equation}
\label{sec:5:eq:identity}
\mathcal {L}^2(V_{a})\rightarrow \mathcal {L}^2(V_{ba}^0)~\hat{\otimes}~\mathcal {L}^2(V_{a})\cong \mathcal {L}^2(V_{b}),\quad\xi\mapsto \xi_0~\hat{\otimes}~\xi.
\end{equation}
It is easy to check that the collection
\[
\left\{\mathcal {L}^2(V_a)~\big|~V_a\subseteq V\ \text{is a finite dimensional affine subspace}\right\}
\]
forms a directed system. We define
\[
\mathcal {L}^2(V)=\lim\limits_{\longrightarrow}\mathcal {L}^2(V_a),
\]
where the limit is taken over the above directed system.

Let $\SS(V_a)\subseteq \mathcal {L}^2(V)$ be the subspace of Schwartz functions from $V_a$ to $\Cliff(V_a^0)$. Choosing an orthonormal basis $\{e_1,e_2,\cdots, e_n\}$ for $V_a^0$, let $\{x_1, x_2, \cdots, x_n\}$ be its dual coordinates. The \emph{Dirac operator} $D_{V_a}$ is an unbounded operator on $\mathcal {L}^2(V_a)$ with domain $\mathscr{S}(V_a)$ defined by:
\[
D_{V_a}\xi=\sum_{i=1}^n (-1)^{\text{deg} \xi}\frac{\partial \xi}{\partial x_i}e_i,
\]
where $e_i$ acts by Clifford multiplication. Given a vector $v\in V_a$, the \emph{Clifford operator} $C_{V_a,v}$ is an unbounded operator on $\mathcal {L}^2(V_a)$ with domain $\mathscr{S}(V_a)$ defined by:
\[
(C_{V_a,v}\xi)(w)=(w-v)\cdot \xi(w).
\]

For each $\gamma \in \Gamma$, we define a Schwartz subspace of $\mathcal {L}^2(V)$ by
$$\mathscr{S}(\gamma)=\lim_{\longrightarrow}\mathscr{S}(W_k(\gamma)),$$
where $W_k(\gamma)$ is defined in (\ref{EQ:W_k(gamma)}). Set $V_0(\gamma)=W_1(\gamma)$ and $V_k(\gamma)=W_{k+1}(\gamma)\ominus W_k(\gamma)$ if $k\geq 1$. Then we have an algebraic decomposition:
\[
V=V_0(\gamma)\oplus V_1(\gamma)\oplus\cdots\oplus V_n(\gamma)\oplus\cdots.
\]
For each $n\in \mathbb{N}$ and $t\geq1$, we define an unbounded operator $B_{n,t}(\gamma)$ on $\mathcal {L}^2(V)$ (associated to the above decomposition) given by
\[
B_{n,t}(\gamma)=\sum_{k=0}^{n-1}(1+kt^{-1})D_k+\sum_{k=n}^{\infty}(1+kt^{-1})(D_k+C_k)
\]
where $D_k=D_{V_k(\gamma)}$, $C_0=C_{V_0(\gamma),b(\gamma)}$ and $C_k=C_{V_k(\gamma),0}$.
The operator $B_{n,t}(\gamma)$ is well-defined on the Schwartz space $\mathscr {S}(\gamma)$, which is taken to be its domain.

The $\Gamma$-action on $H$ induce an action on $\mathcal {L}^2(V)$ by unitaries. It is easy to compute that
$$B_{n,t}(\gamma \gamma')=\gamma\cdot B_{n,t}(\gamma')\cdot \gamma^{-1},$$
where $\gamma\in \Gamma$ maps the domain $\mathscr{S}(\gamma')$ of $B_{n,t}(\gamma')$ to the domain $\mathscr{S}(\gamma \gamma')$ of $B_{n,t}(\gamma \gamma')$. This implies the following:

\begin{lem}[{\cite[Lemma 5.1]{FW16}}]
For any $g\in C_0(\RR)$ and $\gamma \in \Gamma$, we have $g(B_{n,t}(\gamma \gamma'))=\gamma g(B_{n,t}(\gamma')) \gamma^{-1}$ for each $\gamma'\in \Gamma$.
\end{lem}

\begin{defn}\label{defn:twisted Roe module language compact plus}
We define $\A^r[ (X_\lambda); (B_\lambda \hat{\otimes} \K(\L^2(V))) ]^\Gamma$ to be the $\ast$-subalgebra in $\prod_{\lambda \in \Lambda} \L((\H_{r,\lambda} \otimes B_\lambda)~\hat{\otimes}~ \K(\L^2(V)))^\Gamma$ consisting of elements $T=(T_\lambda)_\lambda$ satisfying the following conditions:
\begin{enumerate}
  \item $\sup_{\lambda \in \Lambda}\ppg_{P}(T_{\lambda})<\infty$;\\[0.05cm]
  \item for $\lambda \in \Lambda$ and bounded Borel subset $K \subseteq P_r(X_{\lambda})$, both $\chi_K T_\lambda$ and $T_\lambda \chi_K$ belong to $(\K(\H_{r,\lambda} ) \otimes B_\lambda) ~\hat{\otimes}~ \K(\L^2(V))$.
\end{enumerate}
Denote $\A^r( (X_\lambda); (B_\lambda \hat{\otimes} \K(\L^2(V)))) ^\Gamma$ the norm closure of $\A^r[ (X_\lambda); (B_\lambda \hat{\otimes} \K(\L^2(V))) ]^\Gamma$.

Define $\A_L^r ( (X_\lambda); (B_\lambda \hat{\otimes} \K(\L^2(V))) )^\Gamma$ to be closure of the collection of uniformly continuous bounded functions $(T_t)$ from $[1,\infty)$ to $\A^r( (X_\lambda); (B_\lambda \hat{\otimes} \K(\L^2(V))) )^\Gamma$ such that the $P$-propagation of $(T_t)$ tends to zero as $t\to +\infty$.
\end{defn}

For every non-negative integer $n$ and $\gamma\in \Gamma$, we define
\[
\theta_t^n(\gamma):\Ac(W_n(\gamma)) \longrightarrow \K(\mathcal {L}^2(V))
\]
by the formula
\[
\theta_t^n(\gamma)(g~\hat{\otimes}~h)=g_t(B_{n,t}(\gamma))M_{h_t}
\]
for $g\in \S$ and $h\in C_0(W_n(\gamma), \Cliff(W^0_n(\gamma))$, where $h_t(v)=h(b(\gamma)+t^{-1}(v-b(\gamma)))$ for $v\in W_n(\gamma)$ and $M_{h_t}$ is the pointwise multiplication operator on $\mathcal {L}^2(V)$ (in the sense that if $W \supseteq W_n(\gamma)$ is a finite-dimensional affine subspace of $V$, then we regard:
\[
(M_{h_t}\xi)(v+w)=h_t(v)\xi(v+w)
\]
for all $\xi\in \L^2(W)$, $v\in W_n(\gamma)$ and $w\in W\ominus W_n(\gamma)$).
It follows from \cite[Lemma 5.8]{HKT98} that the image of $\theta_t^n(\gamma)$ indeed sits inside $\K(\mathcal {L}^2(V))$. Moreover, it follows from the proof of \cite[Lemma 5.3]{FW16} that
\begin{equation}\label{EQ:theta equivariant}
\theta_t^n(\gamma \gamma') = \gamma \cdot \theta_t^n(\gamma') \cdot \gamma^{-1}
\end{equation}
for any $\gamma, \gamma' \in \Gamma$.

Given $T=(T_{\lambda,x,y})\in \A^r[(X_\lambda); (B_\lambda \hat{\otimes} \Ac(V))]^\Gamma$, let $N' \in \NN$ be such that for every $x,y\in X_\lambda$ there exists $T'_{\lambda,x,y}\in \K(\HH_{y,r,\lambda},\HH_{x,r,\lambda})~\hat{\otimes}~ \Ac(W_{N'}(\phi_\lambda(x)))$ satisfying
\[
T_{\lambda,x,y} = (\Id ~\hat{\otimes}~\beta_{N'}(\phi_\lambda(x)))(T'_{\lambda,x,y}).
\]
Taking $N \in \NN$ such that for any $x\in \D_\lambda$, $\gamma' \in \Gamma_x$ and $\gamma \in \Gamma$ we have $W_{N'}(\phi_\lambda(\gamma x)) \subseteq W_N(\gamma \gamma')$. Hence for each $x\in \D_\lambda$, $\gamma' \in \Gamma_x$ and $\gamma \in \Gamma$, we have:
\begin{align*}
T_{\lambda,\gamma x,y} &= (\Id ~\hat{\otimes}~\beta_{N'}(\phi_\lambda(\gamma x)))(T'_{\lambda,\gamma x,y}) = (\Id ~\hat{\otimes}~\beta_{N}(\gamma \gamma')) \circ (\Id ~\hat{\otimes}~\beta_{W_N(\gamma \gamma'), W_{N'}(\phi_\lambda(\gamma x))})(T'_{\lambda,\gamma x,y})\\
&= (\Id ~\hat{\otimes}~\beta_{N}(\gamma \gamma')) (T_{\lambda,\gamma x,y}^{\gamma'}),
\end{align*}
where $T_{\lambda,\gamma x,y}^{\gamma'} = (\Id ~\hat{\otimes}~ \beta_{W_N(\gamma \gamma'), W_{N'}(\phi_\lambda(\gamma x))})(T'_{\lambda,\gamma x,y})$.

\begin{defn}\label{defn:Dirac}
For $t \geq 1$, we define a map
\[
\alpha_t: \A^r[(X_\lambda); (B_\lambda ~\hat{\otimes}~ \Ac(V))]^\Gamma \longrightarrow  \A^r\big( (X_\lambda); (B_\lambda ~\hat{\otimes}~ \K(\L^2(V))) \big)^\Gamma
\]
by setting
\[
\alpha_t(T_{\lambda,x,y})_{\lambda, \gamma x,y}=\frac{1}{|\Gamma_x|}\sum_{\gamma' \in \Gamma_x} (\Id ~\hat{\otimes}~ \theta_t^N(\gamma \gamma'))(T_{\lambda,\gamma x,y}^{\gamma'})
\]
for $x\in \D_\lambda, y\in X_\lambda$ and $\gamma \in \Gamma$, where $T_{\lambda,\gamma x,y} = (\Id ~\hat{\otimes}~\beta_{N}(\gamma \gamma')) (T_{\lambda,\gamma x,y}^{\gamma'})$.
\end{defn}

Similar to (but much easier than) the proof of Lemma \ref{lem:Bott well-defined}, we obtain:

\begin{lem}\label{lem:Dirac well-defined}
For each $t \geq 1$, the map $\alpha_t$ is well-defined.
\end{lem}

Applying $\alpha_t$ pointwise, we obtain a map on the localisation level:
\[
\alpha_{L,t}: \A_L^r[(X_\lambda); (B_\lambda ~\hat{\otimes}~\Ac(V))]^\Gamma \longrightarrow  \A_L^r\big( (X_\lambda); (B_\lambda ~\hat{\otimes}~ \K(\L^2(V))) \big)^\Gamma.
\]

The following lemma was originally from \cite[Lemma 7.2]{Yu00}, and an equivariant version was proved in \cite[Lemma 5.5]{FW16}. Hence we omit the proof here.

\begin{lem}
The maps $\alpha_t, \alpha_{L,t}$ can be extended to asymptotic morphisms
\[
\alpha: \A^r\big((X_\lambda); (B_\lambda ~\hat{\otimes}~ \Ac(V))\big)^\Gamma \longrightarrow  \A^r\big( (X_\lambda); (B_\lambda ~\hat{\otimes}~ \K(\L^2(V))) \big)^\Gamma
\]
and
\[
\alpha_L: \A_L^r\big((X_\lambda); (B_\lambda ~\hat{\otimes}~ \Ac(V))\big)^\Gamma \longrightarrow  \A_L^r\big( (X_\lambda); (B_\lambda ~\hat{\otimes}~ \K(\L^2(V))) \big)^\Gamma.
\]
\end{lem}

Following the proof of \cite[Proposition 7.7]{Yu00} (see also \cite[Proposition 5.9]{FW16}), we have:

\begin{prop}\label{prop:identity-bott-dirac}
The compositions
\[
\alpha_*\circ\beta_*: K_*\left(\A^r\big( (X_\lambda); (B_\lambda) \big)^\Gamma\right) \longrightarrow  K_*\left(\A^r\big( (X_\lambda); (B_\lambda ~\hat{\otimes}~ \K(\L^2(V))) \big)^\Gamma\right)
\]
and
\[
(\alpha_L)_*\circ(\beta_L)_*: K_*\left(\A_L^r\big( (X_\lambda); (B_\lambda) \big)^\Gamma\right) \longrightarrow  K_*\left(\A_L^r\big( (X_\lambda); (B_\lambda ~\hat{\otimes}~ \K(\L^2(V))) \big)^\Gamma\right)
\]
equals the identity homomorphisms, respectively.
\end{prop}

Finally, we are in the position to prove Proposition \ref{prop: local isom for appendix}.

\begin{proof}[Proof of Proposition \ref{prop: local isom for appendix}]
Consider the following commutative diagram:
\[
\xymatrix{
			\lim\limits_{r\to \infty} K_*\left(\A_L^r\big( (X_\lambda); (B_\lambda) \big)^\Gamma\right) \ar[r]^{\ev_\ast}  \ar[d]_{(\beta_L)_\ast} & \lim\limits_{r\to \infty} K_*\left(\A^r\big( (X_\lambda); (B_\lambda) \big)^\Gamma\right)  \ar[d]^{\beta_\ast} \\
			\lim\limits_{r\to \infty} K_*\left(\A_L^r\big((X_\lambda); (B_\lambda ~\hat{\otimes}~ \Ac(V))\big)^\Gamma\right) \ar[r]^{\ev_\ast} \ar[d]_{(\alpha_L)_\ast}  & \lim\limits_{r\to \infty} K_*\left(\A^r\big((X_\lambda); (B_\lambda ~\hat{\otimes}~ \Ac(V))\big)^\Gamma\right) \ar[d]^{\alpha_\ast}  \\
			\lim\limits_{r\to \infty} K_*\left(\A_L^r\big( (X_\lambda); (B_\lambda ~\hat{\otimes}~ \K(\L^2(V))) \big)^\Gamma\right) \ar[r]^{\ev_\ast}  & \lim\limits_{r\to \infty} K_*\left(\A^r\big( (X_\lambda); (B_\lambda ~\hat{\otimes}~ \K(\L^2(V))) \big)^\Gamma\right). \\
	}
\]
The middle horizontal line is an isomorphism due to Proposition \ref{prop:local isomorphism with A(V)}, while the compositions of both of the vertical lines are isomorphisms due to Proposition \ref{prop:identity-bott-dirac}. Consequently, the result follows from a diagram chasing.
\end{proof}

\bibliographystyle{plain}
\bibliography{bib_equivCE}

\begin{thebibliography}{10}

\bibitem{AT19}
Goulnara Arzhantseva and Romain Tessera.
\newblock Admitting a coarse embedding is not preserved under group extensions.
\newblock {\em International Mathematics Research Notices},
  2019(20):6480--6498, 2019.

\bibitem{AS68}
Michael~F. Atiyah and Isadore~M. Singer.
\newblock The index of elliptic operators: {I}.
\newblock {\em Annals of mathematics}, pages 484--530, 1968.

\bibitem{BC88}
Paul Baum and Alain Connes.
\newblock {$K$}-theory for discrete groups.
\newblock {\em Operator algebras and applications}, 1:1--20, 1988.

\bibitem{BC00}
Paul Baum and Alain Connes.
\newblock Geometric {$K$}-theory for {L}ie groups and foliations.
\newblock {\em Enseignement Mathematique}, 46(1/2):3--42, 2000 (firstly
  circulated in 1982).

\bibitem{BCH94}
Paul Baum, Alain Connes, and Nigel Higson.
\newblock Classifying space for proper actions and {$K$}-theory of group
  {$C^*$}-algebras.
\newblock {\em Contemporary Mathematics}, 167:241--241, 1994.

\bibitem{CWY13}
Xiaoman Chen, Qin Wang, and Guoliang Yu.
\newblock The maximal coarse {B}aum--{C}onnes conjecture for spaces which admit
  a fibred coarse embedding into {H}ilbert space.
\newblock {\em Advances in Mathematics}, 249:88--130, 2013.

\bibitem{CM90}
Alain Connes and Henri Moscovici.
\newblock Cyclic cohomology, the {N}ovikov conjecture and hyperbolic groups.
\newblock {\em Topology}, 29(3):345--388, 1990.

\bibitem{DFW21}
Jintao Deng, Benyin Fu, and Qin Wang.
\newblock The equivariant coarse {B}aum--{C}onnes conjecture for metric spaces
  with proper group actions.
\newblock {\em arXiv preprint arXiv:2109.11643}, 2021.

\bibitem{FW16}
Benyin Fu and Xianjin Wang.
\newblock The equivariant coarse {B}aum--{C}onnes conjecture for spaces which
  admit an equivariant coarse embedding into {H}ilbert space.
\newblock {\em Journal of Functional Analysis}, 271(4):799--832, 2016.

\bibitem{FWY20}
Benyin Fu, Xianjin Wang, and Guoliang Yu.
\newblock The equivariant coarse {N}ovikov conjecture and coarse embedding.
\newblock {\em Communications in Mathematical Physics}, 380(1):245--272, 2020.

\bibitem{GWY08}
Guihua Gong, Qin Wang, and Guoliang Yu.
\newblock Geometrization of the strong {N}ovikov conjecture for residually
  finite groups.
\newblock {\em J. Reine Angew. Math.}, 621:159--189, 2008.

\bibitem{Gro93}
Mikhael Gromov.
\newblock Asymptotic invariants of infinite groups.
\newblock In {\em Geometric group theory, {V}ol. 2 ({S}ussex, 1991)}, volume
  182 of {\em London Math. Soc. Lecture Note Ser.}, pages 1--295. Cambridge
  Univ. Press, Cambridge, 1993.

\bibitem{GL83}
Mikhael Gromov and H.~Blaine Lawson.
\newblock Positive scalar curvature and the {D}irac operator on complete
  {R}iemannian manifolds.
\newblock {\em Publications Math{\'e}matiques de l'IH{\'E}S}, 58:83--196, 1983.

\bibitem{GHT00}
Erik Guentner, Nigel Higson, and Jody Trout.
\newblock {\em Equivariant {$E$}-theory for {$C^*$}-algebras}, volume 703.
\newblock American Mathematical Soc., 2000.

\bibitem{Hig00}
Nigel Higson.
\newblock Bivariant {$K$}-theory and the {N}ovikov conjecture.
\newblock {\em Geometric \& Functional Analysis}, 10(3):563--581, 2000.

\bibitem{HK01}
Nigel Higson and Gennadi Kasparov.
\newblock {$E$}-theory and {$KK$}-theory for groups which act properly and
  isometrically on {H}ilbert space.
\newblock {\em Inventiones mathematicae}, 144(1):23--74, 2001.

\bibitem{HKT98}
Nigel Higson, Gennadi Kasparov, and Jody Trout.
\newblock A {B}ott periodicity theorem for infinite-dimensional {E}uclidean
  space.
\newblock {\em Adv. Math.}, 135(1):1--40, 1998.

\bibitem{HR95}
Nigel Higson and John Roe.
\newblock On the coarse {B}aum-{C}onnes conjecture.
\newblock In {\em Novikov conjectures, index theorems and rigidity, {V}ol. 2
  ({O}berwolfach, 1993)}, volume 227 of {\em London Math. Soc. Lecture Note
  Ser.}, pages 227--254. Cambridge Univ. Press, Cambridge, 1995.

\bibitem{HRY93}
Nigel Higson, John Roe, and Guoliang Yu.
\newblock A coarse {M}ayer-{V}ietoris principle.
\newblock {\em Math. Proc. Cambridge Philos. Soc.}, 114(1):85--97, 1993.

\bibitem{Kas88}
Gennadi Kasparov.
\newblock Equivariant {$KK$}-theory and the {N}ovikov conjecture.
\newblock {\em Invent. Math.}, 91(1):147--201, 1988.

\bibitem{KY06}
Gennadi Kasparov and Guoliang Yu.
\newblock The coarse geometric {N}ovikov conjecture and uniform convexity.
\newblock {\em Adv. Math.}, 206(1):1--56, 2006.

\bibitem{KY12}
Gennadi Kasparov and Guoliang Yu.
\newblock The {N}ovikov conjecture and geometry of {B}anach spaces.
\newblock {\em Geom. Topol.}, 16(3):1859--1880, 2012.

\bibitem{Laf12}
Vincent Lafforgue.
\newblock La conjecture de {B}aum-{C}onnes \`a coefficients pour les groupes
  hyperboliques.
\newblock {\em J. Noncommut. Geom.}, 6(1):1--197, 2012.

\bibitem{LM89}
H.~Blaine Lawson and Marie-Louise Michelsohn.
\newblock {\em Spin geometry}, volume~38 of {\em Princeton Mathematical
  Series}.
\newblock Princeton University Press, Princeton, NJ, 1989.

\bibitem{MU32}
S.~Mazur and S.~Ulam.
\newblock Sur les transformations isom{\'e}triques d’espaces vectoriels
  norm{\'e}s.
\newblock {\em CR Acad. Sci. Paris}, 194(946-948):116, 1932.

\bibitem{QR10}
Yu~Qiao and John Roe.
\newblock On the localization algebra of {G}uoliang {Y}u.
\newblock {\em Forum Math.}, 22(4):657--665, 2010.

\bibitem{Roe88}
John Roe.
\newblock An index theorem on open manifolds. {I}, {II}.
\newblock {\em J. Differential Geom.}, 27(1):87--113, 115--136, 1988.

\bibitem{roe1993coarse}
John Roe.
\newblock {\em Coarse cohomology and index theory on complete {R}iemannian
  manifolds}, volume 497.
\newblock American Mathematical Soc., 1993.

\bibitem{Roe96}
John Roe.
\newblock {\em Index theory, coarse geometry, and topology of manifolds},
  volume~90.
\newblock American Mathematical Soc., 1996.

\bibitem{Ros83}
Jonathan Rosenberg.
\newblock {$C^*$}-algebras, positive scalar curvature, and the {N}ovikov
  conjecture.
\newblock {\em Publications Math{\'e}matiques de l'IH{\'E}S}, 58:197--212,
  1983.

\bibitem{shan2008equivariant}
Lin Shan.
\newblock An equivariant higher index theory and nonpositively curved
  manifolds.
\newblock {\em Journal of Functional Analysis}, 255(6):1480--1496, 2008.

\bibitem{Val02}
Alain Valette.
\newblock {\em Introduction to the {B}aum--{C}onnes conjecture}.
\newblock Springer Science \& Business Media, 2002.

\bibitem{WY12}
Rufus Willett and Guoliang Yu.
\newblock Higher index theory for certain expanders and {G}romov monster
  groups, {I}.
\newblock {\em Adv. Math.}, 229(3):1380--1416, 2012.

\bibitem{willett2020higher}
Rufus Willett and Guoliang Yu.
\newblock {\em Higher index theory}, volume 189.
\newblock Cambridge University Press, 2020.

\bibitem{Yu97}
Guoliang Yu.
\newblock Localization algebras and the coarse {B}aum-{C}onnes conjecture.
\newblock {\em $K$-Theory}, 11(4):307--318, 1997.

\bibitem{Yu00}
Guoliang Yu.
\newblock The coarse {B}aum-{C}onnes conjecture for spaces which admit a
  uniform embedding into {H}ilbert space.
\newblock {\em Invent. Math.}, 139(1):201--240, 2000.

\end{thebibliography}

\end{document}